\documentclass[11pt,oneside]{article}

\usepackage{amsthm, amsmath, amssymb, amsfonts,epsfig}
\usepackage{mathtools}
\usepackage{esint}
\usepackage{epstopdf}

\usepackage{float,graphicx}
\usepackage[font=sl,labelfont=bf]{caption}
\usepackage{subcaption}

\usepackage{enumerate}

\usepackage[colorlinks=true, pdfstartview=FitV, linkcolor=blue,
            citecolor=blue, urlcolor=blue]{hyperref}
\usepackage[usenames]{color}
\definecolor{Red}{rgb}{0.7,0,0.1}
\definecolor{Green}{rgb}{0,0.7,0}

\usepackage[numbers]{natbib}

\usepackage{url}
\makeatletter\def\url@leostyle{%
 \@ifundefined{selectfont}{\def\UrlFont{\sf}}{\def\UrlFont{\scriptsize\ttfamily}}} \makeatother

\usepackage{accents, comment}

\usepackage{bbm}

\usepackage{mathrsfs}

\newtheorem{theorem}{Theorem}
\newtheorem{proposition}[theorem]{Proposition}
\newtheorem{lemma}[theorem]{Lemma}
\newtheorem{corollary}[theorem]{Corollary}

\theoremstyle{definition}
\newtheorem{definition}[theorem]{Definition}

\theoremstyle{remark}
\newtheorem{remark}[theorem]{Remark}

\numberwithin{equation}{section}
\numberwithin{theorem}{section}

\def\cA{\mathcal{A}}
\def\cB{\mathcal{B}}
\def\cC{\mathcal{C}}
\def\cD{\mathcal{D}}

\def\cI{\mathcal{I}}
\def\cJ{\mathcal{J}}
\def\cK{\mathcal{K}}

\def\cM{\mathcal{M}}

\def\cR{\mathcal{R}}

\def\cT{\mathcal{T}}

\def\cZ{\mathcal{Z}}

\def\bC{\mathbb{C}}

\def\bN{\mathbb{N}}

\def\bR{\mathbb{R}}
\def\bS{\mathbb{S}}

\def\bZ{\mathbb{Z}}

\title{Time-Global Regularity of the Navier-Stokes System \\ with Hyper-Dissipation -- Turbulent Scenario\footnote{to appear in \emph{Ann. PDE}}}

\author{
Zoran Gruji\'c \\
\tiny  University of Alabama at Birmingham \\[-0.6ex]
\and 
 Liaosha Xu\\
\tiny  University of California at Riverside \\[-0.6ex]
}

\begin{document}

\maketitle

\begin{abstract}
The question of whether the hyper-dissipative (HD) Navier-Stokes (NS) system can exhibit spontaneous formation
of singularities in the super-critical regime--the hyperviscous effects being represented by a fractional power
of the Laplacian, say $\beta$, confined to interval $\bigl(1, \frac{5}{4}\bigr)$--has been a major open problem in
the mathematical fluid dynamics since the foundational work of J.L. Lions in 1960s. In this work,
an evidence of criticality of the Laplacian is presented, more precisely, a class of plausible blow-up scenarios
is ruled out as soon as $\beta$ is greater than one.
While the framework is based on the  `scale of sparseness' of
the super-level sets of the positive and negative parts of the components of the higher-order derivatives of the velocity 
previously introduced by the authors, a major novelty in the current work is classification of the HD flows near a potential spatiotemporal 
singularity in two main categories, `homogeneous' (the case consistent with a near-steady behavior) and `non-homogenous' (the case consistent 
with the formation and 
decay of turbulence). The main theorem states that in the non-homogeneous case any $\beta$ greater than one prevents a singularity. 
In order to illustrate the impact of this result
in a methodology-free setting, a two-parameter family of dynamically rescaled blow-up profiles is considered, and it is shown that as soon as 
$\beta$ is greater than one, a new region in the parameter space is ruled out. More importantly, the region is a neighborhood (in the parameter space) 
of the self-similar profile, i.e.,
the approximately self-similar blow-up, a prime suspect in possible singularity formation, is ruled out for all HD NS models.
\end{abstract}

\section{Introduction}

\bigskip

Recall that 3D hyper-dissipative (HD) Navier-Stokes (NS) system reads
\begin{align}
&\partial_t u+(-\Delta)^\beta u+ u\cdot\nabla u+\nabla p=0, &&\textrm{ in }\bR^3\times (0,T) \label{eq:HypNSE1}
\\
&\textrm{div}\ u=0,                                &&\textrm{ in }\bR^3\times (0,T) \label{eq:HypNSE2}
\\
&u(\cdot,0)=u_0(\cdot),                            &&\textrm{ in }\bR^3\times \{t=0\} \label{eq:HypNSE3}
\end{align}
where the exponent $\beta > 1$ calibrates the hyperviscos effects, the vector field $u$ is the velocity of the fluid, and the scalar field
$p$ the pressure. For simplicity, the hyperviscosity coefficient is set to one, the external force to zero, and
the spatial domain taken to be the whole space
(in this case, $(-\Delta)^\beta$ is a Fourier multiplier with the symbol $|\xi|^{2 \beta}$).

\medskip

Under the intrinsic scaling of the system,
$\Bigl(
\lambda^{2\beta -1} u(\lambda x, \lambda^{2 \beta} t), \,
\lambda^{4\beta -2} p(\lambda x, \lambda^{2\beta} t)
\Bigr)$,
the only exponent leaving the energy invariant is $\beta = \frac{5}{4}$, signaling the criticality.
It has been known since 1960s, more precisely, since the work of \citet{Lions1959, Lions1969}, that the system (\ref{eq:HypNSE1}) is
indeed globally-in-time regular for any $\beta \ge \frac{5}{4}$. \citet{Tao2009} extended this to the case of a logarithmic
correction to the critical diffusion given by
\begin{align*}
\frac{(-\Delta)^\frac{5}{4}}{\log^\frac{1}{2} (I - \Delta)}\ ;
\end{align*}
this was later improved to
\begin{align*}
\frac{(-\Delta)^\frac{5}{4}}{\log{(I - \Delta)}}
\end{align*}
in \citet{Barbato2014}. A different approach to reducing hyperviscous effects was presented in
\citet{Yang2019}, considering a globally anisotropic hyper-dissipation in the Fourier space (setting
some of the Fourier components to zero) of order $\frac{5}{4}$.

\medskip

Let us note that--up to now--mathematical theories in the NS case ($\beta =1$) and the super-critical
HD case $\Bigl(\beta \in \bigl(1, \frac{5}{4}\bigr)\Bigr)$ have been indistinguishable,
in other words, the HD theory in the super-critical regime has been a rescaled
analogue of the Navier-Stokes theory (with the additional difficulties stemming from the non-locality of
the fractional diffusion).
An example from the realm of the flows initiated at regular data is the work by \citet{Tao2016} where the NS nonlinearity
is replaced by an `averaged nonlinearity' enjoying the same scaling properties and the same fundamental
cancellation relation as the NS nonlinearity, leading to a finite time blow-up for a class of suitably (with respect to the
averaged nonlinearity) constructed Schwarz initial data. This type of blow-up is--by nature of its
construction--super-critical,
and can be adopted (as remarked in \citet{Tao2016}) to the HD case as well.
Note that the methods developed in our work do not apply to the `averaged NS' as they are fundamentally
pointwise. In particular, there is no H\"ormander-Mikhlin that would faithfully estimate the averaged nonlinearity 
for $p=\infty$.
An example from the realm of the flows initiated at the finite-energy data is the work
\citet{Buckmaster2018}, building on the fundamental paper of \citet{Buckmaster2019}, demonstrating
non-uniqueness of the finite-energy and integrable-vorticity initialized flows with an additional restriction
on the size of the time-singular set, for any $\beta \in [1, \frac{5}{4})$.

\medskip

A recent asymptotic criticality result for the NS system by \citet{Grujic2019} presented a mathematical framework--based
on a suitably defined `scale of sparseness' of the super-level sets of the positive and negative parts of the
components of the higher-order derivatives--in which the `scaling gap' between a regularity criterion and the
corresponding \emph{a priori} bound vanishes as the order of the derivative goes to infinity.
Since the radius of spatial analyticity of the solutions--a key player in the theory--is intimately related to
the strength of diffusion, it seems plausible that the type of analysis presented would have a bearing on the regularity of 
the HD NS system as well.
In particular, an evidence of criticality of the Laplacian, i.e., the emergence of a mechanism preventing the possible blow-up 
as soon as $\beta > 1$ is expected.

\medskip

Before diving into the world of hyperviscosity, let us mention two instances of the critical behavior of the NS system \emph{per se} within 
the realm of sparseness of the regions of intense vorticity. Recall that in this approach, based on the spatial analyticity of
solutions and the harmonic measure maximum principle, a possible blow-up will be prevented as long as the scale of
sparseness of the super-level sets of the field of interest (in this case the vorticity)--cut at a fraction of the $L^\infty$-norm--stays below 
a fraction of the scale of the radius of spatial analyticity.

\medskip

The first instance transpires if we consider a fully developed turbulent flow, and suppose that the regions of the intense vorticity
are comprised of the vortex filaments of length $O(1)$ (there is some numerical evidence supporting the emergence and persistence
of $O(1)$-long vortex filaments in turbulent flows). Then, the \emph{a priori} bound on the volume of the filament, stemming from the
\emph{a priori} $L^1$-bound on the vorticity (\citet{Constantin1990}), implies that the transversal scale of the filament--which is at the same time comparable
to the scale of sparseness--matches the scale of the analyticity radius, i.e., one arrives at the criticality (\citet{Grujic2016}).

\medskip

The second instance concerns a computational simulation of a Kida flow (\citet{Kida1985, KM87}) performed in \citet{TG2021}. 
Boratav and Pelz (\citet{BP94}) considered the Kida flow as a laboratory for 
the computational study
of the possible singularity formation in solutions to the 3D NS and Euler flows. In particular, they discovered a
time-interval of extreme intermittency (preceding the peak of the vorticity maximum) in which the local 
quantities increase sharply. Nevertheless, the simulations consistently
showed an eventual disruption in the approximately self-similar, critical scaling, a formation of the peak, and a
subsequent dissipation of the flow, prompting them to conclude that ``However, the increase in peak vorticity
stops at a certain time, possibly due to viscous dissipation effects''.
In \cite{TG2021}, the attention was focused on the time interval leading to the peak of $\|\omega(t)\|_\infty$, and the aim was to
investigate a possibility of a power-law dependence between the actual geometric scale of sparseness $r(t)$ (derived from a
suitable computational geometry algorithm, as well as from the data set generated by the citizen science game `Turbulence', 
ScienceAtHome, Aarhus University)
and the diffusion scale $d(t)=\frac{\nu^\frac{1}{2}}{\|\omega(t)\|_\infty^\frac{1}{2}}$ of the form $r \sim d^\alpha$
(in the vorticity formulation,  $d$ is a lower bound on the radius of spatial analyticity; $\nu$ is the viscosity).
Indeed, data analysis of the time-interval of interest revealed a very strong evidence of a power-law scaling. Moreover,
the scaling exponent $\alpha$ crystalized at $1.098 \pm 0.009$, offering a mechanism
behind the eventual `slump' and dissipation observed by Boratav and Pelz and demonstrating (sub-)criticality in the 
framework of sparseness.

\medskip

In this work we classify the super-critical HD flows near a
possible spatiotemporal singularity in two categories,  `homogeneous and `non-homogeneous'. The former exhibit
a special structure of the chain of derivatives consistent with a near-steady flow behavior, a typical example would be the initial state of a radially or axially
(with small axial component) symmetric flow, e.g. a Taylor-Green vortex, while the latter feature the higher-order analogues 
of the Taylor microscale consistent with the formation, development, and eventual decay (recall that we are in the zero external force scenario) 
of the turbulent flows.

\medskip

The main result of this paper, with precise formulation given in \textit{Theorem~\ref{th:HypNSEReg}}, states that as soon as the exponent $\beta$
is greater than 1, the non-homogeneous flows remain regular, revealing criticality of the Laplacian in the turbulent scenario. Essentially, the condition
identifying the turbulent regime is given by
\begin{equation}\label{t}
\frac{ \|D^{(k)} u(t)\|_\infty^\frac{1}{k+1} }{  \|D^{(2k)} u(t)\|_\infty^\frac{1}{2k+1}  }  \leq  \, c_* \ {(T-t)}^{-\frac{\beta - 1}{2k+1}}
\end{equation}
where $T$ denotes a possible singular time. This is 
required within a suitable spatial neighborhood of the singularity, as the flow approaches $T$, and it suffices that the dominance
of the higher-order spatial fluctuations holds over a finite range of indexes and along a single direction (the precise condition is given in terms of the
directional derivatives, cf.,  \textit{Theorem~\ref{th:HypNSEReg}}).

\medskip

Note that the dimension of the quantity bounded by $c_*$ is
\[
 L^{\frac{2k+2\beta-1}{2k+1}-\frac{k+2\beta-1}{k+1}+\frac{2\beta(\beta-1)}{2k+1}}
\]
which -- for a fixed $\beta>1$ -- becomes dimensionless as $k \to \infty$.

\medskip

In order to illustrate the impact of this result on the regularity theory of the HD
NS system in a methodology-free setting, consider the following two-parameter ($\alpha_x, \alpha_t > 0$) family of 
rescaled blow-up profiles at $(0, 0)$ (cf.  \citet{AlBr2022} where it was pointed out that for $\beta = 1$, the analysis
in \citet{Grujic2019} does not rule out new exponents), 

\begin{equation}\label{ha}
 u(x, t)=\frac{1}{(-t)^{\alpha_t}} \, U(y, s) \ \ \ \mbox{where} \ \ \ y=\frac{x}{(-t)^{\alpha_x}}, \ \ \ 
 s=-\log(-t),
\end{equation}
and $U$ is a smooth (in $y$)  base profile decaying outside $B(0, 1)$ such that
\begin{align}
  \|D^{(k)} U(s)\|_{L^\infty} &\le C_k\\
  \|D^{(k)} U(s)\|_{L^\infty(B(0, 1))} &\ge c_k
\end{align}
uniformly in large $s$, where 
\begin{equation}\label{r}
\frac{(C_k)^\frac{1}{k+1}}{(c_{2k})^\frac{1}{2k+1}} \le c.
\end{equation}
This assures that the scaling is driving the dynamics, making
the roles of the scaling exponents  $\alpha_x, \alpha_t$ more transparent. In particular, (\ref{r}) is consistent with a typical algebraic base profile
$U$, e.g., in the simplest case, $U(y)=\frac{1}{1+|y|}$.

\medskip

According to the state-of-the-art, the range of the scaling exponents $(\alpha_x, \alpha_t)$ 
allowing a blow-up corresponds to the shaded region in Figure 1
(light grey union dark grey; this is a $\beta$-version of the NS diagram in \citet{AlBr2022}). The bounding 
lines are as follows. The right line, $\alpha_t = (2\beta-1) \alpha_x$
corresponds to the scaling-invariant regularity class $L^\infty (0, T; L^\frac{3}{2\beta-1})$, the left line, 
$\alpha_t = \frac{3}{2} \alpha_x$
corresponds to the finite energy (the two lines meet at $\beta = \frac{5}{4}$),
the bottom line, $\alpha_t = \frac{1}{2\beta}$
corresponds to `Type I' blow-up, e.g., to the critical blow-up rate of the $L^\infty$-norm, $(-t)^{-\frac{1}{2\beta}}$,
while the $L^2$ space-time integrability of $D^\beta$ adds the top line, 
making the region bounded.

\medskip

\begin{figure}
  \centering
     \includegraphics[width=\linewidth]{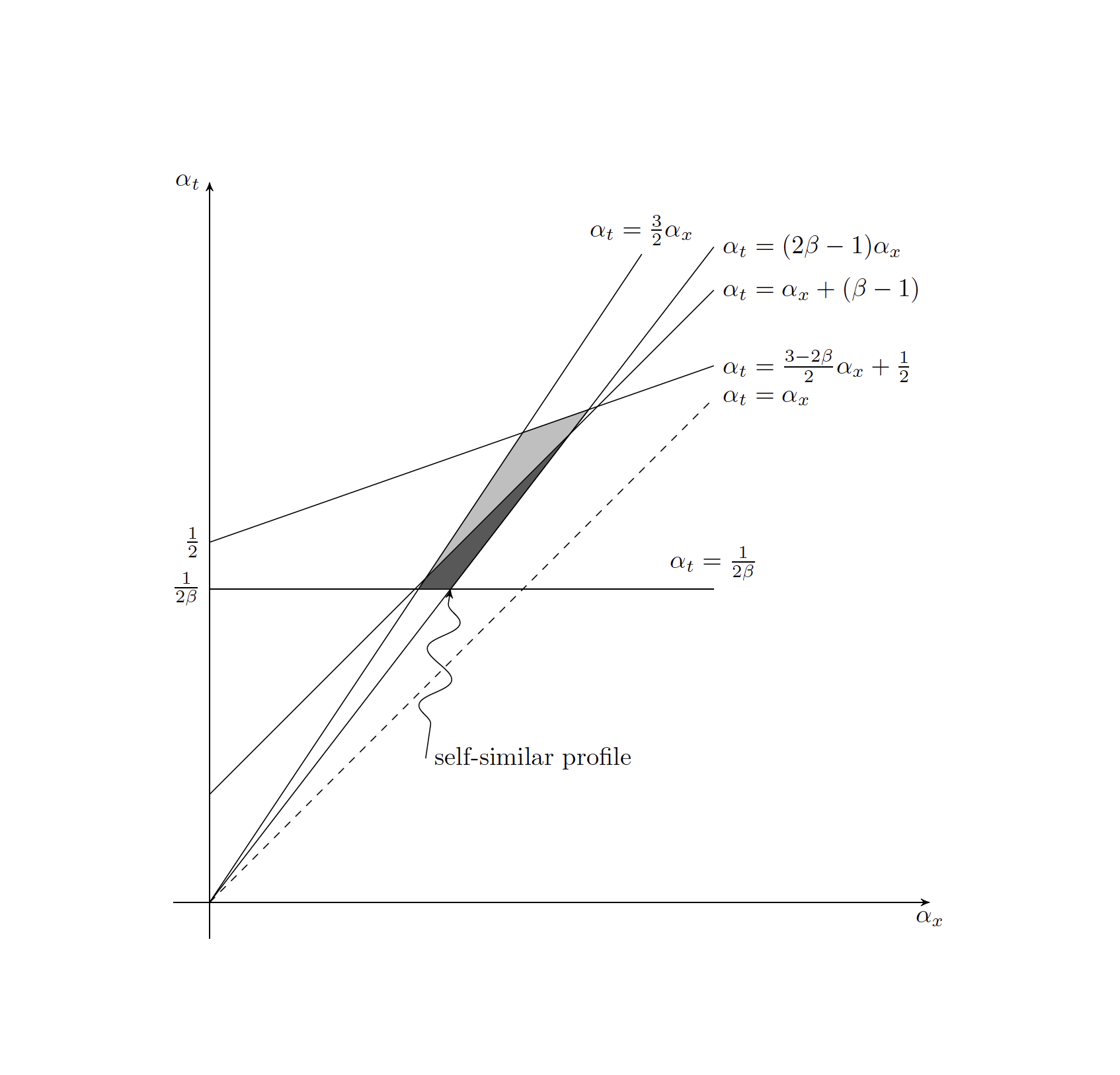}
     \vspace{.5in}
     \caption{The shaded area (light grey union dark grey) represents the region in the parameter space in which the singularity 
     formation--according to the classical, $L^p$-based theory--can not be ruled out.
     The mathematical framework based on the sparseness of 
     the super-level sets of the higher-order derivatives rules out the dark grey region 
     (notice that this is a neighborhood -- within the potentially singular region -- of the 
     self-similar profile, i.e., the approximately self-similar blow-up is ruled out as soon as $\beta$ is greater than one). The figure is a to-scale 
     rendition in the case $\beta=1.15$.}
     \label{yay}
\end{figure}

\medskip

The following theorem follows from  \textit{Theorem~\ref{th:HypNSEReg}} by a straightforward
calculation.

\medskip

\begin{theorem}\label{th:yay}
\noindent (i) Let $1 < \beta < \frac{5}{4}$. Then there exists a neighborhood--within the potentially singular region of the parameter 
space--of the self-similar 
profile in which a blow-up is ruled out (the dark grey region in Figure 1, the bounding line is
given by $\alpha_t = \alpha_x + (\beta-1)$). In other words,
the approximately self-similar blow-up is ruled out.

\smallskip

\noindent (ii) In particular, if $\frac{1+\sqrt{2}}{2} < \beta < \frac{5}{4}$, then the dark grey region covers everything, and a blow-up is ruled out
for all values of the scaling exponents.
\end{theorem}

\medskip

\begin{remark}
Some flows in the light grey region (in the $1 < \beta \le \frac{1+\sqrt{2}}{2}$ regime) might still be in the non-homogeneous
scenario (e.g., provided that the profiles $U$ feature local anisotropy around the origin), the rest is
in the homogeneous scenario which is being addressed in the upcoming work
\citet{Grujic2022}.
\end{remark}

\medskip

\begin{remark}
Note that Type I blow-up in this setting is ruled out as soon as $\displaystyle{\beta > \frac{3+\sqrt{15}}{6}}$.
\end{remark}

\medskip

At the end, we present a bit of heuristics behind the proof of  \textit{Theorem~\ref{th:HypNSEReg}}, identifying a principal source of the scaling gain.
Recall that in this framework, a `level-$k$ dynamics' is realized by evolution of the derivatives of order $k$, taking place in the
physical space. This is in contrast to the Fourier approach where one follows evolution of the modes of order $k$. Working in the physical space
here is necessary since a key component of our analysis is based on the spatial intermittency at different levels.

\medskip

For $\beta > 1$ and a nonnegative integer $k$, consider the following three scales in 3D (in reality, each scale
comes with a dimensional multiplicative constant that makes them a length scale).

\medskip

 \[
 l_k = \|D^{(k)}u\|_\infty^{-\frac{1}{k+\frac{3}{2}}}
 \]
 
 \[
  \rho_k = \|D^{(k)}u\|_\infty^{-\frac{1}{2\beta-1}\frac{3}{2}\frac{1}{k+\frac{3}{2}}}
 \]
  
 \[
   \widetilde{\rho}_k = \|D^{(k)}u\|_\infty^{-\frac{1}{2\beta-1}\frac{1}{k+1}}.
 \]

\medskip

$l_k$ is the level-$k$ `scale of sparseness' and can also be thought of as 
a `typical length scale' at level-$k$. If one can show that the level-$k$ lower bound on the radius of spatial
analyticity (a natural level-$k$ dissipation scale) dominates $l_k$, the road to a non blow-up argument based on the harmonic 
measure maximum principle opens up.

\medskip

$\rho_k$ is a general lower bound on the level-$k$ radius of spatial analyticity (cf. Theorem 4.2, once we
set $d=3$, $p=2$, the external force to 0, and take the worst case scenario). $\rho_k$ will
dominate $l_k$ only in the sub-critical regime delineated by Lions' exponent $\frac{5}{4}$. 
The proof is based on a complexified fixed-point algorithm in $L^\infty$ and
standard Gagliardo-Nirenberg interpolation (assuming $u_0 \in L^\infty \cap L^2$).

\medskip

And then there is $\widetilde{\rho}_k$. The idea here is to show that the `ascending chain condition', i.e., the assumption that
the higher-order derivatives (suitably rescaled) dominate the low-order derivatives, is capable of upgrading
the rigorous lower bound on the radius of spatial analyticity from $\rho_k$ to $\widetilde{\rho}_k$.

\medskip

The most transparent way to arrive at $\widetilde{\rho}_k$, or equivalently, to the time-scale of the order of
\[
 \tau_k ={\widetilde{\rho}_k}^{\, 2 \beta},
\]
is to assume the ascending chain condition uniformly in time, for as 
long as the solution is smooth, and
over the full range of indices, i.e., suppose that there exists a constant $c_0 > 1$ such that
\[
  \|D^ju(t)\|^{\frac{1}{j+1}} \leq c_0 \frac{(j!)^{\frac{1}{j+1}}}{(k!)^{\frac{1}{k+1}}}\|D^{k}u(t)\|^{\frac{1}{k+1}} \ \ \  \mbox{for all}  \ \ \  1 \le  j\le k.
\]
For the time being, let us focus on the real-space iterations. Then, it is straightforward to check that writing Duhamel for $D^{(k)}u$ 
and performing the fixed-point algorithm in $L^\infty$ (using the
above monotonicity assumption in place of the standard Gagliardo-Nirenberg interpolation in the Leibniz expansion of the nonlinearity) 
yields the time of existence precisely of the order of 
$\tau_k$.  Note that the quantity bounded by $c_0$ has a dimension of
\[
 L^{\frac{k+2\beta-1}{k+1}-\frac{j+2\beta-1}{j+1}}
\]
and is dimensionless only in the Navier-Stokes case (however, for a fixed $\beta>1$, and $k \to 2k, \, j \to k$
case, it becomes dimensionless in the limit as $k \to \infty$). This leads to a dimensional multiplicative constant in front of $\tau_k$,
making it a true time-scale once we bring back the hyperviscosity coefficent $\nu_\beta$ (which was -- for simplicity -- set to
1 at the beginning).

\medskip

Of course, this type of monotonicity assumption is unrealistic (plus the argument above in its simple form works only
for the real-space solutions). What one needs to show is that the montonicity assumed only at the initial time of the fixed-point algorithm,
and for a suitable range of indices,
will propagate (locally) in time and into the complex-space. This is essentially Theorem 4.5 which is stated for the much more general
bounding ratios (${\cal{M}}_{j, k}$) than 
\[
  \frac{(j!)^{\frac{1}{j+1}}}{(k!)^{\frac{1}{k+1}}},
\]
and the proof is analogous to the proof in the Navier-Stokes case given in \citet{Grujic2019}, and -- essentially -- a dynamic version
of its rudimentary form delineated above (hence, we refer to it as `dynamic interpolation').

\medskip

Incidentally, this argument is written out in detail in a short paper \citet{FaGr2023} (the results of which follow from the general theory
developed here), in particular, the conclusion of Corollary 3.2 for $j=k$ yields precisely 
$\tau_k$ and $\widetilde{\rho}_k$.

\medskip

For any $\beta>1$, the curves $l_k$ and $\widetilde{\rho}_k$ cross, and that (eventually) yields Theorem 4.1 which -- in turn -- rules out 
a class of super-critical (generalized) self-similar blow-ups (Theorem 1.1). For $\beta = 1$, they only get asymptotically close, 
and that -- as pointed out in 
\citet{AlBr2022} -- does not rule out any (generalized) self-similar blow-up in the super-critical regime.

\medskip

Let us remark that in this work a word `turbulent' is used in the sense that a portion of the chain being ascending can be thought of as a multi-level 
analogue of the conditions 
on either Taylor or Kraichnan scale leading to a cascade (spatial fluctuations of a field dominating the field, wrinkles, turbulence). 
Moreover, as long as $\beta>1$ and for $k$ large enough (depending on $\beta$ as well as some dynamic quantities),
a `typical length scale' at level-$k$ ($l_k$) indeed falls into the level-$k$ dissipation range delineated by the upgraded lower bound on the
radius of spatial analyticity ($\widetilde{\rho}_k$), demonstrating turbulent dissipation.

\medskip

Although falling of $l_k$ into the level-$k$ dissipation range
delineated by $\widetilde{\rho}_k$ is a key to the proof of Theorem 4.1, there was still much work to do. Namely,
one had to confirm that under the umbrella of the condition (\ref{t})
the ordered dynamics of `chain of derivatives' (ascending \emph{vs.} 
descending portions of the chain) indeed contradicts a blow-up (descending portions \emph{per se} are not a 
problem; here, $l_k$ falls into a level-$k$ dissipation range even in the Navier-Stokes case as demonstrated
in \citet{Grujic2019}). This was accomplished by carefully 
tracking all the relevant time-scales in an algorithmic way.

\medskip

\begin{remark}\label{qc}

Since the proof of Theorem 4.1 is quite long and technical it might not be readily transparent how a dynamic bound on the ratios of lower and higher order derivatives in the assumption (\ref{t}) fits in the overarching parable of the scale of the analyticity radius eventually overtaking the scale of a priori sparseness and it might be helpful to provide a quick rationale. Note that (\ref{t}) is equivalent to the assumption there exists a constant $c_{**} > 1$ such that
\begin{equation}\label{ts}
 \frac{\frac {\|D^{(k)}u(t)\|_\infty^{\frac{1}{k+1}}}{(k!)^{\frac{1}{k+1}}}}{\frac{\|D^{(2k)}u(t)\|_\infty^{\frac{1}{2k+1}}}{((2k)!)^{\frac{1}{2k+1}}}} 
 \le c_{**} \, (T-t)^{-\frac{\beta-1}{2k+1}}
\end{equation}
for $k$ large enough and $t$ near the first possible singular time $T$ (since $\displaystyle{(j!)^\frac{1}{j+1} \approx j}$), while 
a portion of the chain of derivatives was termed `ascending' if there exists a constant $c_0 > 1$ such that
\begin{equation}\label{asc}
 \frac{\frac {\|D^{(j)}u(t)\|_\infty^{\frac{1}{j+1}}}{(j!)^{\frac{1}{j+1}}}}{\frac{\|D^{(k)}u(t)\|_\infty^{\frac{1}{k+1}}}{(k!)^{\frac{1}{k+1}}}} \le c_0  
 \ \ \  \mbox{for all}  \ \ \  l \le  j\le k.
\end{equation}

\medskip

The rationale behind (\ref{asc}) bridging the super-criticality was given earlier in Introduction, here we present a quick scaling argument -- in the
same spirit -- indicating that the bound in (\ref{ts}) is indeed capable of bridging the super-criticality as soon as $\beta > 1$.

\medskip

Writing a Duhamel for $D^{(2k)}u$, the nonlinear term can be written as 
\[
 \iint  \nabla G_\beta \, D^{(2k)}\bigl[ u \otimes u \bigr]\, dx \, dt
\]
where $G_\beta$ is the fractional heat kernel (for simplicity, this is all real). Focusing on the symmetric term in the Leibniz
expansion, 
\[
 \iint  \nabla G_\beta \,\bigl[ D^{(k)}u \otimes D^{(k)}u \bigr]\, dx \, dt,
\]
taking the $L^\infty$-norms and utilizing (\ref{ts}) produces a closed estimate on $\|D^{(2k)}u(t)\|_\infty$ resulting in local time of
existence $T$ of at least 
\[
 \|D^{(2k)}u(t_0)\|_\infty^{-\frac{2\beta}{\alpha(\beta, 2k) (2k+1)}}
\]
where $\displaystyle{\alpha(\beta, 2k)=(2\beta-1)-\frac{4\beta (\beta-1)}{2k+1}}$. A complexified version then yields a lower
bound on the radius of spatial analyticity at $T$ of
\[
 \|D^{(2k)}u(t_0\|_\infty^{-\frac{1}{\alpha(\beta, 2k) (2k+1)}}
\]
which will overtake the \emph{a priori} estimate on the scale of sparseness
as soon as $\alpha(\beta, 2k) (2k+1) > 2k + \frac{3}{2}$ and we are back in business 
$\bigl($asymptotically, this takes place at $k \approx \frac{1}{\beta-1}$$\bigr)$.

\medskip

In the rigorous argument presented in Section 4, dynamics of the chain is decomposed in the primary ascending and descending
portions (the ratios bounded by constants) and the dynamic bound given in (\ref{ts}) serves as a `flexible ceiling' for the ascending pieces
to bounce off.

\end{remark}

\medskip

\begin{remark}\label{2d-3d}

In the case of the 2D NS system which is critical (the scaling-invariant level and the energy level coincide) all three scales  -- the 
\emph{a priori} scale of sparseness, the scale of the general lower
bound on the analyticity radius, and the scale of the lower bound on the analyticity radius in the monotone scenario -- coincide
and are equal to
\[
  \|D^{(k)}u\|_\infty^{-\frac{1}{k+1}}
\]
(in particular, the monotone scenario does not bring any gain).

\medskip

In the case of the 3D NS system which is super-critical (there is a gap between the scaling-invariant level and the energy level), 
the aforementioned scales are
\[
 \|D^{(k)}u\|_\infty^{-\frac{1}{k+\frac{3}{2}}}, \  \  \|D^{(k)}u\|_\infty^{-\frac{1}{\frac{2}{3}k+1}}, \ \mbox{and} \ \|D^{(k)}u\|_\infty^{-\frac{1}{k+1}},
\]
respectively. In this setting, the super-criticality reveals itself as a constant (independent of $k$) scaling gap between the first two scales, 
while the scaling gap between the first and the last scale vanishes as $k \to \infty$ (this was termed `asymptotic criticality' in
 \citet{Grujic2019}). It is informative to note that the last scale is the only true length-scale, the first two come with dimensional
 multiplicative constants attached to them (depending on the $L^2$-norm of the initial velocity), and is in this sense a natural
 level-$k$ scale in turbulent regime.

\end{remark}

\medskip

The paper is organized as follows. Section~\ref{sec:PrfThm} provides a synopsis and refinement of the asymptotic criticality
result for the NS system presented in \citet{Grujic2019}, Section~\ref{sec:HomASS} contains the precise definitions 
of homogeneity and non-homogeneity, Section~\ref{sec:MainThm} states the main theorem for the HD NS in the non-homogeneous case 
(\textit{Theorem~\ref{th:HypNSEReg}}) followed by a proof based on the mechanism developed in \citet{Grujic2019}, 
while Section~\ref{sec:hh} addresses Theorem~\ref{th:yay}.

\section{A Review and Refinement of Asymptotic Criticality}\label{sec:PrfThm}

The NS regularity problem, i.e. the time-global existence of smooth solutions to the Navier-Stokes system in $\bR^d$ ($d\ge 3$)
\begin{align}
&\partial_t u-\Delta u+ u\cdot\nabla u+\nabla p=f, &&\textrm{ in }\bR^d\times (0,T) \label{eq:NSE1}
\\
&\textrm{div}\ u=0,                                &&\textrm{ in }\bR^d\times (0,T) \label{eq:NSE2}
\\
&u(\cdot,0)=u_0(\cdot),                            &&\textrm{ in }\bR^d\times \{t=0\} \label{eq:NSE3}
\end{align}
where the force $f(\cdot,t)$ is real-analytic in space with a uniform analyticity radius $\delta_f$ for all $t\in\bR^+$, which admits some analytic extension $f+ig$, while $u_0$ is the given initial velocity vector field, has been super-critical in the sense that there has been a `scaling gap' between any regularity criterion and the corresponding \emph{a priori} bound. More precisely, all the regularity criteria are at best scaling-invariant (with respect to the intrinsic scaling), while all the corresponding \emph{a priori} bounds had been on the scaling level of the bounded kinetic energy, $u \in L^\infty (0,T; L^2)$. A classical example (in three dimensions) is given by the Ladyzhenskaya-Prodi-Serrin regularity criterion, $u \in L^p (0,T; L^q)$,
\begin{align*}
\frac{3}{q} + \frac{2}{p} = 1
\end{align*}

\noindent \emph{vs.} the corresponding \emph{a priori} bound $u \in L^p (0,T; L^q)$,

\begin{align*}
\frac{3}{q} + \frac{2}{p} = \frac{3}{2}
\end{align*}

\noindent (for a suitable range of the parameters).

The \emph{a priori} bounds are traditionally derived for an arbitrary Leray-Hopf (weak) solution
to the 3D NS system. Since we are primarily interested in reducing the scaling gap in the
regularity problem, henceforth, the \emph{a priori} bounds will be discussed for a smooth
flow approaching a possible singular time.

The mathematical framework based on the suitably defined `scale of sparseness' of the regions of the intense fluid activity (\citet{Grujic2001, Grujic2013, Bradshaw2019, Grujic2019}) has been designed as a laboratory for a rigorous mathematical analysis of the phenomenon of spatial intermittency in turbulent flows, with a hope that it may lead to `scaling deviations' from the classical regularity theory.

In this section we first compile some notions and ideas about sparseness of the regions of intense fluid activity whose initial mathematical setup was developed in \citet{Grujic2001} and reformulated in \citet{Grujic2013, Farhat2017} and \citet{Bradshaw2019}, as well as provide a review of the key steps for the study of spatial intermittency of the higher order derivatives as presented in \citet{Grujic2019}. At the end, we develop some improvements -- giving the portions of the chain more freedom -- in preparation for the proof of \textit{Theorem~\ref{th:HypNSEReg}}.

Let $S$ be an open subset of $\mathbb{R}^3$ and $\mu_d$ the $d$-dimensional Lebesgue measure.

\begin{definition}
For a spatial point $x_0$ and $\delta\in (0,1)$, an open set $S$ is 1D $\delta$-sparse around $x_0$ at scale $r$ if there exists a unit vector $\nu$ such that
\begin{align*}
\frac{\mu_1\left(S\cap (x_0-r\nu, x_0+r\nu)\right)}{2r} \le \delta\ .
\end{align*}
\end{definition}

The volumetric version is the following.

\begin{definition}
For a spatial point $x_0$ and $\delta\in (0,1)$, an open set $S$ is 3D $\delta$-sparse around $x_0$ at scale $r$ if
\begin{align*}
\frac{\mu_3\left(S\cap B_r(x_0)\right)}{\mu_3(B_r(x_0))} \le \delta\ .
\end{align*}
Also $S$ is said to be $r$-semi-mixed with ratio $\delta$ if the above inequality holds for every $x_0\in\bR^3$. (It is straightforward to check that for any $S$, $3$-dimensional $\delta$-sparseness at scale $r$ implies 1D $(\delta)^\frac{1}{3}$-sparseness at scale $r$ around any spatial point $x_0$; however the converse is false, i.e. local-1D sparseness is in general a weaker condition.)
\end{definition}

Based on the scale of sparseness of the super-level sets of the positive and negative parts of the vectorial components of a function $f$, \citet{Bradshaw2019} introduced the regularity classes $Z_\alpha$ as a new device for scaling comparison of solutions to the 3D NSE. In what follows, let us denote the positive and the negative parts of the components of a vector field $f$ by $f_i^\pm$, and calculate the norm of a vector $v=(a, b, c)$ as $|v| = \max \{|a|, |b|, |c|\}$. Then we have the following definition.

\begin{definition}[\citet{Bradshaw2019}]
For a positive exponent $\alpha$, and a selection of parameters $\lambda$ in $(0,1)$, $\delta$ in $(0,1)$ and $c_0>1$, the class of functions $Z_\alpha(\lambda, \delta; c_0)$ consists of bounded, continuous functions $f : \mathbb{R}^3 \to \mathbb{R}^3$ subjected to the following uniformly-local condition. For $x_0$ in $\mathbb{R}^3$, select the/a component $f_i^\pm$ such that $f_i^\pm(x_0) = |f(x_0)|$, and require that the set
\begin{align*}
S_i^\pm:=\biggl\{ x \in \mathbb{R}^3: \, f_i^\pm(x) > \lambda \|f\|_\infty\biggr\}
\end{align*}
be 3D $\delta$-sparse around $x_0$ at scale $c \, \frac{1}{\|f\|_\infty^\alpha}$, for some $c, \frac{1}{c_0} \le c \le c_0$. Enforce this for all $x_0$ in $\mathbb{R}^3$. Here, $\alpha$ is the scaling parameter, $c_0$ is the size-parameter, and $\lambda$ and $\delta$ are the (interdependent) `tuning parameters'.
\end{definition}

Applying the $Z_\alpha$ framework to the vorticity field (\citet{Bradshaw2019}), the regularity class transpired to be $Z_\frac{1}{2}$, while the corresponding class of \emph{a priori} sparseness near a possible singular time transpired to be $Z_\frac{2}{5}$, bringing a scaling gain within the framework. The gain is due to the special structure of the vorticity form of the 3D NS system, namely, if one works with the full gradient, 
the class of  \emph{a priori} sparseness remains $Z_\frac{2}{5}$, while the regularity class worsens to $Z_\frac{3}{5}$ (this corresponds to the standard 
scaling gap).

\medskip

The next step (\citet{Grujic2019}) was to consider the higher-order spatial fluctuations of the velocity field (higher order derivatives)
and investigate the scaling gap--in the $Z_\alpha$ framework--as the order of the derivative goes to infinity. Essentially, the idea was to build a Sobolev scale $Z^{(k)}_{\alpha_k}$ based on $Z_\alpha$,
\begin{align*}
 u \in Z^{(k)}_{\alpha_k} \ \ \ \mbox{if} \ \ \ D^{(k)} u \in Z_{\alpha_k}.
\end{align*}

\medskip

The main results are depicted in the following table.

\medskip

\begin{table}[H]
{
\begin{center}
\begin{tabular}{ | p{5cm} | p{5cm} | }
\hline
 ~&~
\\
Regularity class
& A priori bound
\\  ~&~\\ \hline ~&
\\   $u(\tau) \in \bigcap_{k \ge k^*} Z^{(k)}_\frac{1}{k+1}$ on a suitable $(T^* - \epsilon, T^*)$, small size-parameters,
uniform in time; $k^*$ arbitrary large
&  $u(\tau) \in \bigcap_{k \ge 0} Z^{(k)}_\frac{1}{k+\frac{3}{2}}$  on a suitable $(T^* - \epsilon, T^*)$, the size-parameters
uniform in time
\\ ~  ~ & ~ 
\\\hline
\end{tabular}
\end{center}
}
\end{table}

\noindent It is instructive to take a closer look at the level-$k$ scales of sparseness realizing the above functional classes.

{
\begin{table}[H]
{
\begin{center}
\begin{tabular}{ | p{5cm} | p{5cm}  |}
\hline ~&~
\\
Regularity class-scale
& A priori bound-scale
\\  ~&~\\ \hline ~&
\\  $\frac{1}{C_1(k)} \frac{1}{\|D^{(k)} u\|_\infty^\frac{1}{k+1}}$
&  $C_2(\|u_0\|_2, k) \frac{1}{\|D^{(k)} u\|_\infty^\frac{1}{k+\frac{3}{2}}}$

\\ ~  ~ & ~ 
\\\hline
\end{tabular}
\end{center}
}
\end{table}
}

\noindent A key information of interest is the scaling of dynamic quantities in the table
above, given in the table below.

{
\begin{table}[H]
{
\begin{center}
\begin{tabular}{ | p{3.9cm} | p{3.7cm} | }
\hline
 ~&~
\\
Regularity class-scale
& A priori bound-scale

\\  ~&~\\ \hline ~&~
\\  $\frac{1}{\|D^{(k)} u\|_\infty^\frac{1}{k+1}} \approx r$
&  $\frac{1}{\|D^{(k)} u\|_\infty^\frac{1}{k+\frac{3}{2}}} \approx r^\frac{k+1}{k+\frac{3}{2}}$
\\ ~  ~ & ~ 
\\\hline
\end{tabular}
\end{center}
}
\end{table}
}

\noindent Since
\begin{align*}
r^\frac{k+1}{k+\frac{3}{2}}  \to  r, \ \ \  k \to \infty
\end{align*}
and $k^*$ can be taken arbitrary large, this was termed `asymptotic criticality' of the NS regularity problem within the framework (\citet{Grujic2019}).

\bigskip

The following is a list of some results, either derived or quoted in \citet{Grujic2013, Farhat2017, Bradshaw2019} and \citet{Grujic2019}, which will be referred to in the rest of the paper.

\begin{theorem}[\citet{Guberovic2010} and \citet{Bradshaw2019}]\label{th:LinftyIVP}
Let the initial datum $u_0\in L^\infty$ (resp. $\omega_0\in L^\infty\cap L^2$). Then, for any $M>1$, there exists a constant $c(M)$ such that there is a unique mild solution $u$ (resp. $\omega$) in $C_w([0,T], L^\infty)$ where $T\ge\frac{1}{c(M)^2\|u_0\|_\infty^2}$ (resp. $T\ge\frac{1}{c(M)\|\omega_0\|_\infty}$), which has an analytic extension $U(t)$ (resp. $W(t)$) to the region
\begin{align*}
\cD_t:=\left\{x+iy\in\bC^3\ :\ |y|\le \sqrt{t}/c(M)\ \left(\textrm{resp. }|y|\le \sqrt{t}/\sqrt{c(M)}\right)\right\}
\end{align*}
for all $t\in[0,T]$, and
\begin{align*}
\sup_{t\le T}\|U(t)\|_{L^\infty(\cD_t)}\le M\|u_0\|_\infty\ \left(\textrm{resp. }\sup_{t\le T}\|W(t)\|_{L^\infty(\cD_t)}\le M\|\omega_0\|_\infty\right).
\end{align*}
\end{theorem}

\begin{lemma}[\citet{Nirenberg1959} or \citet{Gagliardo1959}]\label{le:GNIneq}
Suppose $p,q,r\in[1,\infty]$, $s\in\bR$ and $m,j,d\in\bN$ satisfy
\begin{align*}
\frac{1}{p} = \frac{j}{d} + \left(\frac{1}{r}-\frac{m}{d}\right)s + \frac{1-s}{q} \ , \qquad \frac{j}{m}\le s\le 1\ .
\end{align*}
Then, there exists constant $C$ only depending on $m,d,j,q,r,s$ such that for any function $f: \bR^d\to\bR^d$
\begin{align*}
\|D^j f\|_{L^p}\le C\|D^m f\|_{L^r}^s\|f\|_{L^q}^{1-s}
\end{align*}
\end{lemma}

\begin{lemma}[Montel's]\label{le:Montel}
Let $p\in[1,\infty]$ and let $\mathscr{F}$ be a set of analytic functions $f$ in an open set $\Omega\subset \bC^d$ such that
\begin{align*}
\underset{f\in\mathscr{F}}{\sup}\ \|f\|_{L^p(\Omega)}<\infty\ .
\end{align*}
Then $\mathscr{F}$ is a normal family.
\end{lemma}

\begin{theorem}\label{th:MainThmVel}
Assume $u_0\in L^\infty(\bR^d)\cap L^p(\bR^d)$ and $f(\cdot,t)$ is divergence-free and real-analytic in the space variable with the analyticity radius at least $\delta_f$ for all $t\in[0,\infty)$, and the analytic extension $f+ig$ satisfies
\begin{align*}
\Gamma_\infty^k(t) &:= \sup_{s<t}\sup_{|y|<\delta_f} \left(\|D^kf(\cdot,y,s)\|_{L^\infty}+\|D^kg(\cdot,y,s)\|_{L^\infty}\right)<\infty\ ,
\\
\Gamma_p(t) &:= \sup_{s<t}\sup_{|y|<\delta_f} \left(\|f(\cdot,y,s)\|_{L^p}+\|g(\cdot,y,s)\|_{L^p}\right)<\infty\ .
\end{align*}
Fix $k\in\bN$, $M>1$ and $t_0>0$ and let
\begin{align}
T_*&=\min\left\{\left(C_1(M) 2^{2k} \left(\|u_0\|_p + \Gamma_p(t_0)\right)^{2k/(k+\frac{d}{p})} \left(\|D^ku_0\|_\infty + \Gamma_\infty^k(t_0)\right)^{\frac{2d}{p}/(k+\frac{d}{p})} \right)^{-1}, \right. \notag
\\
&\qquad \left. \left(C_2(M) \left(\|u_0\|_p + \Gamma_p(T)\right)^{(k-1)/(k+\frac{d}{p})} \left(\|D^ku_0\|_\infty + \Gamma_\infty^k(T)\right)^{(1+\frac{d}{p})/(k+\frac{d}{p})} \right)^{-1} \right\} \label{eq:TimeLength}
\end{align}
where $C(M)$ is a constant only depending on $M$. Then there exists a solution
\begin{align*}
u\in C([0,T_*),L^p(\bR^d)^d) \cap C((0,T_*),C^\infty(\bR^d)^d)
\end{align*}
of the NSE \eqref{eq:NSE1}-\eqref{eq:NSE3} such that for every $t\in (0,T_*)$, $u$ is a restriction of an analytic function $u(x,y,t)+iv(x,y,t)$ in the region
\begin{align}\label{eq:AnalDom}
\cD_t=: \left\{(x,y)\in\bC^d\ \big|\ |y|\le \min\{ct^{1/2},\delta_f\}\right\}\ .
\end{align}
Moreover, $D^ju\in C((0,T_*),L^\infty(\bR^d)^d)$ for all $0\le j\le k$ and
\begin{align}
&\underset{t\in(0,T)}{\sup}\ \underset{y\in\cD_t}{\sup} \|u(\cdot,y,t)\|_{L^p} + \underset{t\in(0,T)}{\sup}\ \underset{y\in\cD_t}{\sup} \|v(\cdot,y,t)\|_{L^p}\le M\left(\|u_0\|_p + \Gamma_p(T)\right)\ ,
\\
&\underset{t\in(0,T)}{\sup}\ \underset{y\in\cD_t}{\sup} \|D^ku(\cdot,y,t)\|_{L^\infty} + \underset{t\in(0,T)}{\sup}\ \underset{y\in\cD_t}{\sup} \|D^kv(\cdot,y,t)\|_{L^\infty}\le M\left(\|D^ku_0\|_\infty + \Gamma_\infty^k(T)\right)\ .
\end{align}
\end{theorem}

\bigskip

An analogous result for the vorticity is the following (here we set the external force to zero).

\begin{theorem}\label{th:MainThmVor}
Assume the initial value $\omega_0\in L^\infty(\bR^3)\cap L^p(\bR^3)$ where $1\le p<3$. Fix $k\in\bN$, $M>1$ and $t_0>0$ and let
\begin{align}
T_*&=C(M)\cdot \min\left\{2^{-k} \left( \|\omega_0\|_p^{k/(k+\frac{d}{p})} \cdot \|D^k\omega_0\|_\infty^{\frac{d}{p}/(k+\frac{d}{p})} + \|\omega_0\|_p \right)^{-1}, \right. \notag
\\
&\qquad \left. \left(\|D^k\omega_0\|_\infty^{\frac{d/p}{k+\frac{d}{p}}}\|\omega_0\|_p^{\frac{k}{k+\frac{d}{p}}} + \|D^k\omega_0\|_\infty^{\frac{1+\frac{d}{p}}{k+\frac{d}{p}}}\|\omega_0\|_p^{\frac{k-1}{k+\frac{d}{p}}} + \|D^k\omega_0\|_\infty^{\frac{1}{k+\frac{d}{p}}}\|\omega_0\|_p^{\frac{k-1+\frac{d}{p}}{k+\frac{d}{p}}} \right)^{-1} \right\}
\end{align}
where $C_i(M)$ is a constant only depending on $M$ and $d=3$. Then there exists a solution
\begin{align*}
\omega\in C([0,T_*),L^p(\bR^3)^3) \cap C((0,T_*),C^\infty(\bR^3)^3)
\end{align*}
of the NSE \eqref{eq:NSE1}-\eqref{eq:NSE3} such that for every $t\in (0,T_*)$, $\omega$ is a restriction of an analytic function $\omega(x,y,t)+i\zeta(x,y,t)$ in the region
\begin{align}\label{eq:AnalDomVor}
\cD_t=: \left\{(x,y)\in\bC^3\ \big|\ |y|\le ct^{1/2}\right\}\ .
\end{align}
Moreover, $D^k\omega\in C((0,T_*),L^\infty(\bR^3)^3)$ and
\begin{align}
&\underset{t\in(0,T)}{\sup}\ \underset{y\in\cD_t}{\sup} \|\omega(\cdot,y,t)\|_{L^p} + \underset{t\in(0,T)}{\sup}\ \underset{y\in\cD_t}{\sup} \|\zeta(\cdot,y,t)\|_{L^p}\le M\|\omega_0\|_p \ ,
\\
&\underset{t\in(0,T)}{\sup}\ \underset{y\in\cD_t}{\sup} \|D^k\omega(\cdot,y,t)\|_{L^\infty} + \underset{t\in(0,T)}{\sup}\ \underset{y\in\cD_t}{\sup} \|D^k\zeta(\cdot,y,t)\|_{L^\infty}\le M\|D^k\omega_0\|_\infty \ .
\end{align}
\end{theorem}

\bigskip

\begin{proposition}[\citet{Ransford1995}]\label{prop:HarMaxPrin}
Let $\Omega$ be an open, connected set in $\bC$ such that its boundary has nonzero Hausdorff dimension, and let $K$ be a Borel subset of the boundary. Suppose that $u$ is a subharmonic function on $\Omega$ satisfying
\begin{align*}
u(z) &\le M\ , \quad \textrm{for }z\in\Omega
\\
\limsup_{z\to\zeta} u(z) &\le m\ , \quad \textrm{for }\zeta\in K.
\end{align*}
Then
\begin{align*}
u(z)\le m \, h(z,\Omega,K) + M(1-h(z,\Omega,K)) \ , \quad \textrm{for }z\in\Omega.
\end{align*}
\end{proposition}

The following extremal property of the harmonic measure in the unit disk $\mathbb{D}$
will be helpful in the calculations to follow.

\medskip

\begin{proposition}[\citet{Solynin1997}]
Let $\lambda$ be in $(0, 1)$, $K$ a closed subset of $[-1,1]$
such that $\mu (K) = 2\lambda$,
and suppose that the origin is in $\mathbb{D} \setminus K$. Then
\[
 h(0,\mathbb{D},K) \ge h(0,\mathbb{D}, K_\lambda) =
 \frac{2}{\pi} \arcsin \frac{1-(1-\lambda)^2}{1+(1-\lambda)^2}
\]
where $K_\lambda = [-1, -1+\lambda] \cup [1-\lambda, 1]$.
\end{proposition}

\medskip

As demonstrated in \citet{Farhat2017} and \citet{Bradshaw2019}, the concept of `escape time' allows for a more
streamlined presentation.

\begin{definition}
Let $u$ (resp. $\omega$) be in $C([0,T^*], L^\infty)$ where $T^*$ is the first possible blow-up time. A time $t\in(0,T^*)$ is an escape time if $\|u(s)\|_\infty>\|u(t)\|_\infty$ (resp. $\|\omega(s)\|_\infty>\|\omega(t)\|_\infty$) for any $s\in(t,T^*)$. (Local-in-time continuity of the $L^\infty$-norm implies there are continuum-many escape times.)
\end{definition}
	
Here we recall the main theorem about the spatial intermittency based regularity criterion for the velocity and the vorticity presented in \citet{Farhat2017}  and  \citet{Bradshaw2019}, respectively.

\vspace{-0.05in}
\begin{theorem}[\citet{Farhat2017} and \citet{Bradshaw2019}]\label{th:SparsityRegVel}
Let $u$ (resp. $\omega$) be in $C([0,T^*), L^\infty)$ where $T^*$ is the first possible blow-up time, and assume, in addition, that $u_0\in L^\infty$ (resp. $\omega_0\in L^\infty\cap L^2$). Let $t$ be an escape time of $u(t)$ (resp. $\omega(t)$), and suppose that there exists a temporal point
\begin{align*}
&\qquad\quad s=s(t)\in \left[t+\frac{1}{4c(M)^2\|u(t)\|_\infty^2},\ t+\frac{1}{c(M)^2\|u(t)\|_\infty^2}\right]
\\
&\left(\textrm{resp. } s=s(t)\in \left[t+\frac{1}{4c(M)\|\omega(t)\|_\infty},\ t+\frac{1}{c(M)\|\omega(t)\|_\infty}\right]\ \right)
\end{align*}
such that for any spatial point $x_0$, there exists a scale $\rho\le \frac{1}{2c(M)^2\|u(s)\|_\infty}$ $\left(\textrm{resp. }\rho\le \frac{1}{2c(M)\|\omega(s)\|_\infty^{\frac{1}{2}}}\right)$ with the property that the super-level set
\begin{align*}
&\qquad\quad V_\lambda^{j,\pm}=\left\{x\in\bR^d\ |\ u_j^\pm(x,s)>\lambda \|u(s)\|_\infty\right\}
\\
&\left(\textrm{resp. }\Omega_\lambda^{j,\pm}=\left\{x\in\bR^3\ |\ \omega_j^\pm(x,s)>\lambda \|\omega(s)\|_\infty\right\}\ \right)
\end{align*}
is 1D $\delta$-sparse around $x_0$ at scale $\rho$; here the index $(j,\pm)$ is chosen such that $|u(x_0,s)|=u_j^\pm(x_0,s)$ (resp. $|\omega(x_0,s)|=\omega_j^\pm(x_0,s)$), and the pair $(\lambda,\delta)$ is chosen such that the followings hold:
\begin{align*}
\lambda h+(1-h)=2\lambda\ ,\qquad h=\frac{2}{\pi}\arcsin\frac{1-\delta^2}{1+\delta^2}\ , \qquad \frac{1}{1+\lambda}<\delta<1\ .
\end{align*}
(Note that such pair exists and a particular example is that when $\delta=\frac{3}{4}$, $\lambda>\frac{1}{3}$.) Then, there exists $\gamma>0$ such that $u\in L^\infty((0,T^*+\gamma); L^\infty)$, i.e. $T^*$ is not a blow-up time.
\end{theorem}

With Theorem~\ref{th:MainThmVel} (setting $p=2$) and Theorem~\ref{th:MainThmVor} (setting $p=1$) we are able to generalize the above results as follows:
\vspace{-0.05in}
\begin{theorem}\label{th:SparsityRegDk}
Let $u$ (resp. $\omega$) be in $C([0,T^*), L^\infty)$ where $T^*$ is the first possible blow-up time, and assume, in addition, that $u_0\in L^\infty\cap L^2$ (resp. $\omega_0\in L^\infty\cap L^1$). Let $t$ be an escape time of $D^ku(t)$ (resp. $D^k\omega(t)$), and suppose that there exists a temporal point
\begin{align*}
&\qquad\quad s=s(t)\in \left[t+\frac{1}{4^{k+1}c(M,\|u_0\|_2)^2\|D^ku(t)\|_\infty^{2d/(2k+d)}},\ t+\frac{1}{4^kc(M,\|u_0\|_2)^2\|D^ku(t)\|_\infty^{2d/(2k+d)}}\right]
\\
&\left(\textrm{resp. } s=s(t)\in \left[t+\frac{1}{4^{k+1}c(M,\|\omega_0\|_1)\|D^k\omega(t)\|_\infty^{3/(k+3)}},\ t+\frac{1}{4^kc(M,\|\omega_0\|_1)\|D^k\omega(t)\|_\infty^{3/(k+3)}}\right]\ \right)
\end{align*}
such that for any spatial point $x_0$, there exists a scale $\rho\le \frac{1}{2^kc(M)\|D^ku(s)\|_\infty^{\frac{d}{2k+d}}}$ $\left(\textrm{resp. }\rho\le \frac{1}{2^kc(M)\|D^k\omega(s)\|_\infty^{\frac{3/2}{k+3}}}\right)$ with the property that the super-level set
\begin{align*}
&\qquad\quad V_\lambda^{j,\pm}=\left\{x\in\bR^d\ |\ (D^ku)_j^\pm(x,s)>\lambda \|D^ku(s)\|_\infty\right\}
\\
&\left(\textrm{resp. }\Omega_\lambda^{j,\pm}=\left\{x\in\bR^3\ |\ (D^k\omega)_j^\pm(x,s)>\lambda \|D^k\omega(s)\|_\infty\right\}\ \right)
\end{align*}
is 1D $\delta$-sparse around $x_0$ at scale $\rho$; here the index $(j,\pm)$ is chosen such that $|D^ku(x_0,s)|=(D^ku)_j^\pm(x_0,s)$ (resp. $|D^k\omega(x_0,s)|=(D^k\omega)_j^\pm(x_0,s)$), and the pair $(\lambda,\delta)$ is chosen such that the followings hold:
\begin{align*}
\lambda h+(1-h)=2\lambda\ ,\qquad h=\frac{2}{\pi}\arcsin\frac{1-\delta^2}{1+\delta^2}\ , \qquad \frac{1}{1+\lambda}<\delta<1\ .
\end{align*}
(Note that such pair exists and a particular example is that when $\delta=\frac{3}{4}$, $\lambda>\frac{1}{3}$.) Then, there exists $\gamma>0$ such that $u\in L^\infty((0,T^*+\gamma); L^\infty)$, i.e. $T^*$ is not a blow-up time.
\end{theorem}

The following lemma is a generalization of the $H^{-1}$-sparseness results in \citet{Farhat2017} and \citet{Bradshaw2019}.

\begin{lemma}\label{le:HkSparse}
Let $r\in(0,1]$ and $f$ be a bounded function from $\bR^d$ to $\bR^d$ with continuous partial derivatives
of order $k$. Then, for any tuple $(\zeta,\lambda, \delta, p)$, $\zeta\in\bN^d$ with $|\zeta|=k$, $\lambda\in (0,1)$, $\delta\in(\frac{1}{1+\lambda},1)$ and $p>1$, there exists $c^*(\zeta,\lambda,\delta,d,p)>0$ such that if
\begin{align}\label{eq:ZalphaCond}
\|D^\zeta f\|_{W^{-k,p}} \le c^*(\zeta,\lambda,\delta,d,p)\ r^{k+\frac{d}{p}} \|D^\zeta f\|_\infty
\end{align}
then each of the super-level sets
\begin{align*}
S_{\zeta,\lambda}^{i,\pm}=\left\{x\in\bR^d\ |\ (D^\zeta f)_i^\pm(x)>\lambda \|D^\zeta f\|_\infty\right\}\ , \qquad 1\le i\le d, \quad \zeta\in\bN^d
\end{align*}
is 3D $\delta$-sparse at scale $r$.
\end{lemma}

This leads to
\begin{theorem}\label{th:LerayZalpha}
Let $u$ be a Leray solution (a global-in-time weak solution satisfying the global energy inequality), and assume that $u$ is in $C((0,T^*), L^\infty)$ for some $T^*>0$. Then for any $t\in (0,T^*)$ the super-level sets
\begin{align*}
&\qquad\quad S_{\zeta,\lambda}^{i,\pm}=\left\{x\in\bR^d\ |\ (D^\zeta u)_i^\pm(x)>\lambda \|D^\zeta u\|_\infty\right\}\ , \qquad 1\le i\le d, \quad \zeta\in\bN^d
\\
&\left(\textrm{resp. } S_{\zeta,\lambda}^{i,\pm}=\left\{x\in\bR^3\ |\ (D^\zeta \omega)_i^\pm(x)>\lambda \|D^\zeta \omega\|_\infty\right\}\ , \qquad 1\le i\le 3, \quad \zeta\in\bN^3\ \right)
\end{align*}
are $d$-dimensional (resp. $3D$) $\delta$-sparse around any spatial point $x_0$ at scale
\begin{align}\label{eq:kNaturalScale}
r^*=c(\|u_0\|_2)\ \|D^\zeta u(t)\|_\infty^{-\frac{2}{2k+d}} \quad \left(\textrm{resp. } r^*=c(\|u_0\|_2)\ \|D^\zeta \omega(t)\|_\infty^{-\frac{2}{2k+5}}\ \right)
\end{align}
provided $r^* \in (0, 1]$.
In other words, $D^\zeta u(t)\in Z_\alpha(\lambda,\delta,c_0)$ with $\alpha=1/(k+d/2)$ (resp. $D^\zeta \omega(t)\in Z_\alpha(\lambda,\delta,c_0)$ with $\alpha=1/(k+5/2)$). Moreover, for any $p>2$, if we assume
\begin{align*}
u\in C((0,T^*), L^\infty)\cap L^\infty((0,T^*], L^p)
\end{align*}
then for any $t\in (0,T^*)$ the super-level sets $S_{\zeta,\lambda}^{i,\pm}$ are $d$-dimensional $\delta$-sparse around any spatial point $x_0$ at scale
\begin{align*}
r^*=c\left(\sup_{t<T^*}\|u(t)\|_{L^p}\right) \frac{1}{\|D^\zeta u(t)\|_\infty^{1/(k+d/p)}}
\end{align*}
provided $r^* \in (0, 1]$,
i.e. $D^\zeta u(t)\in Z_\alpha(\lambda,\delta,c_0)$ with $\alpha=1/(k+d/p)$.
\end{theorem}

Henceforth, we will assume $D^k=\partial_{x_1}^k$ (the proof for other multi-indexes is analogous), and denote the $L^\infty$-norm simply by $\| \cdot \|$.  

\medskip

Let us generalize the notation adopted in \citet{Grujic2019} 
in the following manner,
\begin{align}
&\cR_{m,n}(j,c,t):=\frac{\|D^ju(t)\|^{\frac{1}{j+1}}}{c^{\frac{j-n}{j-n+1}}((j+m)!)^{\frac{1}{j+1}}}\ ,
\qquad T_{j,M}(t) := (M-1)^2\ c^{\frac{2j}{j+1}}\ \|D^j u(t)\|^{-\frac{2}{j+1}}\ , \label{eq:RealScaleNote}
\\
&\cC_{m,n}(j,c,\varepsilon,t_0,t):=\frac{\left(\|D^ju(\cdot,\varepsilon (t-t_0)^{1/2},t)\|+\|D^jv(\cdot,\varepsilon (t-t_0)^{1/2},t)\|\right)^{\frac{1}{j+1}}} {c^{\frac{j-n}{j-n+1}}((j+m)!)^{\frac{1}{j+1}}}\ , \quad m,n\in \bZ. \notag 
\end{align}
Then, the result with `descending assumption' can be improved as follows.

\begin{theorem}\label{le:DescendDer}
Let $u$ be a Leray solution of \eqref{eq:NSE1}-\eqref{eq:NSE3} initiated at $u_0$. Define
\begin{align*}
\cZ_k(t):=\|D^{k}u(t)\|^{\frac{d/2-1}{(k+1)(k+d/2)}}\ .
\end{align*}
Suppose $\ell$ is sufficiently large such that $\|u_0\|\lesssim (1+\epsilon)^{\ell}$. For a fixed $k\ge\ell$ and fixed $n, m\in\bZ$ ($n\le 0$), suppose that
\begin{align}\label{eq:DescDerOrd}
\cR_{m,n}(k,c,t_0) \ge \cR_{m,n}(j,c,t_0) \ , \qquad \forall  j\ge k
\end{align}
for a suitable constant $c=c(k)$ which also satisfies
\begin{align}\label{eq:ParaAdjMaxP}
\lambda h^* + \left(1+\frac{2e}{\eta}\frac{c^{\frac{k}{(k+n+1)^2}}\ \cZ_k(t_0)}{(k+m)^{m/k}} \exp\left(\frac{2e}{\eta}\frac{c^{\frac{k}{(k+n+1)^2}}\ \cZ_k(t_0)}{(k+m)^{m/k}} \right)\right) (1-h^*) \le \mu
\end{align}
where
$h^*=\frac{2}{\pi} \arcsin \frac{1-\delta^{2/d}}{1+\delta^{2/d}}$, $(1+\eta)^d=\frac{\delta(1+\lambda)+1}{2}$ and $\mu$ is a positive constant. Then there exist $t_*, T_*>t_0$ and a constant $\mu_*$ such that
\begin{align*}
\|D^ku(s)\|\le \mu_* \|D^ku_0\| \ , \qquad \forall\ t_*\le s\le T_*\ .
\end{align*}
Here $\mu_*$ is smaller than the threshold $M$ for $D^ku$ given in Theorem~\ref{th:MainThmVel} (and could be less than 1 with proper choices of $c$ and $\mu$). A particular consequence (with an argument by contradiction) of this result is
that--for sufficently small values of $\mu$--\eqref{eq:DescDerOrd} can not coexist with \eqref{eq:ParaAdjMaxP}.
\end{theorem}

\begin{proof}
Pick $k_*$ such that \eqref{eq:AscDerCond} holds for $\ell=k$ and $k=k_*$. According to Theorem~\ref{th:MainThmVel}, there exists
\begin{align*}
T_*=C(M)\|u_0\|_2^2 \cdot \min_{k\le j\le k_*}\! 4^{-j} \|D^{j}u_0\|^{-\frac{d}{j+d/2}}
\end{align*}
such that
\begin{align*}
\sup_{t_0<s<t_0+T_*}\! \|D^{j}u(s)\| \le M\|D^{j}u_0\|\ , \qquad \forall\ k\le j\le k_*\ ,
\end{align*}
i.e. the uniform time span for the real solutions from $k$-th level to $k_*$-th level.

We first prove for the case that the order of `the tail of \eqref{eq:DescDerOrd} after $k_*$' continues for all $s$ up to $t_0+T_*$, that is assuming, for any $t_0<s<t_0+T_*$,
\begin{align}\label{eq:DescDerOrdAlls}
\cR_{m,n}(k,c,s) \ge \cR_{m,n}(j,c,s) \ , \qquad \forall\ j\ge k_*\ .
\end{align}
Fix an $x_0\in\bR^d$. Following the assumption~\eqref{eq:DescDerOrdAlls}, if $z\in B_{r_s}(x_0,0)\subset\bC^d$ with
\begin{align}\label{eq:NaScSparse}
r_s= \left(\frac{\sup_{s}\|u(s)\|_{L^2}}{c^*(\zeta,\lambda,\delta,d,p)}\right)^{\frac{1}{k+d/2}} \|D^ku(s)\|^{-\frac{1}{k+d/2}} \approx (\eta/2)^{-1}\|D^ku(s)\|^{-\frac{1}{k+d/2}}
\end{align}
(where $c^*$ is given in \eqref{le:HkSparse} and such choice for the radius becomes natural as we apply Theorem~\ref{th:LerayZalpha} later and Proposition~\ref{prop:HarMaxPrin} at the end of the proof) the complex extension of $D^ku(s)$ at any spatial point $x_0$ satisfies (for $z\neq x_0$)
\begin{align*}
\left|D^ku(z, s)\right| &\le \left(\sum_{0\le i\le k_*-k} + \sum_{i>k_*-k}\right)\frac{\left|D^{k+i}u(x_0, s)\right|}{i!} |z-x_0|^i =: \cI_s(z) + \cJ_s(z)
\end{align*}
where (with in mind \eqref{eq:DescDerOrdAlls} )
\begin{align*}
\cJ_s(z) &\le \sum_{i>k_*-k} \left(\cR_{m,0}(k+i,c,s) \right)^{k+i+1} \frac{(k+m+i)!}{i!}\ c^{k+i}\ |z-x_0|^i
\\
&\le \sum_{i>k_*-k} \left(\cR_{m,n}(k+i,c,s) \right)^{k+i+1} \frac{(k+m+i)!}{c^{-\frac{n}{k+n+i+1}}i!}\ c^{k+i}\ |z-x_0|^i
\\
&\le \sum_{i>k_*-k} \left(\cR_{m,n}(k,c,s)\right)^{k+i+1} \frac{(k+m+i)!}{c^{-\frac{n}{k+n+i+1}}\ i!}\ c^{k+i}\ |z-x_0|^i\ , \qquad \forall\ t_0<s<t_0+T_*\ .
\end{align*}
Thus, for any $s<t_0+T_*$, (with in mind \eqref{eq:NaScSparse} and the fact that $n\le 0$ and $c<1$)
\begin{align*}
\sup_{z\in B_{r_s}(x_0,0)} \cJ_s(z) &\le c^{-1} \sum_{i>k_*-k} \left(c\ \cR_{m,n}(k,c,s)\right)^{k+i+1} \frac{(k+m+i)!}{c^{-\frac{n}{k+n+i+1}}\ i!} \cdot r_s^i
\\
&\le \|D^ku(s)\| \sum_{i>k_*-k} c^{\frac{-n\cdot i}{(k+n+1)(k+n+i+1)}} \frac{(k+m+i)!}{(k+m)!\ i!} \left(c \cdot \cR_{m,n}(k,c,s)\cdot r_s\right)^i
\\
&\le \|D^ku(s)\| \sum_{i>k_*-k} \frac{(k+m+i)!}{(k+m)!\ i!} \left(\frac{c^{\frac{1}{k+n+1}}\ \|D^{k}u(s)\|^{\frac{1}{k+1}}\ r_s}{c^{\frac{n}{(k+n+1)(k+n+i+1)}} \left((k+m)! \right)^{\frac{1}{k+1}}}\right)^i
\\
&\le \|D^ku(s)\| \sum_{i>k_*-k} \frac{(k+m+i)!}{(k+m)!\ i!} \left(\frac{c^{\frac{1}{k+n+1} \frac{k+i+1}{k+n+i+1}} \|D^{k}u(s)\|^{\frac{d/2-1}{(k+1)(k+d/2)}}} {(\eta/2)\left((k+m)!\right)^{\frac{1}{k+1}}}\right)^{i}\ .
\end{align*}
By Theorem~\ref{th:MainThmVel} and the assumption \eqref{eq:DescDerOrd}, if $s\le t_0+T_*$ then
\begin{align*}
\sup_{z\in B_{r_s}(x_0,0)} \cI_s(z) &\le M \sum_{0\le i\le k_*-k} \frac{\|D^{k+i}u_0\|}{i!}\ r_s^i
\\
&\le M \sum_{0\le i\le k_*-k} \left(\cR_{m,n}(k,c,t_0)\right)^{k+i+1} \frac{(k+m+i)!}{c^{-\frac{n}{k+n+i+1}}\ i!}\ c^{k+i}\ r_s^i
\\
&\le M \|D^{k}u_0\| \sum_{0\le i\le k_*-k} \frac{(k+m+i)!}{(k+m)!\ i!} \left(\frac{c^{\frac{1}{k+n+1} \frac{k+i+1}{k+n+i+1}} \|D^{k}u(s)\|^{\frac{d/2-1}{(k+1)(k+d/2)}}} {(\eta/2)\left((k+m)!\right)^{\frac{1}{k+1}}}\right)^{i} \ .
\end{align*}
We will complete the proof by way of contradiction. Suppose there exists an $t<t_0+T_*$ such that $\|D^ku(t)\|> \mu_* \|D^ku_0\|$, then $r_t\le \mu_*^{-\frac{1}{k+d/2}}r_0$ and
\begin{align*}
\sup_{z\in B_{r_t}(x_0,0)} \cI_t(z) &\le M \|D^{k}u_0\| \sum_{0\le i\le k_*-k} \frac{(k+m+i)!}{(k+m)!\ i!} \left(\frac{c^{\frac{1}{k+n+1} \frac{k+i+1}{k+n+i+1}} \|D^{k}u_0\|^{\frac{d/2-1}{(k+1)(k+d/2)}}} {\mu_*^{\frac{1}{k+d/2}}(\eta/2)\left((k+m)!\right)^{\frac{1}{k+1}}}\right)^{i}\ .
\end{align*}
Combining the estimates for $\cJ_t(z)$ and $\cI_t(z)$ yields

\begin{align*}
\sup_{z\in B_{r_t}(x_0,0)} \left|D^ku(z, t)\right|  &\le M\|D^ku_0\| \sum_{i>k_*-k} \frac{(k+m+i)!}{(k+m)!\ i!} \left(\frac{c^{\frac{1}{k+n+1} \frac{k+i+1}{k+n+i+1}}\ \cZ_k(t)} {(\eta/2)\left((k+m)!\right)^{\frac{1}{k+1}}}\right)^{i}
\\
& + M \|D^{k}u_0\| \sum_{0\le i\le k_*-k} \frac{(k+m+i)!}{(k+m)!\ i!} \left(\frac{c^{\frac{1}{k+n+1} \frac{k+i+1}{k+n+i+1}}\ \cZ_k(t_0) } {\mu_*^{\frac{1}{k+d/2}}(\eta/2)\left((k+m)!\right)^{\frac{1}{k+1}}}\right)^{i}
\\
&\le M\|D^ku_0\| \sum_{i\ge 0} \frac{(k+m+i)!}{(k+m)!\ i!} \left(\frac{c^{\frac{1}{k+n+1} \frac{k+i+1}{k+n+i+1}}\ \cZ_k(t_0)} {\mu_*^{\frac{1}{k+d/2}}(\eta/2)\left((k+m)!\right)^{\frac{1}{k+1}}}\right)^{i}
\\
&\le M\|D^ku_0\| \left(1 + \sum_{i=1}^\infty \frac{c^{\frac{1}{k+n+1} \frac{k+i+1}{k+n+i+1}}\ \cZ_k(t_0)} {\mu_*^{\frac{1}{k+d/2}} (\eta/2)\left((k+m)!\right)^{\frac{1}{k+1}}} \right.
\\
&\qquad \left. \times \frac{k+m+i}{i}\ \frac{(k+m+i-1)!}{(k+m)!\ (i-1)!} \left(\frac{c^{\frac{1}{k+n+1} \frac{k+i+1}{k+n+i+1}}\ \cZ_k(t_0)} {\mu_*^{\frac{1}{k+d/2}} (\eta/2)\left((k+m)!\right)^{\frac{1}{k+1}}}\right)^{i-1} \right) \ .
\end{align*}
Note that for $n<0$ and $i$ small,
\begin{align*}
c^{\frac{1}{k+n+1} \frac{k+i+1}{k+n+i+1}} \big/ c^{\frac{k}{(k+n+1)^2}} = c^{\frac{1}{k+n+1} \left(\frac{1}{k+n+1} + \frac{ni}{(k+n+1)(k+n+i+1)}\right)} \lesssim \frac{(k+m+1)i}{k+m+i}
\end{align*}
while for $n<0$ and $i$ large, $c^{\frac{1}{k+n+1} \frac{k+i+1}{k+n+i+1}} \big/ c^{\frac{k}{(k+n+1)^2}}$ is negligible compared to $\left((k+m)!\right)^{\frac{1}{k+1}}$; thus
\begin{align*}
\sup_{z\in B_{r_t}(x_0,0)} \left|D^ku(z, t)\right| &\le M\|D^ku_0\| \left(1 + \frac{c^{\frac{k}{(k+n+1)^2}} \ \cZ_k(t_0)} {\mu_*^{\frac{1}{k+d/2}} (\eta/2)\left((k+m)!\right)^{\frac{1}{k+1}}} \right.
\\
&\qquad\quad \left.\times\ (k+m+1) \left(1- \frac{c^{\frac{k}{(k+n+1)^2}} \ \cZ_k(t_0)} {\mu_*^{\frac{1}{k+d/2}} (\eta/2)\left((k+m)!\right)^{\frac{1}{k+1}}}\right)^{-k-m-1} \right)\ .
\end{align*}
Since the above estimates hold for all $x_0$, if $\|u_0\|\lesssim (1+\epsilon)^{\ell}$, $M, \mu_*\approx1$ and $k$ is sufficiently large,
\begin{align*}
&\underset{y\in B_{r_t}(0)}{\sup} \|D^ku(\cdot,y,t)\|_{L^\infty} + \underset{y\in B_{r_t}(0)}{\sup} \|D^kv(\cdot,y,t)\|_{L^\infty}
\\
&\qquad \le M\left(1 + \frac{c^{\frac{k}{(k+n+1)^2}}(k+m+1)\ \cZ_k(t_0) }{(\eta/2)\left((k+m)!\right)^{\frac{1}{k+1}}}  \exp\left(\frac{c^{\frac{k}{(k+n+1)^2}} (k+m+1)\ \cZ_k(t_0)}{\mu_*^{\frac{1}{k+d/2}} (\eta/2)((k+m)!)^{\frac{1}{k+1}}}\right) \right) \|D^ku_0\|
\\
&\qquad \le M\left(1 + \frac{2e^{1-\frac{m}{k}} c^{\frac{k}{(k+n+1)^2}}\ \cZ_k(t_0) } {\eta\ (k+m)^{m/k}} \exp\left(\frac{2e^{1-\frac{m}{k}} c^{\frac{k}{(k+n+1)^2}} \ \cZ_k(t_0) }{\eta\ \mu_*^{\frac{1}{k+d/2}}(k+m)^{m/k}}\right) \right)\|D^ku_0\|\ .
\end{align*}
By Theorem~\ref{th:LerayZalpha}, for any spatial point $x_0$ there exists a direction $\nu$ along which the super-level set
\begin{align*}
S_{k,\lambda}^{i,\pm}=\left\{x\in\bR^d\ |\ (D^k u(t))_i^\pm(x)>\lambda \|D^k u(t)\|_\infty\right\}
\end{align*}
is 1-D $\delta^{1/d}$-sparse at scale $r_t$ given in \eqref{eq:NaScSparse}. Note that the results in Proposition~\ref{prop:HarMaxPrin} are scaling invariant and--for simplicity--assume $r_t=1$ and $\nu$ is a unit vector. Define
\begin{align*}
K=\overline{(x_0-\nu,x_0+\nu)\setminus S_{k,\lambda}^{i,\pm}}\ .
\end{align*}
Then--by sparseness--$|K|\ge 2(1-\delta^{1/d})$. If $x_0\in K$, the result follows immediately. If $x_0\notin K$, then by Proposition~\ref{prop:HarMaxPrin} and the above estimate for $D^ku(z, t)$,
\begin{align*}
|D^ku(x_0, t)|&\le \lambda \|D^k u(t)\|_\infty\ h^* + \sup_{z\in B_{r_t}(x_0,0)} \left|D^ku(z, t)\right| (1-h^*)
\\
&\le \lambda M\|D^k u_0\|_\infty\ h^* + \frac{M\ c^{\frac{k}{(k+n+1)^2}} \ \cZ_k(t_0) }{(\eta/2e)\ (k+m)^{m/k}} \exp\left(\frac{2e}{\eta} \frac{c^{\frac{k}{(k+n+1)^2}} \ \cZ_k(t_0) }{(k+m)^{m/k}}\right) \|D^k u_0\|_\infty (1-h^*)
\end{align*}
where $h^*=\displaystyle \frac{2}{\pi} \arcsin \frac{1-\delta^{2/d}}{1+\delta^{2/d}}$. Hence, if condition~\eqref{eq:ParaAdjMaxP} is satisfied, we observe a contradiction (from the above result) that $\|D^ku(t)\|\le \mu_* \|D^ku_0\|$ with $\mu_*=M\mu$.

For the opposite case to \eqref{eq:DescDerOrdAlls}, the proof is the same as that part of Theorem~3.9 in \citet{Grujic2019} with the subscripts $m,n$ for $\{\cR(k,c,t)\}$.
\end{proof}

\begin{remark}
If assuming $D^j u(t)\in Z_j(\lambda,\delta,c)$ with $\alpha=1/(k+1)$ for all $j\ge k_*$, then one can prove the statement with the same $\mu_*$ for much longer duration $T_*$.
\end{remark}

With a similar deduction, one can prove the following refined results of Lemma~3.11 and Corollary~3.12 in \citet{Grujic2019}

\begin{lemma}\label{le:ScaleBound}
Suppose $\|u_0\|\lesssim (1+\epsilon)^{\ell}$ and $u(t)$ is the solution to \eqref{eq:NSE1}-\eqref{eq:NSE3}. For any $\kappa>\ell$ and $n, m\in\bZ$ such that $|n|\ll \kappa$ and $|m|\ll\ell$, if \eqref{eq:ParaAdjMaxP} is satisfied (for $\kappa$), one of the two cases must occur:

(I)$^*$ There exist $t$ and $k \ge \kappa$ such that
\begin{align*}
\cR_{m,n}(j,c,t) \ge \cR_{m,n}(k,c,t)\ , \qquad \forall \ell\le j\le k
\end{align*}

(II)$^*$ Otherwise,
\begin{align*}
\sup_{s>t_0}\ \max_{j\ge\ell} \cR_{m,n}(j,c,s)  \le \max_{\ell\le j\le \kappa} \cR_{m,n}(j,c,t_0)
\end{align*}
\end{lemma}

\begin{corollary}\label{cor:DerOrdConfig}
Let $u$ be a Leray solution of \eqref{eq:NSE1}-\eqref{eq:NSE3}. Suppose $\ell$ is sufficiently large such that $\|u_0\|\lesssim (1+\epsilon)^{\ell}$. For any $\kappa>\ell$ and $n, m\in\bZ$ such that $|n|\ll \kappa$ and $|m|\ll \ell$, if there exists a sequence of positive numbers $\{c_j\}_{j=\ell}^k$ such that $c_{j+1}\le c_j<1$ and for some fixed $p\in\bN^+$
\begin{align}\label{eq:ParaAdjMaxPIndW}
\lambda h^* + \left(1+\frac{(1+\epsilon)^{\ell/k} c_{j-p}^{\frac{k}{(k+n+1)^2}}\ \cZ_k(t_0)}{(\eta/2e)(k+m)^{m/k}} \exp\left(\frac{(1+\epsilon)^{\ell/k} c_{j-p}^{\frac{k}{(k+n+1)^2}}\ \cZ_k(t_0)}{(\eta/2e)(k+m)^{m/k}}\right)\right) (1-h^*) \le M^{-1}
\end{align}
is satisfied for each $j$ (where $\eta$ and $h^*$ are defined as in Theorem~\ref{le:DescendDer} and $M$ is given in Theorem~\ref{th:MainThmVel}), then for sufficiently large $t$, one of the following two cases must occur:

(I)$^*$ There exist temporal point $t\ge t_0$, $k\ge\kappa +p$ and constants $\displaystyle \cB_{k,i}\le \prod_{j=i}^k c_j^{-\frac{1}{(j+1)(j+2)}}$ such that
\begin{align*}
\cR_{m,n}(i,c_\ell,t) \le \cB_{k,i} \cdot \cR_{m,n}(k,c_\ell,t) \qquad\textrm{for all}\quad \ell\le i\le k\ ;
\end{align*}

(II)$^*$ Otherwise there exist $t\ge t_0$, $r<\kappa$ and constants $\displaystyle \cB_{r,i}\le \prod_{j=i}^r c_j^{-\frac{1}{(j+1)(j+2)}}$ such that
\begin{align*}
\sup_{s>t}\ \cR_{m,n}(i,c_\ell,s) \le \cB_{r,i} \cdot \cR_{m,n}(r,c_\ell,t) \qquad\textrm{for all}\quad \ell\le i\le r
\end{align*}
and constants $\displaystyle \cC_{r,i}\le c_r^{\frac{1}{r+1}-\frac{1}{i+1}}$ such that
\begin{align*}
\sup_{s>t}\ \cR_{m,n}(i,c_r,s) \le \cC_{r,i} \cdot \max_{\ell\le j\le \ell+p} \cR_{m,n}(j,c_\ell,t_0) \qquad\textrm{for all}\quad i>r\ .
\end{align*}
\end{corollary}

In this setting, the result with `ascending assumption' becomes

\begin{theorem}\label{le:AscendDer}
Let $u$ be a Leray solution initiated at $u_0$ and
suppose that 
\begin{align}\label{eq:AscDerOrd}
\cR_{m,n}(j,c,t_0) \le \cR_{m,n}(k,c,t_0)\, \qquad \forall \ell\le j\le k
\end{align}
where $c$, $\ell$ and $k$ satisfy
\begin{align}\label{eq:AscDerCond}
c \|u_0\|_2 \|u_0\|^{d/2-1} \frac{(\ell!)^{1/2}\ell}{(\ell/2)!} \lesssim k.
\end{align}
If  $T^{1/2}\lesssim C(\|u_0\|_2,\ell,k,m,n)^{-1} \|D^ku_0\|^{-\frac{1}{k+1}}$ (where $C(\|u_0\|_2,\ell,k)$ also depends on the pre-specified constant $c$ and the threshold $M$; the constant $c=c(k)$ in \eqref{eq:AscDerOrd} are chosen according to the formation of the ascending chains in Corollary~\ref{cor:DerOrdConfig} and Lemma~\ref{le:ScaleBound}, which is originally determined by the assumption~\eqref{eq:ParaAdjMaxP} in Theorem~\ref{le:DescendDer}), then for any $\ell\le j \le k$ the complex solution of \eqref{eq:NSE1}-\eqref{eq:NSE3} has the following upper bounds:
\begin{align}\label{eq:AscDerUpBdd}
&\underset{t\in(0,T)}{\sup} \cC_{m,n}(j,c,\varepsilon,t_0,t) \le M^{\frac{1}{j+1}}\cR_{m,n}(j,c,t_0) + \left(j+ k\right)^{\frac{1}{j+1}} \cR_{m,n}(k,c,t_0)
\end{align}
where $\cD_t$ is given by \eqref{eq:AnalDom}. For the real solutions the above result becomes
\begin{align}
\underset{t\in(0,\tilde{T})}{\sup} \cR_{m,n}(j,c,t) \le \cR_{m,n}(j,c,t_0) + \left(j+ k\right)^{\frac{1}{j+1}} \cR_{m,n}(k,c,t_0) \label{eq:AscDerUpBddReal}
\end{align}
where $\tilde{T}$ does not depend on $M$.
\end{theorem}

\begin{proof}
Similar to the proof of Theorem~3.8 in \citet{Grujic2019}, by induction
\begin{align*}
&\left\|\int_0^t e^{(t-s)\Delta}D^j(U^{(n)}\cdot \nabla)U^{(n)}ds\right\| \lesssim t^{\frac{1}{2}} \sum_{i=0}^j \binom j i \sup_{s<T}\|D^iU^{(n-1)}(s)\| \sup_{s<T}\|D^{j-i}U^{(n-1)}(s)\|
\\
&\quad \lesssim t^{\frac{1}{2}} \left(\sum_{\ell \le i\le j-\ell} \binom j i L_{n-1}^{(i)}L_{n-1}^{(j-i)}+ 2\sum_{0\le i\le \ell} \binom j i K_{n-1}^{\frac{\ell-i}{\ell+d/2}}\left(L_{n-1}^{(\ell)}\right)^{\frac{i+d/2}{\ell+d/2}}L_{n-1}^{(j-i)} \right) := t^{\frac{1}{2}} \left(I_n+2J_n\right)
\end{align*}
where $L_n^{(j)} := \underset{t<T}{\sup}\ \|D^jU^{(n)}\|_{L^\infty}+\underset{t<T}{\sup}\ \|D^jV^{(n)}\|_{L^\infty}$ and $K_n :=\underset{t<T}{\sup}\ \|U^{(n)}\|_{L^2}+\underset{t<T}{\sup}\ \|V^{(n)}\|_{L^2}$.
By induction hypothesis, assumptions~\eqref{eq:AscDerOrd} and \eqref{eq:AscDerCond}, for $j>2\ell$,
\begin{align*}
I_n & \lesssim \sum_{\ell \le i\le j-\ell} \binom j i \left(1+ \frac{i+ c\|u_0\|_2 \|u_0\|^{1/2} \frac{(\ell!)^{1/2}\ell}{(\ell/2)!}}{c^{-\frac{1}{k +1}}(k !)^{\frac{1}{k+1}}}\right) \frac{c^{\frac{(i-n)(i+1)}{i-n+1}}\ (i+m)!}{\left(c^{\frac{k-n}{k-n+1}}((k+m)!)^{\frac{1}{k+1}}\right)^{i+1}} \|D^ku_0\|^{\frac{i+1}{k+1}}
\\
&\qquad \times \left(1+ \frac{j-i+ c\|u_0\|_2 \|u_0\|^{1/2} \frac{(\ell!)^{1/2}\ell}{(\ell/2)!}}{c^{-\frac{1}{k +1}}(k !)^{\frac{1}{k+1}}}\right) \frac{c^{\frac{(j-i-n)(j-i+1)}{j-i-n+1}}(j-i+m)!}{\left(c^{\frac{k-n}{k -n+1}}((k+m)!)^{\frac{1}{k+1}} \right)^{j-i+1}} \|D^ku_0\|^{\frac{j-i+1}{k+1}}
\\
&\lesssim \sum_{\ell \le i\le j-\ell} \binom j i \left(\frac{c^{\frac{(i-n)}{i-n+1}}}{c^{\frac{(k-n)}{k-n+1}}}\right)^{i+1} \left(\frac{c^{\frac{(j-i-n)}{j-i-n+1}}}{c^{\frac{(k-n)}{k-n+1}}}\right)^{j-i+1} \frac{(i+m)!\ (j-i+m)!}{((k+m)!)^{\frac{j+2}{k+1}}}\ \|D^ku_0\|^{\frac{j+2}{k+1}}
\\
&\lesssim \sum_{\ell \le i\le j-\ell} \left(\frac{c^{\frac{(j-n)}{j-n+1}}}{c^{\frac{(k-n)}{k-n+1}}}\right)^{j+2} \frac{c^{\frac{j+2}{j-n+1}}}{c^{\frac{i+1}{i-n+1}+\frac{j-i+1}{j-i-n+1}}}\ \frac{j!\ (i+m)!\ (j-i+m)!}{((k+m)!)^{\frac{j+2}{k+1}}\ i!\ (j-i)!}\ \|D^ku_0\|^{\frac{j+2}{k+1}}.
\end{align*}
By Theorem~\ref{th:LinftyIVP} we may assume without loss of generality $\|D^ku_0\|\lesssim k!\|u_0\|^{k+1}$, thus
\begin{align*}
J_n &\lesssim \sum_{0\le i\le \ell} \binom j i \|u_0\|_2^{\frac{\ell-i}{\ell+d/2}} \left(2+ \frac{\ell}{((k+m)!)^{\frac{1}{k+1}}}\right)^{\frac{i+d/2}{\ell+d/2}} \left(\frac{c^{\frac{\ell-n}{\ell-n+1}}((\ell+m)!)^{\frac{1}{i+1}}} {c^{\frac{k-n}{k-n+1}}((k+m)!)^{\frac{1}{k+1}}}\right)^{\frac{i+d/2}{\ell +d/2}(\ell+1)}
\\
&\qquad\qquad \|D^ku_0\|^{\frac{\ell+1}{k+1} \frac{i+d/2}{\ell +d/2}} \left(2+ \frac{j-i}{((k+m)!)^{\frac{1}{k+1}}}\right) \frac{c^{\frac{(j-i-n)(j-i+1)}{j-i-n+1}}(j-i+m)!}{\left(c^{\frac{k-n}{k -n+1}}((k+m)!)^{\frac{1}{k+1}} \right)^{j-i+1}}\ \|D^ku_0\|^{\frac{j-i+1}{k+1}}
\\
&\lesssim \|u_0\|_2 \sum_{0\le i\le \ell} \left(\frac{c^{\frac{(j-n)}{j-n+1}}}{c^{\frac{(k-n)}{k-n+1}}}\right)^{j+2} \frac{c^{\frac{j+2}{j-n+1}}}{c^{\frac{i+1}{i-n+1}+\frac{j-i+1}{j-i-n+1}}}\ \frac{j!\ (i+m)!\ (j-i+m)!}{((k+m)!)^{\frac{j+2}{k+1}} i!\ (j-i)!}\ \|D^ku_0\|^{\frac{j+2}{k+1}}
\\
& \qquad \times \frac{\left(c^{\frac{k-n}{k-n+1}}((k+m)!)^{\frac{1}{k+1}}\right)^{i+1}}{c^{\frac{(i-n)(i+1)}{i-n+1}}\ (i+m)!} \left(\frac{c^{\frac{\ell-n}{\ell-n+1}}((\ell+m)!)^{\frac{1}{i+1}}} {c^{\frac{k-n}{k-n+1}}((k+m)!)^{\frac{1}{k+1}}}\right)^{\frac{i+d/2}{\ell +d/2}(\ell+1)} \|D^ku_0\|^{\frac{\ell-i}{\ell +d/2}\frac{d/2-1}{k+1}}
\\
&\lesssim c\|u_0\|^{d/2-1} \frac{(\ell!)^{1/2}\ell}{(\ell/2)!}\ \|D^ku_0\|^{\frac{j+2}{k+1}}
\sum_{0\le i\le \ell} \left(\frac{c^{\frac{(j-n)}{j-n+1}}}{c^{\frac{(k-n)}{k-n+1}}}\right)^{j+2}\! \frac{c^{\frac{j+2}{j-n+1}}}{c^{\frac{i+1}{i-n+1}+\frac{j-i+1}{j-i-n+1}}}\ \frac{j!\ (i+m)!\ (j-i+m)!}{((k+m)!)^{\frac{j+2}{k+1}} i!\ (j-i)!}
\end{align*}
The other nonlinear terms are estimated in the same way and these implies, if $s\lesssim \cR_{m,n}(k,c,t_0)^{-1/2}$
\begin{align*}
\|D^jU_n(s)\|\big/\left(c^{\frac{(j-n)(j+1)}{j-n+1}}\cdot (j+m)!\right) \le M\cdot \cR_{m,n}(j,c,t_0)^{j+1} + 2j\cdot \cR_{m,n}(k,c,t_0)^{j+1}\ .
\end{align*}
For $\ell \le j\le 2\ell$, such estimate for both $U_n$ and $V_n$ can be obtained in the same way. So, by induction and the above estimates \eqref{eq:AscDerUpBdd} holds for all $\ell\le j\le k$. The proof of \eqref{eq:AscDerUpBddReal} is similar.
\end{proof}

\medskip

Finally, we state the spatial intermittency based regularity criterion in this setting.

\medskip

\begin{theorem}\label{th:MainResult}
Let $u_0 \in L^\infty \cap L^2$ (resp. $\omega_0 \in L^\infty \cap L^1$) and $u$ in $C((0,T^*), L^\infty)$ where $T^*$ is the first possible blow-up time. Let $c$, $\ell$, $k$ be such that $\|u_0\|\lesssim (1+\epsilon)^{(2/d)\ell}$ (resp. $\|\omega_0\|\lesssim (1+\epsilon)^{(2/d)\ell}$) and \eqref{eq:AscDerCond} holds.
For any index $k\ge \ell$ and temporal point $t$ such that \eqref{eq:AscDerOrd} is satisfied and
\begin{align}\label{eq:tIRRegAscBdd}
&\qquad\quad t+\frac{1}{\cC(\|u_0\|,\ell,k)^2 \|D^ku(t)\|_\infty^{2/(k+1)}} < T^* \\
&\left(\textrm{resp. } t+\frac{1}{\cC(\|\omega_0\|,\ell,k)^2 \|D^k\omega(t)\|_\infty^{2/(k+2)}} < T^* \ \right) \notag
\end{align}
we assume that for every $k\ge \ell$ there exists a temporal point
\begin{align*}
&\qquad\quad s=s(t)\in \left[t+\frac{1}{4\cdot\tilde{\cC}(\|u_0\|,\ell,k)\|D^ku(t)\|_\infty^{2/(k+1)}},\ t+\frac{1}{\tilde{\cC}(\|u_0\|,\ell,k)\|D^ku(t)\|_\infty^{2/(k+1)}}\right]
\\
&\left(\textrm{resp. } s=s(t)\in \left[t+\frac{1}{4\cdot\tilde{\cC}(\|\omega_0\|,\ell,k)\|D^k\omega(t)\|_\infty^{2/(k+2)}},\ t+\frac{1}{\tilde{\cC}(\|\omega_0\|,\ell,k)\|D^k\omega(t)\|_\infty^{2/(k+2)}}\right]\ \right)
\end{align*}
such that, with $(j,\pm)$ is chosen such that $|D^ku(x_0,s)|=(D^ku)_j^\pm(x_0,s)$ (resp. $|D^k\omega(x_0,s)|=(D^k\omega)_j^\pm(x_0,s)$),
\begin{align}\label{eq:kRegScale}
(D^ku)_j^\pm(s) \in Z_{\frac{1}{k+1}}(\lambda,\delta,C_k^{-1}) \quad \left(\textrm{resp. } (D^k\omega)_j^\pm(s)\in Z_{\frac{1}{k+2}}(\lambda,\delta,c_0)\ \right)
\end{align}
where the pair $(\lambda,\delta)$ is chosen such that \eqref{eq:ParaAdjMaxP} in Theorem~\ref{le:DescendDer} holds. Then, there exists $\gamma>0$ such that $u\in L^\infty((0,T^*+\gamma); L^\infty)$.

\end{theorem}

The proof of Theorem~\ref{th:MainResult} is the same as that of Theorem~3.14 in \citet{Grujic2019} with the two lemmas below and the same notions of Type-$\cA$ and Type-$\cB$ sections/strings for $\cR_{m,n}(k, c, t)$ as in Definition~3.15 in \citet{Grujic2019}.

\begin{lemma}\label{le:AstrBdd}
Suppose $\sup_{t>t_0}\|u(t)\|\lesssim (1+\epsilon)^{\ell_i}$ such that \eqref{eq:AscDerCond} is satisfied at any temporal point with $\ell=\ell_i$ and $k=\ell_{i+q}$, and the assumption~\eqref{eq:AscDerCond} holds for all $k\ge \ell_i$. Fix a pair $(m,n)$. If a string $[\ell_i,\ell_{i+q}]$ is of Type-$\cA$ at an initial time $t_0$, then for any $i\le r<i+q$,
\begin{align}\label{eq:RestrAtoB}
\max_{\ell_r\le j\le \ell_{i+q}} \sup_{t_0<s<\tilde t} \cR_{m,n}(j, c(\ell_r), s) \lesssim \Theta(p^*,r) \max_{i\le p\le i+q} \cR_{m,n}(w_p,c(\ell_p), t_0)
\end{align}
where $\Theta(p^*,r) \lesssim \tilde{\cB}_{\ell_{p^*}, \ell_r}\cdot c(\ell_{p^*})/c(\ell_r)$ with $\tilde{\cB}_{i,j}:=\cB_{i,j}$ defined in Corollary~\ref{cor:DerOrdConfig} if $i>j$ and $\tilde{\cB}_{i,j}:=(\cB_{j,i})^{-1}$ if $i<j$, and $\tilde{\epsilon}(\ell_{i+q})$ is a small quantity which will be given explicitly in the proof; the subscript $p^*$ is the index for the maximal $\cR(m_p,c(\ell_p), t_0)$, and $\tilde t$ is the first time when $[\ell_i,\ell_{i+q}]$ switches to a Type-$\cB$ string; we set $\tilde t= \infty$ if $[\ell_i,\ell_{i+q}]$ is always of Type-$\cA$.
\end{lemma}

\begin{lemma}\label{le:BstrBdd}
Suppose $\sup_{t>t_0}\|u(t)\|\lesssim (1+\epsilon)^{\ell_i}$ such that \eqref{eq:ParaAdjMaxP} is satisfied at any temporal point with $\ell=\ell_i$ and $(k,c(k))=(\ell_p,c(\ell_p))$ for any $i\le p\le i+q$. Fix a pair $(m,n)$. If a string $[\ell_i,\ell_{i+q}]$ is of Type-$\cB$ at an initial time $t_0$, then for any $i\le r<i+q$,
\begin{align}\label{eq:RestrBtoA}
\max_{\ell_r\le j\le \ell_{i+q}} \sup_{t_0<s<\tilde t} \cR_{m,n}(j, c(\ell_r), s) \le  \max_{r\le p\le i+q} \cR_{m,n}(w_p,c(\ell_p), t_0)
\end{align}
where $\tilde t$ is the first time when $[\ell_i,\ell_{i+q}]$ switches to a Type-$\cA$ string; we set $\tilde t= \infty$ if $[\ell_i,\ell_{i+q}]$ is always of Type-$\cB$.
\end{lemma}

The proofs of the above two Lemmas, though requiring Corollary~\ref{cor:DerOrdConfig} and Theorem~\ref{le:AscendDer}, are essentially the same as those in \citet{Grujic2019}.

\section{Homogeneity of $\{\cR(k,c,t)\}$}\label{sec:HomASS}

Here we generalize the notion of quotient of derivatives introduced in \citet{Grujic2019} to directional derivative and $L^\infty$-norm over an arbitrary domain.
\begin{definition}
For any measurable set $\Omega\subset \bR^d$ and unit vector $\nu\in\bS^{d-1}$, define
\begin{align*}
\cR_{m,n}^{\Omega,\nu} (k,c,t):= \frac{\sup_{y\in\Omega}\left|(\nabla_\nu)^k u(y,t)\right|^{\frac{1}{k+1}}}{c^{\frac{k+n}{k+n+1}}((k+m)!)^{\frac{1}{k+1}}} \ , \qquad \theta_{m,n}^{\Omega,\nu} (i,j,\ell,t):= \cR_{m,n}^{\Omega,\nu}(i,c_\ell,t) \big/ \cR_{m,n}^{\Omega,\nu}(j,c_\ell,t)\ ,
\end{align*}
and for any $\alpha=1,\ldots,d$, define
\begin{align*}
\cR_{m,n}^{\Omega,\nu}[\alpha] (k,c,t):= \frac{\sup_{y\in\Omega}\left|(\nabla_\nu)^k u_\alpha(y,t)\right|^{\frac{1}{k+1}}}{c^{\frac{k+n}{k+n+1}}((k+m)!)^{\frac{1}{k+1}}} \ , \qquad \theta_{m,n}^{\Omega,\nu} [\alpha] (i,j,\ell,t):= \frac{\cR_{m,n}^{\Omega,\nu}[\alpha](i,c_\ell,t)}{\cR_{m,n}^{\Omega,\nu}[\alpha](j,c_\ell,t)}\ .
\end{align*}
In the following $\|\cdot\|$ always denotes $\|\cdot\|_{L^\infty(\Omega)}$ unless specified.
We divide all the indexes into sections at $\ell=\ell_0<\ell_1<\cdots<\ell_i<\ell_{i+1}<\cdots$ such that $\ell_{i+1}\ge\phi(\ell_i)$ for some increasing function $\phi(x)\ge \Lambda x$ with $\Lambda>1$ and each pair $(\ell_i, \ell_{i+q})$ satisfies the condition~\eqref{eq:AscDerCond} (with $\ell=\ell_i$ and $k=\ell_{i+q}$) for some fixed integer $q$. With the notation introduced in \eqref{eq:RealScaleNote}, at any temporal point $t<T^*$ and for each index $i$ we pick $p_i\in [\ell_i,\ell_{i+1}]$ such that
\begin{align*}
\cR_{m,n}^{\Omega,\nu} (p_i,c(\ell_i),t) = \max_{\ell_i\le j\le \ell_{i+1}} \cR_{m,n}^{\Omega,\nu} (j,c(\ell_i),t)
\end{align*}
while
\vspace{-0.1in}
\begin{align*}
\cR_{m,n}^{\Omega,\nu} (p_i,c(\ell_i),t) > \max_{\ell_i\le j< p_i} \cR_{m,n}^{\Omega,\nu} (j,c(\ell_i),t) \
\end{align*}
where $c(\ell_i):=c_{\ell_{i+1}}$ which is the constant defined by \eqref{eq:ParaAdjMaxPIndW} with $j=\ell_{i+1}$. (If such index $p_i$ does not exist in $[\ell_i,\ell_{i+1}]$ then we let $p_i=\ell_i$.)
Note that $m_i(t)$ may be variant in time, and we will always assume $p_i$ corresponds to the temporal point $t$ in $\cR\left(p_i,\cdot, t\right)$ without ambiguity.
Then, we divide the argument into two basic scenarios: (I) either there exists $k_i>\ell_{i+1}$ such that
\begin{align}\label{eq:NonDescAtLi}
\cR_{m,n}^{\Omega,\nu}(k_i,c(\ell_i),t)\ge \max_{p_i\le j\le k_i} \cR_{m,n}^{\Omega,\nu}(j,c(\ell_i),t)\ ,
\end{align}
(II) or
\vspace{-0.18in}
\begin{align}\label{eq:MaxDescAtLi}
\cR_{m,n}^{\Omega,\nu}(p_i,c(\ell_i),t) > \max_{j> p_i} \cR_{m,n}^{\Omega,\nu}(j,c(\ell_i),t)\ .
\end{align}
We say a section $[\ell_i,\ell_{i+1}]:=\left\{\cR_{m,n}^{\Omega,\nu}(j,c(\ell_i),t)\right\}_{j=\ell_i}^{\ell_{i+1}}$ is of Type-$\cA$ if it satisfies \eqref{eq:NonDescAtLi}, and say a section $[\ell_i,\ell_{i+1}]$ is of Type-$\cB$ if it satisfies \eqref{eq:MaxDescAtLi}. We call the union of sections $[\ell_i, \ell_j]:= \cup_{i\le r\le j-1}[\ell_r,\ell_{r+1}]$ a string if $j-i\ge q$ or the condition~\eqref{eq:AscDerCond} is satisfied with $\ell=\ell_i$ and $k=\ell_j$, and we call a string is of Type-$\cA$ if it consists of only Type-$\cA$ sections and of Type-$\cB$ if it contains at least one Type-$\cB$ section.
Moreover, a section $[\ell_i,\ell_{i+1}]$ is said to be of strong Type-$\cA$ if \eqref{eq:NonDescAtLi} is componentwise satisfied, i.e. for each $\alpha=1,\ldots,d$, there exists $k_i(\alpha)>\ell_{i+1}$ such that
\begin{align}\label{eq:NonDescAtLiCompw}
\cR_{m,n}^{\Omega,\nu}[\alpha] (k_i(\alpha),c(\ell_i),t)\ge \max_{p_i\le j\le k_i(\alpha)} \cR_{m,n}^{\Omega,\nu}[\alpha] (j,c(\ell_i),t)\
\end{align}
and it is said to be of strong Type-$\cB$ if \eqref{eq:MaxDescAtLi} is componentwise satisfied.
Strong Type-$\cA$ and Type-$\cB$ strings are defined in the same way.
We remark that the sections $[\ell_r,\ell_{r+1}]$'s within a string may have different pairs of parameters $(m_r,n_r)$.
\end{definition}

\begin{definition}
We say $\{\cR_{m,n}^{\Omega,\nu}(k,c_k,t)\}_{k=i_*}^\infty$ is \textit{homogeneous as $t\to T^*$} with respect to the time-variational bounds $c_1^* , c_2^*$ if there exists $t_*<T^*$ such that, for all $\ell, k\ge i_*$ and $t_*<t<T$,
\begin{align}\label{eq:RelHom}
c_1^*[k](t) \le \theta_{m,n}^{\Omega,\nu} (k,2k,\ell,t) \le c_2^*[k](t)
\end{align}
We say $\{\cR_{m,n}^{\Omega,\nu}(k,c_k,t)\}_{k=i_*}^\infty$ is \textit{non-homogeneous as $t\to T^*$} with respect to some predetermined $c_1^* , c_2^*$ if \eqref{eq:RelHom} is not satisfied.
Moreover, $\{\cR_{m,n}^{\Omega,\nu}(k,c_k,t)\}_{k=i_*}^\infty$ is said to be \textit{strong homogeneous as $t\to T^*$} if there exists $t_*<T^*$ such that \eqref{eq:RelHom} holds componentwise (with possibly different constants $c_1^*(\alpha) , c_2^*(\alpha)$).
\end{definition}

\begin{lemma}\label{le:Rk-transitivity}
Let $\Omega=\bR^d$ and $\nu\in\bS^{d-1}$. The following implications hold for all $t$ and $p,j\in \bN$:
\begin{align}
\cR_{m,n}^{\Omega,\nu}(k+j,c(\ell_i),t)\lesssim \cR_{m,n}^{\Omega,\nu}(k,c(\ell_i),t) \quad \Longrightarrow \quad \cR_{m,n}^{\Omega,\nu}(k,c(\ell_i),t)\lesssim \cR_{m,n}^{\Omega,\nu}(k-p,c(\ell_i),t)\ ,\label{eq:Rk-transBack}
\\
\cR_{m,n}^{\Omega,\nu}(k+j,c(\ell_i),t)\gtrsim \cR_{m,n}^{\Omega,\nu}(k,c(\ell_i),t) \quad \Longleftarrow \quad \cR_{m,n}^{\Omega,\nu}(k,c(\ell_i),t)\gtrsim \cR_{m,n}^{\Omega,\nu}(k-p,c(\ell_i),t)\ ,\label{eq:Rk-transFor}
\end{align}
i.e. monotonicity of $\cR(k,c,t)$ is transitive if $\Omega=\bR^d$. Moreover, such implications hold componentwise, i.e. monotonicity of each $\cR[\alpha](k,c,t)$ is transitive.
A particular consequence is that for fixed $m,n,k,\nu,c(\ell)$ and $t$, there exists a unique maximal extension of Type-$\cA$ section (or Type-$\cB$ section) $[\ell_i,\ell_{i+1}]$ which contains the index $k$. (An extreme scenario would be the whole $\{\cR(k,c,t)\}$ is of Type-$\cA$ or Type-$\cB$.)
\end{lemma}

\begin{proof}
We only prove \eqref{eq:Rk-transBack} (the proof for \eqref{eq:Rk-transFor} is the same). Note that $\cR_{m,n}^{\Omega,\nu}(k+j,c(\ell_i),t)\lesssim \cR_{m,n}^{\Omega,\nu}(k,c(\ell_i),t)$ indicates
\begin{align*}
\|(\nabla_\nu)^{k+j}u(t)\|^{\frac{1}{k+j+1}} &\lesssim \left(c(\ell_i)\right)^{\frac{j}{(k+n+1)(k+j+n+1)}} \frac{(k+j+m)^{1+\frac{m-1}{k+j+1}}} { (k+m)^{1+\frac{m-1}{k+1}}} \|(\nabla_\nu)^k u(t)\|^{\frac{1}{k+1}}
\\
&\lesssim \frac{(k+m_i)^{m_i/k^2}}{(2e/\varepsilon_i\eta)^{1/k} \cZ_{k+1}(t_0)^{1/k}} \|(\nabla_\nu)^ku(t)\|^{\frac{1}{k+1}} \lesssim \|(\nabla_\nu)^ku(t)\|^{\frac{1}{k+1}}
\end{align*}
assuming $\|(\nabla_\nu)^ku(t)\|\lesssim k^{C\cdot k^2}$.
Then it is proved by Lemma~\ref{le:GNIneq} with a contradictory argument that $\|(\nabla_\nu)^ku(t)\|^{\frac{1}{k+1}}\lesssim \|(\nabla_\nu)^{k-p}u(t)\|^{\frac{1}{k-p+1}}$ and the statement follows.
\end{proof}

\section{Time-Global Regularity with Dissipation Degree $\beta>1$ -- The Non-Homogeneous Case}\label{sec:MainThm}

In this section $\cR_{m,n}(k,c,t)$ always refers to $\cR_{m,n}^{\bR^d,\nu}(k,c,t)$ with a predetermined $\nu\in \bS^{d-1}$ unless otherwise specified.
The main result is the following:
\begin{theorem}\label{th:HypNSEReg}
Consider the $d$-dimensional Navier-Stokes system with Hyper-Dissipation ($d\ge 3$), i.e. the equations \eqref{eq:HypNSE1}-\eqref{eq:HypNSE3} in $\bR^d$. Fix an order of dissipation $\beta>1$ and $T>0$. Suppose $u_0\in L^\infty(\bR^d) \cap L^2(\bR^d)$ and for any $t<T$ there exist $\nu\in \bS^{d-1}$ and a finite range of indexes $k$ (the range depends on $\beta$ as well as on the fluctuations of the $L^\infty$-norm of $u$) such that for all component indexes $\alpha$, the \textit{homogeneity index} $\bar\theta_\nu[\alpha] (k,2k,t)$ satisfies
\begin{align}\label{eq:NonHomCond}
\bar\theta_\nu[\alpha] (k,2k,t) \le C_1^*[\beta](k,t) \approx_\eta \max\left\{ (T-t)^{-\frac{\beta-1}{2k+1}+\frac{(2\beta-1)(d-2)}{4(2k+1)^2}} \ ,\ \|(\nabla_\nu)^{2k} u(t)\|^{\frac{2(\beta-1)}{(2k+1)^2}-\frac{(2\beta-1)(d-2)}{2(2k+1)^3}}\right\}\
\end{align}
where $\bar\theta_\nu[\alpha] (k,2k,t) \approx_\eta \|(\nabla_\nu)^k u(t)^{(\alpha)}\|^{\frac{1}{k+1}}/\|(\nabla_\nu)^{2k} u(t)^{(\alpha)}\|^{\frac{1}{2k+1}}$
($\eta$ is the parameter featured in the descending chain arguments, cf.Theorem \ref{le:DescendDer}), and whose precise definition will be given later in this section (cf. Definition \ref{h-index}).
Then, the classical solution of \eqref{eq:HypNSE1}-\eqref{eq:HypNSE3} exists on $[0,T]$.
\end{theorem}

We start with several preliminary results. Similar to Theorem~\ref{th:MainThmVel} we have
\begin{theorem}\label{th:MainThmVelHyp}
Assume $u_0\in L^\infty(\bR^d)\cap L^p(\bR^d)$ and $f(\cdot,t)$ is divergence-free and real-analytic in the space variable with the analyticity radius at least $\delta_f$ for all $t\in[0,\infty)$, and the analytic extension $f+ig$ satisfies
\begin{align*}
\Gamma_\infty^k(t) &:= \sup_{s<t}\sup_{|y|<\delta_f} \left(\|D^kf(\cdot,y,s)\|_{L^\infty}+\|D^kg(\cdot,y,s)\|_{L^\infty}\right)<\infty\ ,
\\
\Gamma_p(t) &:= \sup_{s<t}\sup_{|y|<\delta_f} \left(\|f(\cdot,y,s)\|_{L^p}+\|g(\cdot,y,s)\|_{L^p}\right)<\infty\ .
\end{align*}
Fix an index $k\in\bN$ and $t_0>0$ and let
\begin{align}
T_*&=\min\left\{\left(C_1(M) 2^{k} \left(\|u_0\|_p + \Gamma_p(t_0)\right)^{k/(k+\frac{d}{p})} \left(\|D^ku_0\|_\infty + \Gamma_\infty^k(t_0)\right)^{\frac{d}{p}/(k+\frac{d}{p})} \right)^{-\frac{2\beta}{2\beta-1}}, \right. \notag
\\
&\qquad \left. \left(C_2(M) \left(\|u_0\|_p + \Gamma_p(T)\right)^{(k-1)/(k+\frac{d}{p})} \left(\|D^ku_0\|_\infty + \Gamma_\infty^k(T)\right)^{(1+\frac{d}{p})/(k+\frac{d}{p})} \right)^{-1} \right\} \label{eq:TimeLengthHyp}
\end{align}
where $C(M)$ is a constant only depending on $M$. Then there exists a solution
\begin{align*}
u\in C([0,T_*),L^p(\bR^d)^d) \cap C((0,T_*),C^\infty(\bR^d)^d)
\end{align*}
of \eqref{eq:HypNSE1}-\eqref{eq:HypNSE3} such that for every $t\in (0,T_*)$, $u$ is a restriction of an analytic function $u(x,y,t)+iv(x,y,t)$ in the region
\begin{align}\label{eq:AnalDomHyp}
\cD_t=: \left\{(x,y)\in\bC^d\ \big|\ |y|\le \min\{ct^{\frac{1}{2\beta}},\delta_f\}\right\}\ .
\end{align}
Moreover, $D^ju\in C([0,T_*),L^\infty(\bR^d)^d)$ for all $0\le j\le k$ and
\begin{align}
&\underset{t\in(0,T)}{\sup}\ \underset{y\in\cD_t}{\sup} \|u(\cdot,y,t)\|_{L^p} + \underset{t\in(0,T)}{\sup}\ \underset{y\in\cD_t}{\sup} \|v(\cdot,y,t)\|_{L^p}\le M\left(\|u_0\|_p + \Gamma_p(T)\right)\ ,
\\
&\underset{t\in(0,T)}{\sup}\ \underset{y\in\cD_t}{\sup} \|D^ku(\cdot,y,t)\|_{L^\infty} + \underset{t\in(0,T)}{\sup}\ \underset{y\in\cD_t}{\sup} \|D^kv(\cdot,y,t)\|_{L^\infty}\le M\left(\|D^ku_0\|_\infty + \Gamma_\infty^k(T)\right)\ .
\end{align}
\end{theorem}

\begin{proof}
Following the same procedure for constructing approximation sequence as in Theorem~\ref{th:MainThmVel}, we obtain the iteration formulas for \eqref{eq:HypNSE1}-\eqref{eq:HypNSE3} as follows:
\begin{align}
D^jU^{(n)}(x,t)& =G_t^{(\beta)}*D^ju_0 - \int_0^t G_{t-s}^{(\beta)}*\nabla D^j\left(U^{(n-1)}\otimes U^{(n-1)}\right) ds + \int_0^t G_{t-s}^{(\beta)}*\nabla D^j\left(V^{(n-1)}\otimes V^{(n-1)}\right) ds \notag
\\
&\quad -\int_0^t G_{t-s}^{(\beta)}*\nabla D^j\Pi^{(n-1)}ds + \int_0^t G_{t-s}^{(\beta)}*D^jF\ ds - \int_0^t G_{t-s}^{(\beta)}*\alpha\cdot\nabla D^jV^{(n)}ds\ , \label{eq:IterationHypDU}
\\
D^jV^{(n)}(x,t)& = - \int_0^t G_{t-s}^{(\beta)}* D^j\left(U^{(n-1)}\cdot\nabla\right)V^{(n-1)} ds - \int_0^t G_{t-s}^{(\beta)}* D^j\left(V^{(n-1)}\cdot\nabla\right)U^{(n-1)} ds \notag
\\
&\quad -\int_0^t G_{t-s}^{(\beta)}*\nabla D^j R^{(n-1)}ds + \int_0^t G_{t-s}^{(\beta)}* D^jG\ ds - \int_0^t G_{t-s}^{(\beta)}*\alpha\cdot\nabla D^jU^{(n)}ds \label{eq:IterationHypDV}
\end{align}
where $G_t^{(\beta)}$ denotes the fractional heat kernel of order $\beta$, and $U^{(n)}$, $V^{(n)}$, $\Pi^{(n)}$, $R^{(n)}$ are the same as those in Theorem~\ref{th:MainThmVel}. Similar argument leads to
\begin{align*}
\|D^kU^{(n)}\|_{L^\infty} &+ \|D^kV^{(n)}\|_{L^\infty} \lesssim \|D^ku_0\|_{L^\infty} + T^{1-\frac{1}{2\beta}} 2^k (L_{n-1})^{k/(k+\frac{d}{p})} \left(L_{n-1}'\right)^{(k+\frac{2d}{p})/(k+\frac{d}{p})}
\\
& + T\left(\underset{t<T}{\sup}\|D^kF\|_{L^\infty}+\underset{t<T}{\sup}\|D^kG\|_{L^\infty}\right) + |\alpha|t^{1-\frac{1}{2\beta}} \left(\|D^kU^{(n)}\|_{L^\infty} + \|D^kV^{(n)}\|_{L^\infty}\right)
\end{align*}
and
\begin{align*}
\|U^{(n)}\|_{L^p} + \|V^{(n)}\|_{L^p} &\lesssim \|u_0\|_{L^p} + T\ (L_{n-1})^{(2k-1+\frac{d}{p})/(k+\frac{d}{p})} \left(L_{n-1}'\right)^{(1+\frac{d}{p})/(k+\frac{d}{p})}
\\
& + T\left(\underset{t<T}{\sup}\|F\|_{L^p}+\underset{t<T}{\sup}\|G\|_{L^p}\right) + |\alpha|t^{1-\frac{1}{2\beta}} \left(\|U^{(n)}\|_{L^p}+\|V^{(n)}\|_{L^p}\right) \
\end{align*}
where $L_n$ and $L_n'$ are the same as in Theorem~\ref{th:MainThmVel}. If $|\alpha|t^{1-\frac{1}{2\beta}}\le \frac{1}{2}$ and if
$$T\le C\left(2^{k}\cdot M^2/(M-1) \cdot \left(\|u_0\|_2 + \Gamma_2(T)\right)^{k/(k+\frac{d}{p})} \cdot \left(\|D^ku_0\|_\infty + \Gamma_\infty^k(T)\right)^{\frac{d}{p}/(k+\frac{d}{p})} \right)^{-\frac{2\beta}{2\beta-1}}$$
and
$$T\le C\left(M^2/(M-1) \cdot \left(\|u_0\|_2 + \Gamma_2(T)\right)^{(k-1)/(k+\frac{d}{p})} \cdot \left(\|D^ku_0\|_\infty + \Gamma_\infty^k(T)\right)^{(1+\frac{d}{p})/(k+\frac{d}{p})} \right)^{-1}$$
Then the statement follows by a standard converging argument.
\end{proof}

The next result is stated for the Navier-Stokes case first.

\begin{theorem}\label{le:AscendDerMod}
Suppose $\ell$ and $k$ are sufficiently large and
\begin{align}\label{eq:AscDerOrdMod}
\|D^ju_0\|^{\frac{1}{j+1}} \le \cM_{j,k}\cdot \|D^ku_0\|^{\frac{1}{k+1}}\ , \qquad \forall\ 2\ell\le j\le k
\end{align}
where the constants $\{\cM_{j,k}\}$ and $\|u_0\|$ satisfy
\begin{align}\label{eq:BkCondM}
\sum_{\ell \le i\le j-\ell} \binom j i  \cM_{i,k}^{i+1}\ \cM_{j-i,k}^{j-i+1} \lesssim 1\ , \qquad \forall\ 2\ell\le j\le k
\end{align}
and
\begin{align}\label{eq:BkCondS}
\|u_0\|_2 \sum_{0\le i\le \ell} \binom j i \cM_{\ell,k}^{\frac{(\ell+1)(i+d/2)}{\ell+d/2}} \cM_{j-i,k}^{j-i+1} \left(k!\|u_0\|^k\right)^{\frac{(d/2-1)(\ell-i)}{\ell+d/2}\frac{1}{k+1}}\lesssim 1\ , \qquad \forall\ 2\ell\le j\le k
\end{align}
If $T^{1/2}\lesssim\|D^ku_0\|^{-\frac{1}{k+1}}$, then for any $2\ell\le j \le k$, the complex solution of \eqref{eq:NSE1}-\eqref{eq:NSE3} has the upper bound:
\begin{align}\label{eq:AscDerUpBddMod}
&\underset{t\in(0,T)}{\sup}\ \underset{y\in\cD_t}{\sup} \|D^ju(\cdot,y,t)\|_{L^\infty} + \underset{t\in(0,T)}{\sup}\ \underset{y\in\cD_t}{\sup} \|D^jv(\cdot,y,t)\|_{L^\infty} \lesssim \|D^ju_0\| +  \|D^ku_0\|^{\frac{j+1}{k+1}}
\end{align}
where $\cD_t$ is given by \eqref{eq:AnalDom}.
\end{theorem}

\begin{proof}
Similar to the proof of Theorem~\ref{le:AscendDer}, we have estimates for the nonlinear terms as follows:
\begin{align*}
\|D^jU_n(t)\| & \lesssim \|D^ju_0\| + t^{1/2} \left(\sum_{\ell \le i\le j-\ell} \binom j i \sup_{s<t}\|D^iU_{n-1}(s)\| \sup_{s<t}\|D^{j-i}U_{n-1}(s)\| \right.
\\
& \left.+ 2\sum_{0\le i\le \ell} \binom j i \left(\sup_{s<t}\|U_{n-1}(s)\|_2\right)^{\frac{\ell-i}{\ell+d/2}} \left(\sup_{s<t}\|D^\ell U_{n-1}(s)\|\right)^{\frac{i+d/2}{\ell+d/2}} \sup_{s<t}\|D^{j-i}U_{n-1}(s)\|\right)
\\
&:= \|D^ju_0\| + t^{1/2} \left(I+2J\right)
\end{align*}
By the induction hypothesis and the assumptions~\eqref{eq:AscDerOrdMod}, if $\ell$, $k$ and $\{\cM_{j,k}\}$ are chosen as in \eqref{eq:BkCondM} and \eqref{eq:BkCondS},
\begin{align*}
I &\lesssim \sum_{\ell \le i\le j-\ell} \binom j i\  \cM_{i,k}^{i+1} \|D^k u_0\|^{\frac{i+1}{k+1}}\  \cM_{j-i,k}^{j-i+1} \|D^k u_0\|^{\frac{j-i+1}{k+1}}
\lesssim \|D^k u_0\|^{\frac{j+2}{k+1}}
\end{align*}
Without loss of generality we assume $\|D^ku_0\|\lesssim k!\|u_0\|^{k+1}$, then
\begin{align*}
J &\lesssim \sum_{0\le i\le \ell} \binom j i\  \|u_0\|_2^{\frac{\ell-i}{\ell+d/2}} \left(\cM_{\ell,k}^{\ell+1} \|D^k u_0\|^{\frac{\ell+1}{k+1}}\right)^{\frac{i+d/2}{\ell+d/2}}  \cM_{j-i,k}^{j-i+1}  \|D^k u_0\|^{\frac{j-i+1}{k+1}}
\\
&\lesssim \|u_0\|_2\|D^k u_0\|^{\frac{j+2}{k+1}} \sum_{0\le i\le \ell} \binom j i \cM_{\ell,k}^{\frac{(\ell+1)(i+d/2)}{\ell+d/2}} \cM_{j-i,k}^{j-i+1} \|D^k u_0\|^{\frac{(d/2-1)(\ell-i)}{\ell+d/2}\frac{1}{k+1}}
\\
&\lesssim \|u_0\|_2\|D^k u_0\|^{\frac{j+2}{k+1}} \sum_{0\le i\le \ell} \binom j i \cM_{\ell,k}^{\frac{(\ell+1)(i+d/2)}{\ell+d/2}} \cM_{j-i,k}^{j-i+1}  \left(k!\|u_0\|^k\right)^{\frac{(d/2-1)(\ell-i)}{\ell+d/2}\frac{1}{k+1}}
\end{align*}
To sum up, if $T^{1/2}\lesssim\|D^ku_0\|^{-\frac{1}{k+1}}$, then for all $n$
\begin{align*}
\sup_{s<t}\|D^jU_n(s)\| \lesssim \|D^ju_0\| + T^{1/2}\|D^k u_0\|^{\frac{j+2}{k+1}} \lesssim \|D^ju_0\| + \|D^k u_0\|^{\frac{j+1}{k+1}}
\end{align*}
and similarly, for all $n$,
\begin{align*}
\sup_{s<t}\|D^jV_n(s)\| \lesssim \|D^k u_0\|^{\frac{j+1}{k+1}}
\end{align*}
and a standard converging argument proves \eqref{eq:AscDerUpBddMod}.
\end{proof}

\begin{corollary}\label{le:DerOrdConfigHyp}
Theorem~\ref{le:DescendDer}, Lemma~\ref{le:ScaleBound} and Corollary~\ref{cor:DerOrdConfig} still hold if $u(t)$ is the solution to \eqref{eq:HypNSE1}-\eqref{eq:HypNSE3}.
\end{corollary}

The hyper-dissipative version of the previous theorem is as follows.

\begin{theorem}\label{le:AscendDerHyp}
Suppose $\ell$, $k$ and $\{\cM_{j,k}\}$ satisfy \eqref{eq:AscDerOrdMod} with $\{\cM_{j,k}\}$ and $\|u_0\|$ satisfying
\begin{align}\label{eq:BkCondMHyp}
\sum_{\ell \le i\le j-\ell} \binom j i  \cM_{i,k}^{i+1}\ \cM_{j-i,k}^{j-i+1} \lesssim \phi(j,k)\ , \qquad \forall\ 2\ell\le j\le k
\end{align}
and
\begin{align}\label{eq:BkCondSHyp}
\|u_0\|_2 \sum_{0\le i\le \ell} \binom j i \cM_{\ell,k}^{\frac{(\ell+1)(i+d/2)}{\ell+d/2}} \cM_{j-i,k}^{j-i+1} \left(k!\|u_0\|^k\right)^{\frac{(d/2-1)(\ell-i)}{\ell+d/2}\frac{1}{k+1}}\lesssim \psi(j,k)\ , \qquad \forall\ 2\ell\le j\le k\ .
\end{align}
If $T\lesssim \left(\phi(j,k)+\psi(j,k)\right)^{-\frac{2\beta}{2\beta-1}} \|D^ku_0\|^{-\frac{2\beta}{(2\beta-1)(k+1)}}$, then for any $\ell\le j \le k$, the complex solution of \eqref{eq:HypNSE1}-\eqref{eq:HypNSE3} has the upper bound:
\begin{align}\label{eq:AscDerUpBddHyp}
&\underset{t\in(0,T)}{\sup}\ \underset{y\in\cD_t}{\sup} \|D^ju(\cdot,y,t)\|_{L^\infty} + \underset{t\in(0,T)}{\sup}\ \underset{y\in\cD_t}{\sup} \|D^jv(\cdot,y,t)\|_{L^\infty} \lesssim \|D^ju_0\| +  \|D^ku_0\|^{\frac{j+1}{k+1}}
\end{align}
where $\cD_t$ is given by \eqref{eq:AnalDomHyp}. In particular, if $\phi(j,k), \psi(j,k)\lesssim A^k$ with any constant $A>1$, then the descending chain described by \eqref{eq:DescDerOrd} in Theorem~\ref{le:DescendDer} forms at a time after $T$ with some constant $c\approx \eta \cdot A^{-d/2+1}$.
\end{theorem}

\begin{proof}
Similar to the proof of Theorem~\ref{le:AscendDer}, we have estimates for the nonlinear terms as follows:
\begin{align*}
\|D^jU_n(t)\| &\lesssim \|D^ju_0\| + t^{1-\frac{1}{2\beta}} \sum_{i=0}^j \binom j i \sup_{s<t}\|D^iU_{n-1}(s)\| \sup_{s<t}\|D^{j-i}U_{n-1}(s)\|
\\
&\lesssim \|D^ju_0\| + t^{1-\frac{1}{2\beta}} \left(I+2J\right) \lesssim \|D^ju_0\| + t^{1-\frac{1}{2\beta}} (\phi(j,k)+\psi(j,k))\|D^k u_0\|^{\frac{j+2}{k+1}}
\end{align*}
Similar argument leads to \eqref{eq:AscDerUpBddHyp} with $T^{1-\frac{1}{2\beta}}\lesssim (\phi(j,k)+\psi(j,k))^{-1} \|D^ku_0\|^{-\frac{1}{k+1}}$.
\end{proof}

\bigskip

In order to indicate the idea behind the proof of Theorem~\ref{th:HypNSEReg}, we first make some heuristic arguments about combining Theorem~\ref{le:DescendDer} and Theorem~\ref{le:AscendDerHyp} alternately in time as well as deriving some sharp estimates of $\cR_{m,n}(k,c,t)$ to achieve the requirements \eqref{eq:BkCondMHyp} and \eqref{eq:BkCondSHyp}. Consider the following division of the indexes:
\begin{align*}
k_{p-1}\le \Lambda \cdot k_p\ , \qquad m_p:= \mu_p\cdot k_p\ ,\qquad n_p\le \xi_p\cdot k_p\ , \qquad c_{j,p}:= \left( \frac{(j+m_p)^{m_p/j}}{(2e/\varepsilon_p\eta)\ \cZ_j(t_0)} \right)^{(j-n_p+1)^2/j}
\end{align*}
and define $\lambda_{j,p}=j/k_p$. Then, for any $k_{p-1}\le j\le k_p$
\begin{align*}
\frac{2e}{\eta} \frac{c_{j,p}^{\frac{j}{(j-n_p+1)^2}}\ \cZ_j(t_0)} {(j+m_p)^{m_p/j}} \exp\left(\frac{2e}{\eta} \frac{c_{j,p}^{\frac{j}{(j-n_p+1)^2}}\ \cZ_j(t_0)} {(j+m_p)^{m_p/j}}\right) &\lesssim \varepsilon_p e^{\varepsilon_p} \lesssim \varepsilon_0
\end{align*}
so the condition~\eqref{eq:ParaAdjMaxPIndW} is satisfied and Lemma~\ref{le:ScaleBound} is applicable for such choices. Then, by Lemma~\ref{le:ScaleBound} either Case~(I)$^*$ or Case~(II)$^*$ occurs. If Case~(II)$^*$ occurs, then it is trivial that $T^*$ is not a blow-up time. If Case~(I)$^*$ occurs, then with the above choices for $c_{j,p}$ and $n_p$ we compute that for any $k_{p-1}\le i,j\le k_p$
\begin{align*}
\cB_{i,j}^{(p)} &\lesssim \prod_{r=i}^j \left( \left(\frac{(r+m_p)^{m_p/r}}{(2e/\varepsilon_p\eta)\ \cZ_r(t_0)} \right)^{(r-n_p+1)^2/r} \right)^{-\frac{1}{(r-n_p+1)(r-n_p+2)}}
\\
&\lesssim k_p^{-\sum_{r=i}^j m_p/r^2} \left(\frac{\varepsilon_p\eta}{2e}\right)^{-\sum_{r=i}^j r^{-1}} \prod_{r=i}^j (\lambda_r+\lambda_{m_p})^{-m_p/r^2} \prod_{r=i}^j \cZ_r(t_0)^{1/r}
\\
&\lesssim k_p^{-m_p\left(i^{-1}- j^{-1}\right)} \left(\frac{\varepsilon_p\eta}{2e}\right)^{-\ln(j/i)} \exp\left(-\frac{m_p}{r^2} \sum_{r=i}^j \ln(\lambda_r+\lambda_{m_p})\right) \prod_{r=i}^j\ \cZ_r(t_0)^{1/r}
\end{align*}
For $\Lambda, \mu_p\approx 1$, the above estimate is reduced to
\begin{align*}
\cB_{i,j}^{(p)} &\lesssim k_p^{-m_p\left(i^{-1}- j^{-1}\right)} \left(\frac{2e}{\varepsilon_p\eta}\right)^{\ln(j/i)} \prod_{r=i}^j\ \cZ_r(t_0)^{1/r}
\end{align*}
Then, with sufficiently large $k_p$'s (compared to $\|u_0\|$ and $(\varepsilon_p\eta)^{-1}$) and $\Lambda, \mu_p\approx 1$,
\begin{align*}
\frac{\cB_{i,j}^{(p)}((i+m_p)!)^{\frac{1}{i+1}}}{((j+m_p)!)^{\frac{1}{j+1}}}  &\lesssim \frac{k_p^{-m_p\left(i^{-1}- j^{-1}\right)} \displaystyle\left(\frac{\varepsilon_p\eta}{2e}\right)^{\ln(i/j)} \prod_{r=i}^j\ \cZ_r(t_0)^{1/r} } {e^{m_p\left(i^{-1}- j^{-1}\right)} \displaystyle (j+m_p)^{1+\frac{m_p}{j}}/ (i+m_p)^{1+\frac{m_p}{i}}}
\lesssim \left(\frac{2e}{\varepsilon_p\eta}\right)^{\ln(j/i)} \prod_{r=i}^j\ \cZ_r(t_0)^{1/r}
\end{align*}
We set up $\cR(k,c,t)$ and $\cC(k,c,t)$ in \eqref{eq:RealScaleNote} with additional indexes $m$ and $n$ as well as the lacunary arguments above, expecting that application of Lemma~\ref{le:ScaleBound} for multiple times with varying indexes $m$ and $n$ can significantly reduce the ratios between $\|D^iu(s)\|^{\frac{1}{i+1}}$ and $\|D^ju(s)\|^{\frac{1}{j+1}}$ in the descending chains described in the lemma and the corollary; more precisely with proper choices of the pairs $(m_i,n_i)$ the descending chains with different indexes $(m_i,n_i)$ form simultaneously and it is possible to joint the portions with different indexes so that for any $\frac{k_1}{2}\le i\le k_1$ and $k_{\theta} \le j\le 2k_{\theta}$,
\begin{align*}
\|D^iu(s)\|^{\frac{1}{i+1}} &\lesssim \frac{\cB_{i,k_1}^{(1)}((i+m_1)!)^{\frac{1}{i+1}}}{((k_1+m_1)!)^{\frac{1}{k_1+1}}} \cdot \frac{\cB_{k_1,k_2}^{(2)}((k_1+m_2)!)^{\frac{1}{k_1+1}}}{((k_2+m_2)!)^{\frac{1}{k_2+1}}} \cdots
\\
&\quad \cdots \frac{\cB_{k_{\theta-1},k_{\theta}}^{(\theta)} ((k_{\theta-1}+m_{\theta})!)^{\frac{1}{k_{\theta-1}+1}}}{((k_{\theta}+m_{\theta})!)^{\frac{1}{k_{\theta}+1}}} \cdot \frac{\cB_{k_{\theta},j}^{(\theta+1)}((k_{\theta}+m_{\theta+1})!)^{\frac{1}{k_{\theta}+1}}} {((j+m_{\theta+1})!)^{\frac{1}{j+1}}}\ \|D^ju(s)\|^{\frac{1}{j+1}}
\\
\lesssim & \left(\frac{\varepsilon_1\eta}{2e}\right)^{\ln(i/k_1)} \left(\frac{\varepsilon_2\eta}{2e}\right)^{\ln(k_1/k_2)} \cdots \left(\frac{\varepsilon_\theta\eta}{2e}\right)^{\ln(k_{\theta}/j)} \prod_{r=i}^j\ \cZ_r(t_0)^{1/r} \|D^ju(s)\|^{\frac{1}{j+1}} =: k^{\alpha(i,j)} \|D^ju(s)\|^{\frac{1}{j+1}}
\end{align*}
while
\begin{align*}
\sum_{\ell \le i\le j-\ell} \binom j i  \cM_{i,k}^{i+1}\ \cM_{j-i,k}^{j-i+1}
&\lesssim \sum_{\ell \le i\le j-\ell} \binom j i \left(k^{\alpha(i,k)}\right)^{i+1} \left(k^{\alpha(j-i,k)}\right)^{j-i+1}
\lesssim j\cdot 2^j \cdot k^{\alpha(j/2,k)\cdot j/2}
\end{align*}
so, for sufficiently large $k$, \eqref{eq:BkCondM} is satisfied, and similarly, if $k$ is sufficiently large and $\|u_0\|\lesssim k$, then
\begin{align*}
\sum_{0\le i\le \ell} & \binom j i \cM_{\ell,k}^{\frac{(\ell+1)(i+d/2)}{\ell+d/2}} \cM_{j-i,k}^{j-i+1} \left(k!\|u_0\|^k\right)^{\frac{(d/2-1)(\ell-i)}{\ell+d/2}\frac{1}{k+1}}
\\
&\lesssim \sum_{0\le i\le \ell} \binom j i \left(k^{\alpha(\ell,k)}\right)^{\frac{(\ell+1)(i+d/2)}{\ell+d/2}} \left(k^{\alpha(j-i,k)}\right)^{j-i+1} \left(k!\|u_0\|^k\right)^{\frac{(d/2-1)(\ell-i)}{\ell+d/2}\frac{1}{k+1}} \lesssim k^{\tilde{\beta}}
\end{align*}
so \eqref{eq:BkCondS} is satisfied.

At each joint (between $(m_i,n_i)$ and $(m_{i+1},n_{i+1})$) we need the following implication
\begin{align}
&\cR_{m_{i+1},n_{i+1}}(k,c_{k,i+1},t) \ge \cR_{m_{i+1},n_{i+1}}(k+1,c_{k,i+1},t) \notag
\\
&\Longrightarrow \qquad \cR_{m_i,n_i}(k,c_{k,i},t) \ge \cR_{m_i,n_i}(k+1,c_{k,i},t) \label{eq:JointMonoImp}
\end{align}
and an equivalent requirement would be
\begin{align*}
\cR_{m_i,n_i}(k,c_{k,i},t) \approx \cR_{m_{i+1},n_{i+1}}(k,c_{k,i+1},t)
\end{align*}
that is
\begin{align*}
c_{k,i+1}^{\frac{k-n_{i+1}}{k-n_{i+1}+1}} \left((k+m_{i+1})!\right)^{\frac{1}{k+1}} \approx c_{k,i}^{\frac{k-n_i}{k-n_i+1}} \left((k+m_i)!\right)^{\frac{1}{k+1}}
\end{align*}
Plugging in $c_{k,i}$'s yields
\begin{align*}
\frac{\displaystyle\left((\varepsilon_i\eta/2e)\cdot (k+m_i)^{m_i/k}/ \cZ_k(t_0) \right)^{(k-n_i+1)^2/k} } {\displaystyle\left((\varepsilon_{i+1}\eta/2e)\cdot (k+m_{i+1})^{m_{i+1}/k}/ \cZ_k(t_0)\right)^{(k-n_{i+1}+1)^2/k}} \cdot \frac{\left((k+m_i)!\right)^{\frac{1}{k+1}}}{\left((k+m_{i+1})!\right)^{\frac{1}{k+1}}} \approx 1
\end{align*}
Let $k_i=\Lambda^i\cdot k_0$ and $n_i=\xi_i\cdot k_i=\xi_i \Lambda^i k_0$. Then $\Delta n_i= n_{i+1}- n_i\approx (\Lambda-1) n_i$ (assuming all $\xi_i>1/2$ and $\Lambda$ is close to 1) and based on the requirement above we need for each $i$
\begin{align*}
&\left(\frac{\varepsilon_i\eta}{2e}\right)^{2(1-\xi_i)\Delta n_i} \frac{(k+m_i)^{m_i(1-\xi_i)^2}}{(k+m_{i+1})^{m_{i+1}(1-\Lambda \xi_{i+1})^2}}\ \frac{(k+m_i)^{1+\frac{m_i}{k}}}{(k+m_{i+1})^{1+\frac{m_{i+1}}{k}}} \gtrsim \cZ_k(t_0)^{2(1-\xi_i)\Delta n_i} \ .
\end{align*}
If $m_i\approx \mu_i\cdot k_i$ with $\mu_i\approx1$, then the above relation is simplified as
\begin{align*}
& \left(\frac{\varepsilon_i\eta}{2e}\right) \frac{(k+m_i)^{m_i(1-\xi_i)/2\Delta n_i}}{(k+m_{i+1})^{m_{i+1}(1-\Lambda \xi_{i+1})^2/2(1-\xi_i)\Delta n_i }} \gtrsim \|D^{k}u_0\|^{\frac{d/2-1}{(k+1)(k+d/2)}}
\\
\Longrightarrow \quad & \left(\frac{\varepsilon_i\eta}{2e}\right) \left((1+\mu_i)k\right)^{\frac{\mu_i\left((1-\xi_i)^2 - \Lambda (1-\Lambda \xi_{i+1})^2\right)}{2\xi_i(1-\xi_i)(\Lambda-1)}} \gtrsim \|D^{k}u_0\|^{\frac{d/2-1}{(k+1)(k+d/2)}} \ .
\end{align*}
Recall that $k_i\le k\le k_{i+1}$ and we are particularly interested in the above requirement at each $k_i$. Unless there are significant difference between $\|D^{k_i}u_0\|$ and $\|D^{k_{i+1}}u_0\|$ (e.g. the non-homogeneous case) we assume $\mu_i\approx \mu_{i+1}\approx \mu$ and $\xi_i\approx \xi_{i+1}\approx \xi$.
Then, if $\xi$ and $\Lambda$ are chosen to satisfy $(1-\xi)^2 - \Lambda (1-\Lambda \xi)^2\approx 1$ (e.g. $\xi=1/2$ and $\Lambda=2$), the above relation can be reduced to
\begin{align*}
\|D^{k_i}u_0\|^{\frac{d/2-1}{(k_i+1)(k_i+d/2)}} \lesssim k_i^{\frac{\mu\left((1-\xi)^2 - \Lambda (1-\Lambda \xi)^2\right)}{2\xi(1-\xi)(\Lambda-1)}} \qquad \Longrightarrow \qquad \|D^{k_i}u_0\| \lesssim k_i^{C\cdot k_i^2}
\end{align*}
with some constant $C\approx 1$.

\medskip

Note that at any temporal point $t$ the above requirement must hold for large $k_i$ as Theorem~\ref{th:LinftyIVP} guarantees 
$\|D^ku(t)\|\lesssim k! \|u_0\|^{k+1}$ within a period after each $t_0$.

\medskip

Moreover, the recursive arguments are valid in Lemma~\ref{le:ScaleBound} and Corollary~\ref{cor:DerOrdConfig} when $c_{r+1, i}\le c_{r,i}$, so the following is also required:
\begin{align*}
\frac{c_{r+1, i}}{c_{r,i}} &= \left( \frac{(r+1+m_i)^{m_i/(r+1)}}{(2e/\varepsilon_i\eta)\ \cZ_{r+1}(t_0)} \right)^{(r-n_i+2)^2/(r+1)} \bigg/ \left( \frac{(r+m_i)^{m_i/r}}{(2e/\varepsilon_i\eta)\ \cZ_r(t_0)} \right)^{(r-n_i+1)^2/r}
\\
&= \left(\varepsilon_i\eta/2e\right)^{\frac{(r-n_i+2)^2}{r+1} - \frac{(r-n_i+1)^2}{r}} \left(r+1+m_i\right)^{\frac{m_i(r-n_i+2)^2}{(r+1)^2} - \frac{m_i(r-n_i+1)^2}{r^2}}
\\
&\qquad\qquad\times  \left(\frac{r+1+m_i}{r+m_i}\right)^{\frac{m_i(r-n_i+1)^2}{r^2}} \cZ_r(t_0)^{\frac{(r-n_i+1)^2}{r}} \big/ \cZ_{r+1}(t_0)^{\frac{(r-n_i+2)^2}{r+1}}
\\
&\lesssim \left(\varepsilon_i\eta/2e\right)^{1 - \frac{(n_i-1)^2}{r(r+1)}} \left(r+1+m_i\right)^{\frac{m_i(n_i-1)}{r(r+1)}} \exp\left(\frac{m_i}{r+m_i}\left(1-\frac{n_i-1}{r}\right)^2\right)
\\
&\qquad\qquad\times  \left(\cZ_r(t_0)\big/ \cZ_{r+1}(t_0)\right)^{\frac{(r-n_i+1)^2}{r}} \cZ_{r+1}(t_0)^{-1 + \frac{(n_i-1)^2}{r(r+1)}}
\end{align*}
Assuming $\|D^{r}u_0\|^{\frac{1}{r+1}} / \|D^{r+1}u_0\|^{\frac{1}{r+2}}\approx 1$ (Recall Lemma~\ref{le:Rk-transitivity}), it suffices to impose
\begin{align*}
\cZ_{r+1}(t_0)^{1 - \frac{(n_i-1)^2}{r(r+1)}} \gtrsim \left(r+1+m_i\right)^{\frac{m_i(n_i-1)}{r(r+1)}}
\end{align*}
Thus, under all the above conditions,
\begin{align}
\cM_{i,j} &\lesssim \frac{\cB_{i,k_1}^{(0)}((i+m_0)!)^{\frac{1}{i+1}}}{((k_1+m_0)!)^{\frac{1}{k_1+1}}} \cdot \frac{\cB_{k_1,k_2}^{(1)}((k_1+m_1)!)^{\frac{1}{k_1+1}}}{((k_2+m_1)!)^{\frac{1}{k_2+1}}} \cdots \frac{\cB_{k_{\theta},j}^{(\theta)}((k_{\theta}+m_{\theta})!)^{\frac{1}{k_{\theta}+1}}} {((j+m_{\theta})!)^{\frac{1}{j+1}}} \notag
\\
&\lesssim \left(2e/\varepsilon \eta \right)^{\ln(j/i)}\ \prod_{r=i}^j\ \|D^{r}u_0\|^{\frac{d/2-1}{r(r+1)(r+d/2)}} \label{eq:PureCmpctRatio}
\end{align}
To guarantee that the analyticity radius in Theorem~\ref{le:AscendDerHyp} equals the natural scale of sparseness in Theorem~\ref{th:LerayZalpha}, 
\begin{align*}
r_g:=\left(\phi(j,k)+\psi(j,k)\right)^{-\frac{2\beta}{2\beta-1}} \|D^ku_0\|^{-\frac{2\beta}{(2\beta-1)(k+1)}} \approx r^*:= c(\|u_0\|_2)\ \|D^k u(t)\|^{-\frac{1}{k+d/2}}.
\end{align*}
In such scenario, \eqref{eq:BkCondMHyp} requires that at almost all $k$-levels
\begin{align*}
\binom k {k/2} \cM_{k/2,k}^{k/2+1} \cdot \cM_{k/2,k}^{k/2+1} \lesssim \phi(k)
\end{align*}
which (essentially) reads  $\cM_{k/2,k}\lesssim \phi(k)^{\frac{1}{k+2}}$. This can be proved by way of contradiction.  Assume the opposite, that is $\cM_{k/2,k}\gtrsim \phi(k)^{\frac{1}{k+2}}$ for almost all $k$; then for sufficiently large $j/i$,
\begin{align*}
\cM_{i,j} = \cM_{i,2i}\cdot \cM_{2i,4i} \cdots \cM_{j/2,j} \gtrsim \phi(i)^{\frac{\ln(j/i)/\ln 2}{i+2}}
\end{align*}
which violates  $\cM_{i,j}\lesssim \phi(i,j)^{\frac{1}{j+2}}\lesssim \phi(i)^{\frac{\ln(j/i)/\ln 2}{i+2}}$. Notice that the natural upper bound for $\cM_{i,j}$, as indicated by \eqref{eq:PureCmpctRatio}, is comparable to
\begin{align*}
\left(\frac{2e}{\varepsilon \eta}\right)^{\ln(j/i)} \prod_{r=i}^j\ \|D^{r}u_0\|^{\frac{d/2-1}{r(r+1)(r+d/2)}} &\lesssim \left(\frac{2e}{\varepsilon \eta}\right)^{\ln(j/i)} \prod_{r=i}^j\ (T-t)^{-\frac{d/2-1}{2r(r+d/2)}}
\\
&\lesssim \left(\frac{2e}{\varepsilon \eta}\right)^{\ln(j/i)} (T-t)^{-(d/2-1)(i^{-1}-j^{-1})}.
\end{align*}
As we will see later, a smaller upper bound for $\cM_{i,j}$ is possible if \eqref{eq:NonHomCond} is satisfied.
In the above scenario, to make \eqref{eq:BkCondSHyp} hold, it suffices to require $k/\ell$ to be sufficiently large. Thus, Theorem~\ref{le:AscendDerHyp} is applicable. A rigorous argument will be provided in the proof of Theorem~\ref{th:HypNSEReg}.

\bigskip

\begin{lemma}\label{le:DescendDerRefn}
Let $u$ be a Leray solution of \eqref{eq:HypNSE1}-\eqref{eq:HypNSE3} initiated at $u_0$. Suppose $\ell$ is sufficiently large such that $\|u_0\|\lesssim (1+\epsilon)^{\ell}$. For a fixed $k\ge\ell$ and $n, m\in\bZ$, suppose that
\begin{align}\label{eq:DescDerOrdRefn}
&\cR_{m_p,n_p}[\alpha](k_p,c_p,t_0) \ge \cR_{m_p,n_p}[\alpha](j,c_p,t_0) \ , \qquad \forall  k_p\le j <k_{p+1}
\end{align}
for suitable constants $c_p=c_p(k)$ and pairs $(m_p,n_p)$ which satisfy
\begin{align}\label{eq:JointCompact}
\left(k_p+m_p\right)^{\frac{m_pn_p}{k_p^2-n_p^2}} \lesssim \cZ_{k_p}(t_0) \lesssim \left(k_p+m_p\right)^{\frac{\left((1-\xi)^2 - \Lambda (1-\Lambda \xi)^2\right)m_p}{2\xi(1-\xi)(\Lambda-1)k_p}}\ , \quad \textrm{for all } p
\end{align}
where $\xi,\Lambda$ are constants defined as above and
\begin{align}\label{eq:ParaAdjMaxPRefn}
\lambda h^* + \left(1+\frac{2e}{\eta}\frac{c_p^{\frac{k_p}{(k_p-n_p+1)^2}}\ \cZ_{k_p}(t_0)}{(k_p+m_p)^{m_p/k_p}} \exp\left(\frac{2e}{\eta}\frac{c_p^{\frac{k_p}{(k_p-n_p+1)^2}}\ \cZ_{k_p}(t_0)}{(k_p+m_p)^{m_p/k_p}} \right)\right) (1-h^*) \le \mu
\end{align}
where $h^*=\frac{2}{\pi} \arcsin \frac{1-\delta^{2/d}}{1+\delta^{2/d}}$, $(1+\eta)^d=\frac{\delta(1+\lambda)+1}{2}$ and $\mu$ is a positive constant. Then there exist $t_*, T_*>t_0$ and a constant $\mu_*$ such that
\begin{align*}
\|D^ku^{(\alpha)}(s)\|\le \mu_* \|D^ku^{(\alpha)}_0\| \ , \qquad \forall\ t_*\le s\le T_*\ .
\end{align*}
Here $\mu$ and $\mu_*$ are determined in the same way as described in Theorem~\ref{le:DescendDer}.
\end{lemma}

\begin{proof}
Similar to the proof of Theorem~\ref{le:DescendDer}, we divide the expansion of $D^ku(t)$ at some index $k_*$ such that
\begin{align*}
\sup_{t_0<s<t_0+T_*}\! \|D^{j}u(s)\| \le M\|D^{j}u_0\|\ , \qquad \forall\ k\le j\le k_*\ .
\end{align*}
We first consider the case in which the order of `the tail of \eqref{eq:DescDerOrdRefn} after $k_*$' continues for all $s$ up to $t_0+T_*$, that is assuming, for any $t_0<s<t_0+T_*$,
\begin{align}\label{eq:DescDerOrdAlls}
\cR_{m_p,n_p}[\alpha](k_p,c_p,s) \ge \cR_{m_p,n_p}[\alpha](j,c_p,s) \ , \qquad \forall\ j\ge k_p^*\ .
\end{align}
Suppose there exists an $t<t_0+T_*$ such that $\|D^ku^{(\alpha)}(t)\|> \mu_* \|D^ku^{(\alpha)}_0\|$, then the argument before the theorem with condition~\eqref{eq:DescDerOrdRefn} leads to
\begin{align*}
\sup_{z\in B_{r_t}(x_0,0)} \left|D^ku(t)(z)\right|  &\le \left(\sum_{0\le i\le k_1} + \sum_{k_1<i\le k_2} + \cdots + \sum_{k_{p-1}<i\le k_p} + \sum_{i>k_p}\right)\frac{\left|D^{k+i}u(x_0,s))\right|}{i!} |z-x_0|^i
\\
&\le M \|D^{k}u_0\| \sum_{0\le i\le k_1} \frac{(k+m+i)!}{(k+m)!\ i!} \left(\frac{c_1^{\frac{1}{k+n+1} \frac{k+i+1}{k+n+i+1}}\ \cZ_k(t_0) } {\mu_*^{\frac{1}{k+d/2}}(\eta/2)\left((k+m)!\right)^{\frac{1}{k+1}}}\right)^{i}
\\
& \quad \cdots + M \|D^{k}u_0\| \sum_{i> k_p} \frac{(k+m+i)!}{(k+m)!\ i!} \left(\frac{c_p^{\frac{1}{k+n+1} \frac{k+i+1}{k+n+i+1}}\ \cZ_{k_p}(t_0) } {\mu_*^{\frac{1}{k+d/2}}(\eta/2)\left((k+m)!\right)^{\frac{1}{k+1}}}\right)^{i}
\\
&\le M\|D^ku_0\| \left(1 + \frac{c_1^{\frac{1}{k+n+1}}\|D^{k}u_0\|^{\frac{d/2-1}{(k+1)(k+d/2)}}} {\mu_*^{\frac{1}{k+d/2}} (\eta/2)\left((k+m)!\right)^{\frac{1}{k+1}}} \right.
\\
&\qquad \left.\times \sum_{i=1}^\infty \frac{k+m+i}{i} \cdot \frac{(k+m+i-1)!}{(k+m)!\ (i-1)!} \left(\frac{c_1^{\frac{1}{k+n+1}} \|D^{k}u_0\|^{\frac{d/2-1}{(k+1)(k+d/2)}}} {\mu_*^{\frac{1}{k+d/2}} (\eta/2)\left((k+m)!\right)^{\frac{1}{k+1}}}\right)^{i-1} \right)
\end{align*}
The same argument (with Proposition~\ref{prop:HarMaxPrin}) as in the proof of Theorem~\ref{le:DescendDer} leads to
\begin{align*}
|D^ku(t)(x_0)|&\le \lambda M\|D^k u_0\|_\infty\ h^* + \frac{M\ c_1^{\frac{k}{(k-n+1)^2}}\ \cZ_{k}(t_0)}{(\eta/2e)\  (k+m)^{m/k}} \exp\left(\frac{c_1^{\frac{k}{(k-n+1)^2}}\ \cZ_{k}(t_0)}{(\eta/2e)\  (k+m)^{m/k}}\right) \|D^k u_0\|_\infty (1-h^*)
\end{align*}
If condition~\eqref{eq:ParaAdjMaxPRefn} is satisfied, then the above result contradicts $\|D^ku(t)\|> \mu_* \|D^ku_0\|$ assumed at the beginning.

Now we prove for the opposite case, that is the order \eqref{eq:DescDerOrdAlls} stops at some temporal points $t_\tau<t_0+T_*$ for some indexes $k_\tau>k_*$.
Make a descending argument for each pair $(m_p,n_p)$; as time $t$ approaches to a possible blow-up time $T$, the process described in Lemma~\ref{le:AstrBdd} takes effect simultaneously on different `layers' with the subscript $(m_p,n_p)$ and it terminates from the larger pairs until the smallest pair $(m_p,n_p)$ which satisfies either \eqref{eq:ParaAdjMaxPIndW} or \eqref{eq:kRegScale}.
Define
\begin{align*}
T_{j}^*(t) := (M_*-1)^2\ c^{\frac{2j}{j+1}}\ \|D^j u(t)\|^{-\frac{2}{j+1}}\ ,
\end{align*}
where $M_*$ is chosen such that $T_*=(M_*-1)^2(k_*!)^{-\frac{2}{k_*+1}}\cR_{m_p,n_p}(k,c,t_0)^{-2}$.
For any such $t_\tau$, one can assume that at least one index $k_\tau$ (at $t_\tau$) satisfies
\begin{align}\label{eq:DescDerOrdOpp}
\cR_{m_p,n_p}(k_\tau,c,t_\tau)\ge M^{\frac{1}{k+1}}\cR_{m_p,n_p}(k,c,t_0)\ ,
\end{align}
because the opposite for all $k_\tau$ implies $\|D^ku(t_\tau)\|\le \mu_* \|D^ku_0\|$, using the same argument as before. Moreover, we place such indexes in ascending order: $k_*<k_1<k_2<\cdots<k_\tau<k_{\tau+1}<\cdots$ and assume (if such $t_\tau$ ever exists) $t_p$ is the first time that \eqref{eq:DescDerOrdOpp} occurs for $k_p$ (so the order \eqref{eq:DescDerOrdAlls} persists (for $k_p$) until $s=t_p$ at most) while
\begin{align*}
\cR_{m_p,n_p}(k_p,c,t_p) = \max_{k\le j< k_p} \cR_{m_p,n_p}(j,c,t_p)\ .
\end{align*}
We claim that, for some $n_p\le (k_p/k_*)^2$,
\begin{align}\label{eq:BddCplxExt}
\sup_{t_p<s<t_0+T_*} \cR_{m_p,n_p}(k_p,c,s) \le M_*^{\frac{n_p}{k_p+1}} \cR_{m_p,n_p}(k_p,c,t_p)\ .
\end{align}
\textit{Proof of the claim:} Based on the choice of $T_*$ and the assumption~\eqref{eq:DescDerOrd},
\begin{align}\label{eq:BddLowerScale}
\sup_{t_0<s<t_0+T_*}\max_{k\le j\le k_*}\! \cR_{m_p,n_p}(j,c,s) \le \max_{k\le j\le k_*}\! M^{\frac{1}{j+1}}\cR_{m_p,n_p}(j,c,t_0) \le M^{\frac{1}{k+1}}\cR_{m_p,n_p}(k,c,t_0)\ .
\end{align}
Recall that $k_*$ is chosen according to the condition~\eqref{eq:AscDerCond}, while $k_1$ and $t_1$ are, respectively, the smallest index and the first temporal point for \eqref{eq:DescDerOrdOpp} (i.e. the equality holds), implying
\begin{align*}
\sup_{t_0<s<t_1}\max_{k_*< i<k_1}\! \cR_{m_p,n_p}(i,c,s) < M^{\frac{1}{k+1}}\cR_{m_p,n_p}(k,c,t_0)\ ,
\end{align*}
which, together with \eqref{eq:BddLowerScale}, guarantees \eqref{eq:AscDerOrd} (at $s=t_1$, with $\ell=k$ and $k=k_1$), then, by Theorem~\ref{le:AscendDer}
\begin{align*}
\sup_{t_1<s<t_1+T_{k_1}}\! \cR_{m_p,n_p}(i,c,s)  \le M_*^{\frac{1}{k_1+1}} \cR_{m_p,n_p}(k_1,c,t_1) \ ,\qquad \forall\ k_*\le i\le k_1\ .
\end{align*}

If $\displaystyle\sup_{t_1+T_{k_1}<s<t_0+T_*}\cR_{m_p,n_p}(k_1,c,s) \le M_*^{\frac{1}{k_1+1}} \cR_{m_p,n_p}(k_1,c,t_1)$, then \eqref{eq:BddCplxExt} is achieved immediately; otherwise we repeat the above procedure until the above inequality is attained at some $s=t_1+r\cdot T_{k_1}$ or until $t_1+n_1T_{k_1}\ge t_0+T_*$, and this shall lead to
\begin{align*}
\sup_{t_0<s<t_0+T_*}\max_{k_*\le j\le k_1}\! \cR_{m_p,n_p}(j,c,s) \le \left(M_*^{\frac{1}{k_1+1}}\right)^{n_1}\cR_{m_p,n_p}(k_1,c,t_1)
\end{align*}
where, based on the choice of $T_*$ and $M_*$,
\begin{align*}
n_1 &\le (t_0+T_*-t_1)/T_{k_1}(t_1)\lesssim T_*\ (M_*-1)^{-2}(k_1!)^{\frac{2}{k_1+1}}\cR_{m_p,n_p}(k_1,c,t_1)^2
\\
&\lesssim T_*\ (M_*-1)^{-2}(k_1!)^{\frac{2}{k_1+1}}\cR_{m_p,n_p}(k,c,t_0)^2\lesssim (k_1/k_*)^2\ .
\end{align*}
This implies, if $k_1$ satisfies $M_*^{k_1/k_*^2}\le M$ then
\begin{align*}
\sup_{t_0<s<t_0+T_*}\max_{k_*\le j\le k_1}\! \cR_{m_p,n_p}(j,c,s) \le M\ \cR_{m_p,n_p}(k_1,c,t_1)\ .
\end{align*}
If $k_p$ such that $\prod_{\tau=1}^p M_*^{n_\tau/(k_\tau+1)}\le M$, an induction argument leads to
\begin{align*}
\sup_{t_0<s<t_0+T_*}\max_{k_*\le j\le k_p}\! \cR_{m_p,n_p}(j,c,s) \le M_*^{\frac{n_p}{k_p+1}} \cR_{m_p,n_p}(k_p,c,t_p)\le M\ \cR_{m_p,n_p}(k_1,c,t_1)\ ,
\end{align*}
where $n_\tau\le (t_0+T_*-t_\tau)/T_{k_\tau}(t_\tau)\lesssim (k_\tau/k_*)^2$. In fact, for any $k_p$ such that $M_*^{k_p/k_*^2}\le M$,
\begin{align*}
\sup_{t_0<s<t_0+T_*}\max_{k\le j\le k_p}\! \cR_{m_p,n_p}(j,c,s) \le M\ \cR_{m_p,n_p}(k_1,c,t_1)\lesssim M\ \cR_{m_p,n_p}(k,c,t_0)\ .
\end{align*}
This proves the claim (stronger than the claim). On the other hand, we claim that, `for however large index $k_1$',
\begin{align}\label{eq:RkbetaBdd}
\sup_{t_0<s<t_0+T_*}\cR_{m_p,n_p}(k_1,c,s) \le M^\beta\ \cR_{m_p,n_p}(k,c,t_0) \qquad \textrm{with }\ \beta<1\ .
\end{align}
\textit{Proof of the claim:} Recall that $k_1$ is the foremost index for \eqref{eq:DescDerOrdOpp}, so
\begin{align*}
\sup_{s<t_0+T_*} \max_{k_*\le i<k_1}\! \cR_{m_p,n_p}(i,c,s) \le \sup_{s<t_0+T_*} \max_{k\le i\le k_*}\! \cR(i,c,s) \le M^{\frac{1}{k+1}}\cR_{m_p,n_p}(k,c,t_0)\ .
\end{align*}
The opposite of the claim, together with the above restriction (for $i<k_1$), implies there exists $t<t_0+T_*$ such that, for some $\tilde{\beta}<\beta$,
\begin{align*}
\max_{k\le i<k_1}\! \cR_{m_p,n_p}(i,c,t) \le M^{-\tilde{\beta}}\ \cR_{m_p,n_p}(k_1,c,t)
\end{align*}
and with a similar argument to the proof of Theorem~\ref{le:AscendDer} we deduce that
\begin{align*}
\sup_{t<s<t+\tilde{T}} \cC_{m_p,n_p}(k_1,c,\varepsilon,t,s) \le M^{\tilde{\beta}}\ \cR_{m_p,n_p}(k_1,c,t)
\end{align*}
where $|\varepsilon|\gtrsim 1-M^{-\tilde{\beta}}$ and $\tilde{T}\gtrsim (\eta/2)^{-2}\left(1-M^{-\tilde{\beta}}\right)^{-2}\|D^{k_1}u(t)\|^{-\frac{2}{k_1+1}}$. Then, by Theorem~\ref{th:LerayZalpha},  Proposition~\ref{prop:HarMaxPrin} and the above estimate for $D^{k_1}u$,
\begin{align*}
\cR_{m_p,n_p}\left(k_1,c,t+\tilde{T}\right) \le \tilde{\mu}\cdot \cR_{m_p,n_p}(k_1,c,t) \qquad \textrm{with }\ \tilde{\mu}<1\ ,
\end{align*}
which shows that either `spatial intermittency' of $D^{k_1}u$ occurs before $s=t+\tilde{T}$($<t_0+T_*$) with $\cR(k_1,c,s) \le M^{2\tilde{\beta}}\ \cR(k,c,t_0)$ or
\begin{align*}
\sup_{s<t_0+T_*}\cR_{m_p,n_p}(k_1,c,s) < M^{2\tilde{\beta}}\ \cR_{m_p,n_p}(k,c,t_0)\ .
\end{align*}
This proves that \eqref{eq:RkbetaBdd} must hold provided $k_1$( $>k_*$) is the foremost index for \eqref{eq:DescDerOrdOpp} to occur. In summary of the above two claims (i.e. \eqref{eq:BddCplxExt} and \eqref{eq:RkbetaBdd}), we have shown that
\begin{align*}
\sup_{t_0<s<t_0+T_*}\max_{k\le j\le \ell_*}\! \cR_{m_p,n_p}(j,c,s) \le M\ \cR_{m_p,n_p}(k,c,t_0)
\end{align*}
where $\ell_*$ is chosen such that $M_*^{\ell_*/k_*^2}\le M$ (where $\ell_*\gg k_*$ since $M_*\ll M$). Finally, we claim that
\begin{align*}
\sup_{t_0<s<t_0+T_*}\max_{\ell_*< j\le 2\ell_*}\! \cR_{m_p,n_p}(j,c,s) \le M^\beta\! \sup_{t_0<s<t_0+T_*}\max_{k\le j\le \ell_*}\! \cR_{m_p,n_p}(j,c,s) \qquad \textrm{with }\ \beta\lesssim k_*^{-1}\ .
\end{align*}
Assume the opposite, then there exist $\ell_*<i\le 2\ell_*$ and $t<t_0+T_*$ such that
\begin{align*}
\sup_{t_0<s<t_0+T_*}\max_{k\le j\le \ell_*}\! \cR_{m_p,n_p}(j,c,s) < M^{-\beta} \cR_{m_p,n_p}(i,c,t)\ .
\end{align*}
With a similar argument to the proof of Theorem~\ref{le:AscendDer} we deduce that
\begin{align*}
\sup_{t<s<t+\tilde{T}} \cC_{m_p,n_p}(i,c,\varepsilon,t,s) \le \tilde{M}\ \cR_{m_p,n_p}(i,c,t)
\end{align*}
where $|\varepsilon|\gtrsim 1-\tilde{M}^{-1}$ and $\tilde{T}\gtrsim_{\eta, \beta} \|D^iu(t)\|^{-\frac{2}{i+1}}$. Similar to the above argument for $k_1$, by Theorem~\ref{th:LerayZalpha} and Proposition~\ref{prop:HarMaxPrin}, `spatial intermittency' of $D^i u$ occurs at $s=t+\tilde{T}$. Thus, $\|D^iu\|$ starts decreasing whenever it reaches the critical state as above, and this proves the claim. Inductively, one can show
\begin{align*}
\sup_{t_0<s<t_0+T_*}\max_{2^n\ell_*< j\le 2^{n+1}\ell_*}\! \cR_{m_p,n_p}(j,c,s) \le M^{\beta_n}\!\! \sup_{t_0<s<t_0+T_*}\max_{k\le j\le 2^n\ell_*}\! \cR_{m_p,n_p}(j,c,s) \qquad \textrm{with }\ \beta_n\lesssim 2^{-n}k_*^{-1}\ ,
\end{align*}
and therefore
\begin{align*}
\sup_{t_0<s<t_0+T_*}\max_{j>\ell_*}\ \cR_{m_p,n_p}(j,c,s) \le M^{2/k_*}\!\! \sup_{t_0<s<t_0+T_*}\max_{k\le j\le \ell_*}\! \cR_{m_p,n_p}(j,c,s) \ ,
\end{align*}
which, together with the summary of the previous two claims, indicates that
\begin{align*}
\sup_{t_0<s<t_0+T_*}\max_{j\ge k}\ \cR_{m_p,n_p}(j,c,s) \lesssim M\ \cR_{m_p,n_p}(k,c,t_0)
\end{align*}
and the complex extension $\displaystyle\sup_{z\in B_{r_t}(x_0,0)}\left|D^ku(z, t)\right|$ has the same upper estimate as in the initial case, and Proposition~\ref{prop:HarMaxPrin} completes the proof.
\end{proof}

\begin{definition}\label{h-index}
Recall that by Lemma~\ref{le:Rk-transitivity}, for fixed $m,n,k,\nu,c(\ell)$ and $t$, there exists a unique maximal extension of Type-$\cA$ or Type-$\cB$ section containing the index $k$.
We call the union $[\ell_i, \ell_j]:= \cup_{i\le r\le j-1}[\ell_r,\ell_{r+1}]$ ($j-i\ge q$ or the condition~\eqref{eq:AscDerCond} holds) a \textit{sharp string} if all the sections within the string are maximal with respect to the corresponding subscripts $(m,n)$ and all the pairs $(m_r,n_r)$'s within the string satisfy \eqref{eq:JointCompact}, denoted by $[\ell_i, \ell_j]^*$. For any two indexes $i,j$ contained in the same sharp string (with the pairs $(m_r,n_r)$'s determined), define the \textit{homogeneity index} in \eqref{eq:NonHomCond} as follows
\begin{align*}
\bar \theta_\nu [\alpha] (i,j,t)=: \cR_{m_1,n_1}^\nu [\alpha](i,c_1,t) \big/ \cR_{m_p,n_p}^\nu [\alpha](j,c_p,t) \ .
\end{align*}
Particularly, if the sharp string only contains one section, then $(m_1,n_1,c_1)=(m_p,n_p,c_p)$.
Let $\bar \theta_\nu (i,j,t):=\sup_\alpha \bar \theta_\nu [\alpha] (i,j,t)$.
Note that $\bar \theta_\nu [\alpha] (i,j,t)\approx \|(\nabla_\nu)^i u(t)\|^{\frac{1}{i+1}}/ \|(\nabla_\nu)^j u(t)\|^{\frac{1}{j+1}}$ with $(m_1,n_1,c_1)$ and $(m_p,n_p,c_p)$ properly chosen. Moreover, for any $t,\nu,i,j,k,\alpha$, the homogeneity index is multiplicative, i.e.
\begin{align*}
\bar \theta_\nu [\alpha] (i,k,t) \approx \bar \theta_\nu [\alpha] (i,j,t) \cdot \bar \theta_\nu [\alpha] (j,k,t) \ .
\end{align*}
We say a sharp string $[\ell_i, \ell_j]_\alpha^*$ is of Type-$\cA^\beta$ at a temporal point $t$ if there exists $k>\ell_{j+1}$ such that $\bar \theta_\nu [\alpha] (i,k,t)$ satisfies \eqref{eq:NonHomCond}.
\end{definition}

\begin{theorem}\label{cor:LwBddTheta}
Fix an order of dissipation $\beta\ge 1/2$ and $T>0$. Let $u$ be a Leray solution of \eqref{eq:HypNSE1}-\eqref{eq:HypNSE3}. Suppose $\|u_0\|\lesssim (T-t_0)^{-\epsilon}$ such that \eqref{eq:AscDerCond} holds and $t$ is sufficiently close to $T$, then the homogeneity index $\bar \theta$ has a natural upper bound, that is, for any $t$ close to $T$, $\nu\in \bS^{d-1}$ and indexes $k_1,k_2,\alpha$,
\begin{align}\label{eq:HomIndUppBdd}
\bar \theta_\nu [\alpha] (k_1,k_2,t) \le C_2^*(k_1,k_2,t) \approx \left(\frac{2e}{\varepsilon \eta}\right)^{\ln(k_2/k_1)} \prod_{k_1\le r\le k_2} \|(\nabla_\nu)^{r}u(t)\|^{\frac{d/2-1}{r(r+1)(r+d/2)}} \ .
\end{align}
If the above bound is not satisfied at the initial time $t_0$, then either all the derivatives between $k_1$ and $k_2$ are non-increasing until $T$ and the solution is regular, or \eqref{eq:HomIndUppBdd} holds from some temporal point $t>t_0$ up to $T$.
\end{theorem}

\begin{proof}
Without loss of generality we assume $k_1$ and $k_2$ are both contained in a Type-$\cB$ sharp string at the initial time $t_0$, since any potential Type-$\cA$ sections between $k_1$ and $k_2$ shall reduce the size of $\bar \theta_\nu [\alpha] (k_1,k_2,t)$, and without loss of generality assume that if $k_1$ and $k_2$ are contained in a Type-$\cA$ sharp string then $k_2$ is the maximal index described in \eqref{eq:NonDescAtLi}.

Suppose $k_1$ and $k_2$ are contained in one Type-$\cB$ section, then by Corollary~\ref{cor:DerOrdConfig} there exists $t>t_0$ such that
\begin{align*}
&\cR_{m,n}(k_1,c_\ell,t) \le \cB_{k_2,k_1} \cdot \cR_{m,n}(k_2,c_\ell,t)
\\
\Longrightarrow \quad &\bar \theta_\nu [\alpha] (k_1,k_2,t) \le \cB_{k_2,k_1} \lesssim k_2^{-m\left(k_1^{-1}- k_2^{-1}\right)} \left(\frac{2e}{\varepsilon\eta}\right)^{\ln(k_2/k_1)} \prod_{r=k_1}^{k_2}\ \cZ_r(t)^{1/r}
\\
& \lesssim \eta^{-\ln(k_2/k_1)}\ \prod_{k_1\le r\le k_2}\ \|(\nabla_\nu)^{r}u(t)\|^{\frac{d/2-1}{r(r+1)(r+d/2)}} \lesssim \eta^{-\epsilon_2} (T-t)^{-(d/2-1)(k_1^{-1}-k_2^{-1})}
\end{align*}
Now suppose $k_1$ and $k_2$ are contained in a Type-$\cB$ sharp string $[\ell_i,\ell_{i+q}]^*$. Without loss of generality assume the number and sizes of the maximal sections within the sharp string remains unchanged, i.e. all the subscripts $(m_r,n_r)$'s are fixed. We claim that
\begin{align}\label{eq:RestrBtoASp}
\max_{i\le p\le i+q} \cR_{m_p,n_p}(k_p,c(\ell_p), \tilde t) \lesssim  \cR_{m_{i+q},n_{i+q}}(k_{i+q},c(\ell_{i+q}), \tilde t)
\end{align}
where $\tilde t$ is the first time when $[\ell_i,\ell_{i+q}]^*$ switches to a Type-$\cA$ sharp string.

\textit{Proof of the claim:} Let $[\ell_i,\ell_{i+q}]_{(p)}$ denote the string $\{\cR_{m_p,n_p}(k_r, c(\ell_r), t)\}_{i\le r\le i+q}$ with uniform subscripts $(m_p,n_p)$. For fixed $p$, by Lemma~\ref{le:BstrBdd},
\begin{align*}
\max_{\ell_r\le j\le \ell_{i+q}} \sup_{t_0<s<\tilde t_p} \cR_{m_p,n_p}(j, c(\ell_r), s) \le  \max_{r\le v\le i+q} \cR_{m_p,n_p}(w_v,c(\ell_v), t_0)
\end{align*}
where $\tilde t_p$ is the first time when $[\ell_i,\ell_{i+q}]_{(p)}$ switches to a Type-$\cA$ string.
By the implication~\eqref{eq:JointMonoImp} we know that if $[\ell_i,\ell_{i+q}]_{(p)}$ is of Type-$\cA$ then $[\ell_i,\ell_{i+q}]_{(p+1)}$ is of Type-$\cA$ as well and vice versa. Therefore $\tilde t_p\ge \tilde t_{p+1}$.
If some $[\ell_i,\ell_{i+q}]_{(p)}$ is always of Type-$\cB$ up to $T$, then by Lemma~\ref{le:BstrBdd} (or Theorem~\ref{le:DescendDer}), all the derivatives within the string are non-increasing and by the above restriction \eqref{eq:RestrBtoASp} follows.
If all $[\ell_i,\ell_{i+q}]_{(p)}$'s switch to Type-$\cA$ strings at some $\tilde t$, then $[\ell_i,\ell_{i+q}]^*$ becomes a Type-$\cA$ sharp string at $\tilde t$. Now by the implication~\eqref{eq:JointMonoImp} at all joints $(m_r,n_r)$'s and the order of switching of $[\ell_i,\ell_{i+q}]_{(p)}$'s,
\begin{align*}
\max_{r\le p\le i+q} \cR_{m_p,n_p}(k_p,c(\ell_p), \tilde t_r) \lesssim  \cR_{m_{i+q},n_{i+q}}(k_{i+q},c(\ell_{i+q}), \tilde t_r)
\end{align*}
which proves the claim.

Now by the claim and the computation result \eqref{eq:PureCmpctRatio} satisfied by Type-$\cA$ sharp string, since $k_1\in [\ell_i,\ell_{i+1}]_{(i)}$ and $k_2\in [\ell_{i+q},\ell_{i+q+1}]_{(i+q)}$,
\begin{align*}
\bar \theta_\nu [\alpha] (k_1,k_2,\tilde t) &=: \frac{\cR_{m_1,n_1}^\nu [\alpha](k_1,c_1,\tilde t)}{\cR_{m_p,n_p}^\nu [\alpha](k_2,c_p,\tilde t)}
\lesssim \left(\frac{2e}{\varepsilon \eta}\right)^{\ln(k_2/k_1)} \prod_{k_1\le r\le k_2} \|(\nabla_\nu)^{r}u(\tilde t)\|^{\frac{d/2-1}{r(r+1)(r+d/2)}}  \
\end{align*}

\end{proof}

\begin{corollary}\label{cor:tLwBddHomInd}
Fix an order of dissipation $\beta>1$, a direction $\nu\in \bS^{d-1}$ and $T>0$. Suppose $u_0\in L^\infty(\bR^d) \cap L^2(\bR^d)$. If $T$ is the first blow-up time then
\begin{align*}
\|(\nabla_\nu)^k u(t)\|^{\frac{1}{k+1}} \lesssim (T-t)^{-\frac{k+2}{2\beta(k+1)}-\epsilon}
\end{align*}
for any fixed $\epsilon$. And \eqref{eq:HomIndUppBdd} may also be written as
\begin{align*}
\bar \theta_\nu [\alpha] (k_1,k_2,t) \lesssim \eta^{-\epsilon_2} (T-t)^{-(d/2-1)(k_1^{-1}-k_2^{-1})} \ .
\end{align*}
Moreover, if a sharp string $[\ell_i, \ell_j]_\alpha^*$ is of Type-$\cA^\beta$ then it is of Type-$\cA$.
\end{corollary}

\begin{proof}
It can be shown by induction and Lemma~\ref{le:Rk-transitivity} that $\displaystyle\|D^{r}u(t)\|^{\frac{1}{r+1}}\lesssim (T-t)^{-\frac{r+2}{2(r+1)}}$ as $t\to T$; otherwise larger size of $\displaystyle\|D^{r}u(t)\|^{\frac{1}{r+1}}$ results in intermittency of $D^{r}u(t)$ as $t\to T$ and $T$ is not a blow-up time.
Moreover, for any $s>\tilde t$ when $[\ell_i,\ell_{i+q}]^*(s)$ is a Type-$\cA$ sharp string, the above bound still holds.
\end{proof}

\begin{theorem}\label{le:AscendDerHypSp}
Let $u$ be a Leray solution initiated at $u_0$ and $[\ell,\ell_q]^*$ be a Type-$\cA^\beta$ sharp string at the initial time $t_0$. That is
\begin{align}\label{eq:AscDerOrdHypSp}
\cR_{m_p,n_p}(j,c_p,t_0) \le \cR_{m_p,n_p}(k_{p+1},c_p,t_0) \ , \qquad \forall  k_p\le j <k_{p+1}
\end{align}
for suitable constants $c_p=c_p(k)$ and pairs $(m_p,n_p)$ which satisfy \eqref{eq:JointCompact}.
If $c_1$, $\ell$ and $k_1$ satisfy
\begin{align}\label{eq:AscDerCondHypSp}
c_1 \|u_0\|_2 \|u_0\|^{d/2-1} \frac{(\ell!)^{1/2}\ell}{(\ell/2)!} \lesssim (k_1!)^{1/(k_1+1)}
\end{align}
and $T^{1/2}\lesssim \left(\phi(\ell,k_q)+\psi(\ell,k_q)\right)^{-\frac{2\beta}{2\beta-1}} \|D^{k_q}u_0\|^{-\frac{2\beta}{(2\beta-1)(k_q+1)}}$ with $\phi,\psi$ defined in Theorem~\ref{le:AscendDerHyp}, then for any $\ell_r\le j \le \ell_{r+1}$, the complex solution of \eqref{eq:HypNSE1}-\eqref{eq:HypNSE3} has the following upper bounds:
\begin{align}\label{eq:AscDerUpBddHypSp}
&\underset{t\in(0,T)}{\sup} \cC_{m_r,n_r}(j,c(\ell_r),\varepsilon,t_0,t)^{j+1} \le M\cdot\cR_{m_r,n_r}(j,c(\ell_r),t_0)^{j+1} + \Theta[\beta](q,r)^{-1} \cR_{m_q,n_q}(k_q,c(\ell_q),t_0)^{j+1}
\end{align}
where $\cD_t$ is given by \eqref{eq:AnalDomHyp}. For the real solutions the above result becomes
\begin{align}
\underset{t\in(0,\tilde{T})}{\sup} \cR_{m_r,n_r}(j,c(\ell_r),t)^{j+1} \le \cR_{m_r,n_r}(j,c(\ell_r),t_0)^{j+1} + \Theta[\beta](q,r)^{-1} \cR_{m_q,n_q}(k_q,c(\ell_q),t_0)^{j+1} \label{eq:AscDerUpBddRealHypSp}
\end{align}
where $\tilde{T}$ does not depend on $M$ and $\Theta[\beta](q,r) := C_1^*[\beta](k_q,t_0)\cdot C_2^*(\ell_r,k_q,t_0)\cdot c(\ell_q)/c(\ell_r)$  with $C_1^*$ and $C_2^*$ defined in Theorem~\ref{th:HypNSEReg} and Theorem~\ref{cor:LwBddTheta} respectively.
\end{theorem}

\begin{lemma}\label{le:AstrBddSp}
Suppose $\sup_{t>t_0}\|u(t)\|\lesssim (T-t_0)^{-\epsilon_i}$ such that \eqref{eq:AscDerCond} holds and the assumption~\eqref{eq:NonHomCond} holds for all $k\ge \ell_i$. If a sharp string $[\ell_i,\ell_{i+q}]^*$ is of Type-$\cA^\beta$ at an initial time $t_0$, then for any $i\le r<i+q$,
\begin{align}\label{eq:RestrAtoB}
\max_{\ell_r\le j\le \ell_{i+q}} \sup_{t_0<s<\tilde t} \cR_{m_r,n_r}(j, c(\ell_r), s) &\lesssim \left(1+\Theta[\beta](w_{i+q},r)^{-1}\right)^{C_1^*[\beta](\ell_r,t_0)\big/\bar\theta_\nu (\ell_i,\ell_{i+q}, t_0)} \notag
\\
&\qquad\qquad \Theta[\beta](i+q,r)\cdot \cR_{m_{i+q},n_{i+q}}(w_{i+q},c(\ell_{i+q}), t_0)
\end{align}
where $C_1^*$ is defined by \eqref{eq:NonHomCond}, $\displaystyle\Theta[\beta](p,r) \lesssim  (T-\tilde t)^{-\epsilon}$ is defined in Theorem~\ref{le:AscendDerHypSp} and $\tilde t$ is the first time when $[\ell_i,\ell_{i+q}]^*$ becomes a non-Type-$\cA^\beta$ string; we set $\tilde t= T$ if $[\ell_i,\ell_{i+q}]^*$ is always of Type-$\cA^\beta$ before $T$.
\end{lemma}

\begin{proof}
Recall that by the implication~\eqref{eq:JointMonoImp} we know that if $[\ell_i,\ell_{i+q}]_{(p)}$ is of Type-$\cA$ then $[\ell_i,\ell_{i+q}]_{(p+1)}$ is of Type-$\cA$ as well.
Make an ascending argument for each pair $(m_p,n_p)$; as time $t$ approaches to a possible blow-up time $T$, the process described in Lemma~\ref{le:AstrBdd} takes effect simultaneously on different `layers' with the subscript $(m_p,n_p)$ and it terminates from the larger pairs until the smallest pair $(m,n)$ which satisfies either \eqref{eq:ParaAdjMaxPIndW} or \eqref{eq:kRegScale}.
Meanwhile, the head of the string is restricted by Theorem~\ref{le:AscendDerHypSp}.
\end{proof}

\begin{lemma}\label{le:MaxIndxAStrSp}
Suppose $\sup_{t>t_0}\|u(t)\|\lesssim (T-t_0)^{-\epsilon_i}$ such that \eqref{eq:AscDerCond} holds while \eqref{eq:JointCompact} and \eqref{eq:ParaAdjMaxPRefn} are satisfied at any temporal point with $\ell=\ell_i$ and $(k,c(k))=(\ell_p,c(\ell_p))$ for any $i\le p\le i+q$.
If a sharp string $[\ell_i,\ell_{i+q}]^*$ is of Type-$\cB$ (resp. Type-$\cA$) at an initial time $t_0$ and $\tilde t$ is the first time when it switches to a Type-$\cA$ (resp. Type-$\cA^\beta$) sharp string, then the index $k_p$ described in \eqref{eq:NonDescAtLi} for any $i\le p\le i+q$ has a maximum; more precisely, with the notation in the proof of Lemma~\ref{le:AstrBdd} and $p^*$ being the index for the maximum in $\{\cR_{m_p,n_p}(w_p, c(\ell_p), \tilde t)\}_{i\le p\le i+q}$, there exists an index $k_*$ such that
\begin{align*}
\bar \theta_\nu [\alpha] (j,w_{p^*},\tilde t) &\lesssim C_2^*(j,w_{p^*},\tilde t) \ , \qquad \forall\ w_{p^*}\le j\le k_* \ ,
\\
\bar \theta_\nu [\alpha] (j,w_{p^*},\tilde t) &\lesssim C_2^*(j,w_{p^*},\tilde t) \ , \qquad \forall\ j>k_* \ ,
\end{align*}
and at $j=k_*$, in particular, $\bar \theta_\nu [\alpha] (k_*,w_{p^*},\tilde t) \approx C_2^*(k_*,w_{p^*},\tilde t)$. Moreover, $k_*\le \ell_{i+3q}$.
\end{lemma}

\begin{proof}
Make an ascending or descending argument as in Theorem~\ref{cor:LwBddTheta} or Lemma~3.18 in \citet{Grujic2019} for each layer $(m_p,n_p)$.
\end{proof}

\bigskip

\bigskip

\begin{proof}[Proof of Theorem~\ref{th:HypNSEReg}]

Define $\hat{\ell}_i:=\ell_{iq}$ and $\hat{w}_i:=w_{\hat{p}_i}$ where $\hat{p}_i$ is the minimal index within $\{iq, \cdots, (i+1)q\}$ such that
\begin{align*}
\cR_{m_{\hat{p}_i},n_{\hat{p}_i}}\left(w_{\hat{p}_i},c(\ell_{\hat{p}_i}), t\right) = \max_{iq\le p\le (i+1)q} \cR_{m_p,n_p}(w_p,c(\ell_p), t) \ .
\end{align*}
In the following we write $\hat{\cR}_{\hat{p}_i}(t)$ for short.
Note that $\hat{p}_i(t)$ and $\hat{w}_i(t)$ may be variant in time, and we will always assume $\hat{p}_i$ and $\hat{w}_i$ correspond to the temporal point $t$ in $\cR\left(\cdot,\cdot, t\right)$ if there is no ambiguity.
Let $\hat{t}_1(i)$ be the first time when $[\hat{\ell}_i, \hat{\ell}_{i+1}]$ switches to a Type-$\cA$ sharp string if it is of Type-$\cB$ at $t_0$ (in particular, $\hat{t}_1(i)=t_0$ if $[\hat{\ell}_i, \hat{\ell}_{i+1}]$ is of Type-$\cA$ at $t_0$) and let $\tilde{t}_1(i)$ be the first time when $[\hat{\ell}_i, \hat{\ell}_{i+1}]$ switches to a Type-$\cB$ sharp string after $\hat{t}_1(i)$.
Inductively, we let $\hat{t}_n(i)$ (resp. $\tilde{t}_n(i)$) be the first time when $[\hat{\ell}_i, \hat{\ell}_{i+1}]$ switches to a Type-$\cA$ (resp. Type-$\cB$) sharp string after $\tilde{t}_{n-1}(i)$ (resp. after $\hat{t}_n(i)$).
Let $\tilde t_n[\beta](i)$ the first time when a Type-$\cA^\beta$ string $[\hat{\ell}_i, \hat{\ell}_{i+1}]^*$ switches to non-Type-$\cA^\beta$ from the last temporal point $\hat{t}_n[\beta](i)$ that it becomes a Type-$\cA^\beta$ string.
In the following we will write $\tilde t_n(i)$ and $\hat{t}_n(i)$ for short if no ambiguity.

We first prove the statement by assuming that switching between Type-$\cA^\beta$ and non-Type-$\cA^\beta$ only occurs within $[\hat{\ell}_0, \hat{\ell}_1]$. We will verify in the proof step by step that Theorem~\ref{cor:LwBddTheta}, Lemma~\ref{le:AstrBddSp} and Lemma~\ref{le:MaxIndxAStrSp} are applicable for all $i$ by showing $\|u(t)\|\lesssim (T-t)^{-\epsilon_i}$.
With this and the assumption~\eqref{eq:NonHomCond} for all $k\ge \ell_0$, in particular, for $i=0$, Theorem~\ref{cor:LwBddTheta} and Lemma~\ref{le:AstrBddSp} indicate that, for any $0\le r\le q$,
\begin{align}
\max_{\ell_r\le j\le \ell_q} \sup_{t_0<s<\hat{t}_1(0)} \cR_{m_r,n_r}(j, c(\ell_r), s) &\le  \max_{r\le p\le q} \cR_{m_p,n_p}(w_p, c(\ell_p), t_0) \ , \notag
\\
\max_{\ell_r\le j\le \ell_q} \sup_{\hat{t}_1(0)<s<\tilde{t}_1(0)} \cR_{m_r,n_r}(j, c(\ell_r), s) &\lesssim \Theta[\beta](\hat{p}_0,r) \cdot \hat{\cR}_{\hat{p}_0}(\hat{t}_1(0)) \ , \label{eq:StrN0B1}
\end{align}
where $\Theta[\beta](\hat{p}_0,r)$ is a constant given by Lemma~\ref{le:AstrBddSp}. (Note that the first estimate above can be trivial in a sense that $\hat{t}_1(i)=t_0$.) Connection of the above results at $\hat{t}_1(0)$ yields, for any $0\le r\le q$,
\begin{align*}
\max_{\ell_r\le j\le \ell_q} \sup_{t_0<s<\tilde{t}_1(0)} \cR_{m_r,n_r}(j, c(\ell_r), s) &\lesssim \Theta[\beta](\hat{p}_0,r) \cdot \hat{\cR}_{\hat{p}_0}(t_0) \ .
\end{align*}
In particular, $\Theta[\beta](\hat{p}_0,0)\lesssim C_1^*[\beta](\ell_{\hat{p}_0},t_0)\cdot C_2^*(\ell,\ell_{\hat{p}_0},t_0)\cdot c(\ell_{\hat{p}_0})/c(\ell)\lesssim  (T-t_0)^{-\epsilon}$.
By Lemma~\ref{le:GNIneq} and the above result, for any $t_0<s< \tilde{t}_1(0)$,
\begin{align*}
\|u\left(s\right)\| &\lesssim \|u_0\|_2 \|(\nabla_\nu)^{\ell}u\left(s\right)\|^{\frac{d/2}{\ell+d/2}} \notag
\\
&\lesssim \|u_0\|_2 \left(c(\ell)^{\frac{\ell}{\ell+1}} (\ell!)^{\frac{1}{\ell+1}} \cR_{0,0}\left(\ell, c(\ell), s\right) \right)^{\frac{(d/2)(\ell+1)}{\ell+d/2}}  \notag
\\
&\lesssim \|u_0\|_2 \left(c(\ell)^{\frac{\ell}{\ell+1}} (\ell!)^{\frac{1}{\ell+1}} \Theta[\beta](\hat{p}_0,0) \cdot \hat{\cR}_{\hat{p}_0}(t_0)\right)^{\frac{(d/2)(\ell+1)}{\ell+d/2}} \ .  
\end{align*}
Suppose that $\|u_0\|\lesssim (T-t_0)^{-1/d}$ up to some constant. By Theorem~\ref{th:LinftyIVP}, we may assume without loss of generality that
\begin{align*}
\|(\nabla_\nu)^{\hat{w}_0}u_0\|^{\frac{1}{\hat{w}_0+1}} \lesssim (\hat{w}_0!)^{\frac{1}{\hat{w}_0+1}} \|u_0\| \lesssim (\hat{w}_0!)^{\frac{1}{\hat{w}_0+1}} (T-t_0)^{-1/d} \ .
\end{align*}
Combining the above results yields
\begin{align*}
\sup_{t_0 <s< \tilde{t}_1(0)} \|u(s)\| &\lesssim \|u_0\|_2 \left(c(\ell)^{\frac{1}{\hat{w}_0+1} - \frac{1}{\ell+1}} (\ell!)^{\frac{1}{\ell+1}} \Theta[\beta](\hat{p}_0,0) \cdot (T-t_0)^{-1/d} \right)^{d/2}
\\
&\lesssim_{\|u_0\|_2} \left( (\hat{w}_0/\ell)^{\ln(2e/\eta)} (T-t_0)^{-1/d}\right)^{d/2} \lesssim_{\|u_0\|_2} (\ell_q/\ell)^{\frac{d}{2}\ln\left(\frac{2e}{\eta}\right)} (T-t_0)^{-1/2} \ ,
\end{align*}
which justifies the assumption of the two lemmas for $t< \tilde{t}_1(0)$.
With in mind that $[\hat{\ell}_0, \hat{\ell}_1]$ is of Type-$\cB$ at $\tilde{t}_1(0)$, the particular restriction of \eqref{eq:StrN0B1} at $\tilde{t}_1(0)$ together with Theorem~\ref{cor:LwBddTheta} (starting at $\tilde{t}_1(0)$) indicates that, for any $0\le r\le q$,
\begin{align*}
\max_{\ell_r\le j\le \ell_q} \sup_{\tilde{t}_1(0)<s<\hat{t}_2(0)} \cR_{m_r,n_r}(j, c(\ell_r), s) &\lesssim \Theta[\beta](\hat{p}_0,r) \cdot \hat{\cR}_{\hat{p}_0}(\hat{t}_1(0)) \lesssim  \Theta[\beta](\hat{p}_0,r) \cdot \hat{\cR}_{\hat{p}_0}(t_0) \ .
\end{align*}
And the same argument as above leads to
\begin{align*}
\sup_{t_0 <s< \hat{t}_2(0)} \|u(s)\| &\lesssim \|u_0\|_2 (\ell_q/\ell)^{\frac{d}{2}\ln\left(\frac{2e}{\eta}\right)} (T-t_0)^{-1/2} \ ,
\end{align*}
which justifies the assumption of the two lemmas up to $t< \hat{t}_2(0)$.

We continue the previous argument at $\hat{t}_2(0)$. Again, by Lemma~\ref{le:AstrBddSp}, for any $0\le r\le q$,
\begin{align*}
\max_{\ell_r\le j\le \ell_q} \sup_{\hat{t}_2(0) <s< \hat{t}_2(0) + T_{k_q}}\!\! \cR_{m_r,n_r}(j, c(\ell_r), s) \lesssim \Theta[\beta](\hat{p}_0,r) \cdot \hat{\cR}_{\hat{p}_0}(\hat{t}_2(0)) \ ,
\end{align*}
where $T_{k_q}\approx \left(\phi(\ell,k_q)+\psi(\ell,k_q)\right)^{-\frac{2\beta}{2\beta-1}} \|(\nabla_\nu)^{k_q}u(\hat{t}_2(0))\|^{-\frac{2\beta}{(2\beta-1)(k_q+1)}}$ with $\phi,\psi$ introduced in Theorem~\ref{le:AscendDerHyp}, and by Lemma~\ref{le:MaxIndxAStrSp} we know $k_q\le \ell_{3q}$, thus
\begin{align*}
T_{k_q} \gtrsim \left(\phi(\ell,\ell_{3q})+\psi(\ell,\ell_{3q})\right)^{-\frac{2\beta}{2\beta-1}} \|(\nabla_\nu)^{\ell_{3q}}u(\hat{t}_2(0))\|^{-\frac{2\beta}{(2\beta-1)(\ell_{3q}+1)}} \ .
\end{align*}
Without loss of generality, we assume that $\hat{p}_0$ is invariant in time and that $\tilde t_2(0) \in [\hat{t}_2(0), \hat{t}_2(0) + T_{k_q}]$. In general, we assume $\tilde t_n(0) \in [\hat{t}_n(0), \hat{t}_n(0) + T_{k_q}]$, and by Lemma~\ref{le:AstrBddSp} and Proposition~\ref{prop:HarMaxPrin}
\begin{align}
\sup_{\hat{t}_n(0) <s< \hat{t}_n(0) + T_{k_q}}\! \cR_{m_r,n_r}(\hat{w}_0, c(\ell_{\hat{p}_0}), s) &\le \left(1+\Theta[\beta](\hat{p}_0,r)^{-1}\right)^{\frac{C_1^*[\beta](\ell_r,\hat{t}_n(0))}{\bar\theta_\nu (\hat{\ell}_0, \hat{\ell}_1, \hat{t}_n(0))}} \hat{\cR}_{\hat{p}_0}(\hat{t}_n(0)) \ , \label{eq:m0IncreN}
\\
\hat{\cR}_{\hat{p}_0}(k_q, c(\ell_{\hat{p}_0}), \hat{t}_n(0) + T_{k_q}) &\le \left(\mu_{k_q}(\hat{p}_0)\right)^{\frac{1}{k_q+1}} \cdot \hat{\cR}_{\hat{p}_0}(k_q, c(\ell_{\hat{p}_0}), \hat{t}_n(0)) \ . \label{eq:kqDecreN}
\end{align}
We claim that with the above settings, one of the followings occurs:

(I) $\hat{\cR}_{\hat{p}_0}\left(\hat{t}_{n+1}(0)\right)\le \hat{\cR}_{\hat{p}_0}\left(\hat{t}_n(0)\right)$;

\vspace{0.04in}

(II) $\hat{t}_{n+1}(0) - \tilde t_n(0)\ge \left(\phi+\psi\right)^{-\frac{2\beta}{2\beta-1}} \|(\nabla_\nu)^{k_q}u\left(\hat{t}_n(0) + T_{k_q}\right)\|^{-\frac{2\beta}{(2\beta-1)(k_q+1)}}$.

\noindent\textit{Proof of the claim:} Assume the opposite of (I), i.e. $\hat{\cR}_{\hat{p}_0}\left(\hat{t}_{n+1}(0)\right)\ge \hat{\cR}_{\hat{p}_0}\left(\hat{t}_n(0)\right)$. Without loss of generality we assume that $k_q$ is invariant with $\hat{t}_n(0)$, and that $\hat t_{n+1}(0)>\hat{t}_n(0) + T_{k_q}$.
With in mind that $[\hat{\ell}_0, \hat{\ell}_1]$ is of Type-$\cA$ at $\hat{t}_n(0)$ and at $\hat{t}_{n+1}(0)$, the opposite of (I) indicates that
\begin{align*}
&\hat{\cR}_{\hat{p}_0}\left(k_q, c(\ell_{\hat{p}_0}), \hat{t}_{n+1}(0)\right) = \hat{\cR}_{\hat{p}_0}\left(\hat{t}_{n+1}(0)\right)
\ge \hat{\cR}_{\hat{p}_0}\left(\hat{t}_n(0)\right) = \hat{\cR}_{\hat{p}_0}\left(k_q, c(\ell_{\hat{p}_0}), \hat{t}_n(0)\right)
\end{align*}
which, combined with \eqref{eq:kqDecreN}, yields
\begin{align*}
\hat{\cR}_{\hat{p}_0}\left(k_q, c(\ell_{\hat{p}_0}), \hat{t}_{n+1}(0)\right) \ge \left(\mu_{k_q}(\hat{p}_0)\right)^{-\frac{1}{k_q+1}} \cdot \hat{\cR}_{\hat{p}_0}(k_q, c(\ell_{\hat{p}_0}), \hat{t}_n(0) + T_{k_q}) \ ,
\end{align*}
in other words, $\|(\nabla_\nu)^{k_q}u\left(\hat{t}_{n+1}(0)\right)\| \ge \left(\mu_{k_q}(\hat{p}_0)\right)^{-1} \|(\nabla_\nu)^{k_q}u\left(\hat{t}_n(0) + T_{k_q}\right)\|$.
By Theorem~\ref{th:MainThmVel} (applied in the opposite way), the time span required for $D^{k_q}u$ to increase by $M_{k_q}=\left(\mu_{k_q}(\hat{p}_0)\right)^{-1}$ is at least (with in mind that $\hat t_{n+1}(0)>\hat{t}_n(0) + T_{k_q}$),
\begin{align*}
T_{k_q}^* := \left(\phi(\ell_{\hat{p}_0},k_q)+\psi(\ell_{\hat{p}_0},k_q)\right)^{-\frac{2\beta}{2\beta-1}} \|(\nabla_\nu)^{k_q}u_0\|^{-\frac{2\beta}{(2\beta-1)(k_q+1)}} \ .
\end{align*}
Recall that $\tilde t_n(0)<\hat{t}_n(0) + T_{k_q}$, so
\begin{align*}
\hat{t}_{n+1}(0) - \tilde t_n(0) \ge T_{k_q}^* := \left(\phi+\psi\right)^{-\frac{2\beta}{2\beta-1}} \|(\nabla_\nu)^{k_q}u_0\|^{-\frac{2\beta}{(2\beta-1)(k_q+1)}} \ .
\end{align*}
This ends the proof of the claim.
Moreover, by Lemma~\ref{le:MaxIndxAStrSp} we know $k_q\le \ell_{3q}$ and
\begin{align*}
\hat{t}_{n+1}(0) - \tilde t_n(0) \ge  T_{k_q}^* \ge \left(\phi+\psi\right)^{-\frac{2\beta}{2\beta-1}} \|(\nabla_\nu)^{k_q}u_0\|^{-\frac{2\beta}{(2\beta-1)(k_q+1)}} \ .
\end{align*}
The above claim together with multiple iterations of \eqref{eq:m0IncreN} leads to
\begin{align*}
\hat{\cR}_{\hat{p}_0}\left(\hat{t}_{n+1}(0)\right) &\le \left(1+\Theta[\beta](\hat{p}_0,0)^{-1}\right)^{\frac{C_1^*[\beta](\hat{\ell}_0,\hat{t}_2(0))}{\bar\theta_\nu (\hat{\ell}_0, \hat{\ell}_1, \hat{t}_2(0))}\xi} \ \hat{\cR}_{\hat{p}_0}\left(\hat{t}_2(0)\right)
\end{align*}
where $\xi$ is the total number of times that Case~(II) in the claim occurs within $[\hat{t}_2(0), \hat{t}_{n+1}(0)]$. The worst scenario is $\xi =n$, that is, Case~(II) in the claim occurs throughout $[\hat{t}_2(0), \hat{t}_{n+1}(0)]$, in which case, the above restriction, together with Theorem~\ref{cor:LwBddTheta} and Lemma~\ref{le:AstrBddSp} (applied for $n$ times), indicates that, for any $0\le r\le q$,
\begin{align*}
\max_{\ell_r\le j\le \ell_q} \sup_{t_0 <s< \hat{t}_{n+1}(0)}\!\! \cR_{m_r,n_r}(j, c(\ell_r), s) &\le \left(1+\Theta[\beta](\hat{p}_0,r)^{-1}\right)^{\frac{C_1^*[\beta](\ell_r,t_0)}{\bar\theta_\nu (\hat{\ell}_0, \hat{\ell}_1, t_0)}n} \Theta[\beta](\hat{p}_0,r) \cdot \hat{\cR}_{\hat{p}_0}(t_0) \ .
\end{align*}
Recall that the precise upper bound for the ratio was given in the proof of Lemma~\ref{le:AstrBddSp}:
\begin{align*}
\left(1+\Theta[\beta](\hat{p}_0,r)^{-1}\right)^{\frac{C_1^*[\beta](\ell_r,t_0)}{\bar\theta_\nu (\hat{\ell}_0, \hat{\ell}_1, t_0)}} &\lesssim \left(1+\eta^{\epsilon_2} (T-\hat{t}_n(0))^{(\beta-1)/(2\ell_q) + \xi_r(d/2-1)(\ell_r^{-1}-\ell_q^{-1})} \right)^{(T-t_0)^{-(\beta-1)/(2\ell_q)}}
\\
&\lesssim \exp\left(\eta^{\epsilon_2} (T-\hat{t}_n(0))^{\xi_r(d/2-1)(\ell_r^{-1}-\ell_q^{-1})}\right)
\end{align*}
where $\xi_r$ was defined earlier in the section, assuming $\bar\theta_\nu (\hat{\ell}_0, \hat{\ell}_1, t_0)\gtrsim 1$ (The opposite case will be discussed later in the proof).

In the rest of the proof, we show that the above iterations of Theorem~\ref{cor:LwBddTheta} and Lemma~\ref{le:AstrBddSp} repeat for finitely many times as $\hat{t}_n(0)$ is approaching towards $T$ by revealing that the time span $\hat{t}_{n+1}(0) - \hat{t}_n(0)$ (or $\hat{t}_{n+1}(i_*) - \hat{t}_n(i_*)$ for some index $i_*$) for each application of Lemma~\ref{le:AstrBddSp} and Theorem~\ref{cor:LwBddTheta} remains greater than a fixed number.
Note that he above argument guarantees that at least for small value of $n$ this is the case:
\begin{align*}
\hat{t}_{n+1}(0) - \hat{t}_n(0) \ge T_{k_q} + T_{k_q}^* \ge 2^{-2\ell_{3q}} \|(\nabla_\nu)^{\ell_{3q}}u\left(\hat{t}_n(0) + T_{k_q}\right)\|^{-\frac{\beta d}{(2\beta-1)(\ell_{3q}+d/2)}} \ .
\end{align*}
In the following we will write $\hat t_n$ for short.
Assuming this would continue as $\hat{t}_n$ goes towards $T$, then the maximal number of iterations before $\hat{t}_n$ reaches $T$ is
\begin{align*}
n^*:= & (T-\hat{t}_n) / (\hat{t}_{n+1} - \hat{t}_n) \le (T-\hat{t}_n) \cdot 2^{2\ell_{3q}} \|(\nabla_\nu)^{\ell_{3q}}u\left(\hat{t}_n + T_{k_q}\right)\|^{\frac{\beta d}{(2\beta-1)(\ell_{3q}+d/2)}}
\\
&\qquad \le (T-\hat{t}_n) \cdot 2^{2\ell_{3q}} \|(\nabla_\nu)^{\ell_{3q}}u\left(\hat{t}_1\right)\|^{\frac{\beta d}{(2\beta-1)(\ell_{3q}+d/2)}} \le (T-\hat{t}_n) \cdot \|u_0\|^{\frac{\beta d}{2\beta-1}}
\end{align*}
while $\hat{\cR}_{\hat{p}_0}(s)$ increases at most by
\begin{align*}
\left(1+\Theta[\beta](\hat{p}_0,r)^{-1}\right)^{\frac{C_1^*[\beta](\ell_r,\hat{t}_n)}{\bar\theta_\nu (\hat{\ell}_0, \hat{\ell}_1, \hat{t}_n)}n^*} &\le \exp\left((T-\hat{t}_n) \cdot \|u_0\|^{\frac{\beta d}{2\beta-1}} \eta^{\epsilon_2} (T-\hat{t}_n)^{\xi_r(d/2-1)(\ell_r^{-1}-\ell_q^{-1})}\right)
\\
&\lesssim \exp\left(\eta^{\epsilon_2} (T-\hat{t}_n)^{\xi_r(d/2-1)(\ell_r^{-1}-\ell_q^{-1})}\right)
\end{align*}
assuming that $(T-\hat{t}_n)^{1+(\frac{d}{2}-1)/\ell_q} \|u_0\|^{\frac{\beta d}{2\beta-1}}\lesssim 1$. Then, similar to the estimates for $\|u(s)\|$ within $[t_0, \tilde{t}_1(0)]$,
\begin{align*}
\sup_{t_0 <s< \hat{t}_{n^*}(0)} \|u(s)\| &\lesssim \|u_0\|_2 \sup_{t_0 <s< \hat{t}_{n^*}(0)} \left(c(\ell)^{\frac{\ell}{\ell+1}} (\ell!)^{\frac{1}{\ell+1}} \cR_{0,0}\left(\ell, c(\ell), s\right)\right)^{\frac{(d/2)(\ell+1)}{\ell+d/2}}
\\
&\lesssim \|u_0\|_2 \left(c(\ell) \cdot \Theta[\beta](\hat{p}_0,0) \cdot \left(1+\Theta[\beta](\hat{p}_0,0)^{-1}\right)^{C_1^*[\beta] \cdot n^*/ \bar\theta_\nu}  \hat{\cR}_{\hat{p}_0}(t_0)\right)^{\frac{(d/2)(\ell+1)}{\ell+d/2}} \ ,
\end{align*}
thus
\vspace{-0.1in}
\begin{align*}
\sup_{t_0 <s< \hat{t}_{n^*}(0)} \|u(s)\| \lesssim_{\|u_0\|_2} (T-\hat{t}_n)^{-\frac{1}{2}} \exp\left((T-\hat{t}_n)^{\xi_0(d/2)(d/2-1)(\ell^{-1}-\ell_q^{-1})}\right) \approx: \cK(t_0)\cdot (T-t_0)^{-\frac{1}{2}} \ .
\end{align*}
As $\cK(t_0)\approx 1$, this justifies the condition $\sup_{t_0 <s< T} \|u(s)\| \lesssim (T-\hat{t}_{n^*})^{-1/2}$ so Theorem~\ref{cor:LwBddTheta} and Lemma~\ref{le:AstrBddSp} are applicable and the process described above may continue until $T$.

If $(T-\hat{t}_n)^{1+(\frac{d}{2}-1)/\ell_q} \|u_0\|^{\frac{\beta d}{2\beta-1}}\gg 1$, we separate $[t_0, T]$ at some $\cT_1<T$ such that the conditions for Theorem~\ref{cor:LwBddTheta} and Lemma~\ref{le:AstrBddSp} are satisfied within $[t_0, \cT_1]$ so the regularity of the solution remains until $\cT_1$.
Then we separate $[\cT_1 , T]$ at some $\cT_2$ such that the conditions hold within $[\cT_1, \cT_2]$ and similarly one can verify that
\begin{align*}
\sup_{\cT_1 <s< \cT_2} \|u(s)\| \lesssim (T-\cT_2)^{-1/2}
\end{align*}
which justifies the condition for Theorem~\ref{cor:LwBddTheta} and Lemma~\ref{le:AstrBddSp} that are applied to the string $[\hat{\ell}_1, \hat{\ell}_2]$ so regularity remains until $\cT_2$. Inductively, we divide $[\cT_i , T]$ at some $\cT_{i+1}$ such that the conditions hold within $[\cT_i , \cT_{i+1}]$ so that
\begin{align*}
\sup_{\cT_i <s< \cT_{i+1}} \|u(s)\| \lesssim (T-\cT_{i+1})^{-1/2}
\end{align*}
and Theorem~\ref{cor:LwBddTheta} and Lemma~\ref{le:AstrBddSp} are applicable to the string $[\hat{\ell}_i, \hat{\ell}_{i+1}]$ until $\cT_{i+1}$. This dividing process stops at some index $\cT_{i_*}$ and regularity remains until $T$ with $\|u(T)\|\lesssim (T-\cT_{i_*+1})^{-1/2}$. In particular, $T$ is not a blow-up time.
If such process repeats for infinitely many times then $\|u(t)\|\lesssim (T-t)^{-1/2}$, which contradicts with the lower bound of the blow-up rate of Leray solution.

The proof is basically the same if switching between Type-$\cA^\beta$ and non-Type-$\cA^\beta$ only occurs within some $[\hat{\ell}_i, \hat{\ell}_{i+1}]$ with $i\le i_*$ for some fixed index $i_*$.
If Type-$\cA^\beta$ strings $[\hat{\ell}_i, \hat{\ell}_{i+1}]^*$ gets higher and higher indexes $i$ when $t\to T$, then the largest time span between the the two states (Type-$\cA^\beta$ and non-Type-$\cA^\beta$)
\begin{align*}
\hat{t}_{n+1}(i) - \hat t_n(i) &\approx T_{k_{i+q}}^* := \left(\phi+\psi\right)^{-\frac{2\beta}{2\beta-1}} \|(\nabla_\nu)^{k_{i+q}}u(\hat t_n(i))\|^{-\frac{2\beta}{(2\beta-1)(k_{i+q}+1)}}
\\
&\lesssim \left(\phi(\hat\ell_i,\hat\ell_{i+1})+\psi(\hat\ell_i,\hat\ell_{i+1})\right)^{-\frac{2\beta}{2\beta-1}} \left(T-\hat t_n(i)\right)^{\frac{1}{2\beta-1}}
\end{align*}
Recall that, with Corollary~\ref{cor:tLwBddHomInd}, $[\hat{\ell}_i, \hat{\ell}_{i+1}]$ is of Type-$\cA^\beta$ at $\hat t_n(i)$ implies
\begin{align*}
\phi(\hat\ell_i,\hat\ell_{i+1}),\ \psi(\hat\ell_i,\hat\ell_{i+1})\approx \|(\nabla_\nu)^{k_{i+q}}u(\hat t_n(i))\|^{\frac{2(\beta-1)}{k_{i+q}+1}} \lesssim \left(T-\hat t_n(i)\right)^{-\frac{\beta-1}{\beta}}
\end{align*}
therefore $T-\hat t_n(i) \gtrsim \hat{t}_{n+1}(i) - \hat t_n(i) := \hat{T}_n(i)$. 
If $T-\hat t_n(i) \gg \hat{T}_n(i)$, then the above argument implies the blow-up rate of some higher order terms $\|(\nabla_\nu)^ku(t)\|^{\frac{1}{k+1}}$ within the Type-$\cA^\beta$ strings is less than the natural rate $(T-t)^{-\frac{1}{2\beta}}$ and the lower order terms are restricted due to Lemma~\ref{le:Rk-transitivity} and the natural upper bound of the homogeneity index $\bar\theta_\nu$ in Corollary~\ref{cor:tLwBddHomInd}.
If $T-\hat t_n(i) \approx \hat{T}_n(i)$, then either $[\hat{\ell}_i, \hat{\ell}_{i+1}]$ keeps switching between the two states as $\hat t_n(i)\to T$ and the previous argument for $[\hat{\ell}_0, \hat{\ell}_1]$ applied to $[\hat{\ell}_i, \hat{\ell}_{i+1}]$, or Type-$\cA^\beta$ strings only exist for some higher indexes after $\hat t_n(i)$ and we apply the same argument as above to those Type-$\cA^\beta$ strings.
In other words, due to the extra $\beta-1$ diffusion, either the Type-$\cA^\beta$ strings with larger indexes cause slower increment of $\|(\nabla_\nu)^ku(t)\|^{\frac{1}{k+1}}$ than the natural rate $(T-t)^{-\frac{1}{2\beta}}$ and the lower order terms are restricted by the higher order terms within the Type-$\cA^\beta$ strings due to Lemma~\ref{le:Rk-transitivity} and Corollary~\ref{cor:tLwBddHomInd}, or the time period for each switch $\hat{T}_n(i)$ is always comparable to $T-\hat t_n(i)$ for all Type-$\cA^\beta$ strings as $\hat t_n(i)\to T$, thus the previous argument for $[\hat{\ell}_0, \hat{\ell}_1]$ implies either the solution $u$ on some $[\cT_i , \cT_{i+1}]$ is always restricted by $(T-\cT_{i+1})^{-1/2}$ and the process stops at some $\cT_{i_*}$ with $\|u(T)\|\lesssim (T-\cT_{i_*})^{-1/2}$, or the process repeats for infinitely many times and $u$ tends to infinity but $\|u(t)\|\lesssim (T-t)^{-1/2}$ which contradicts the lower bound of the blow-up rate of Leray solution.

\end{proof}

\begin{remark}
Theorem~\ref{th:HypNSEReg} and Theorem~\ref{cor:LwBddTheta} together manifest that non-homogeneity (in the sense of \eqref{eq:NonHomCond}) of $\{\cR_{m,n}^{\bR^d,\nu}(k,c_k,t)\}_{k=i_*}^\infty$ as $t\to T$ implies regularity up to $T$.
\end{remark}

\bigskip

\section{Ruling out rescaled blow-up profiles}\label{sec:hh}

Recall that the family of rescaled blow-up profiles of interest is given by

\begin{equation*}\label{ha}
 u(x, t)=\frac{1}{(-t)^{\alpha_t}} \, U(y, s) \ \ \ \mbox{where} \ \ \ y=\frac{x}{(-t)^{\alpha_x}}, \ \ \ 
 s=-\log(-t)
\end{equation*}
and $U$ is a smooth (in $y$) base profile decaying outside $B(0, 1)$ such that

\begin{align*}
  \|D^{(k)} U(s)\|_{L^\infty} &\le C_k\\
  \|D^{(k)} U(s)\|_{L^\infty(B(0, 1))} &\ge c_k
\end{align*}
uniformly in large $s$ where
\[
\frac{(C_k)^\frac{1}{k+1}}{(c_{2k})^\frac{1}{2k+1}} \le c.
\]

\medskip

Theorem~\ref{th:yay} is a simple consequence of Theorem~\ref{th:HypNSEReg}, the proof is
outlined below. Henceforth, the $L^\infty$-norms are the $L^\infty$-norms over the core region in the $x$-space,
$B(0, R(t))$ where $R(t) \approx {(-t)^{\alpha_x}}$, corresponding to the unit ball in the $y$-space.

\medskip

A straightforward calculation yields 

\begin{equation}\label{eq:ratio}
 \theta(k, 2k, t) \approx \frac{ \|D^{(k)} u(t)\|_\infty^\frac{1}{k+1} }{  \|D^{(2k)} u(t)\|_\infty^\frac{1}{2k+1}  }    \approx {(-t)}^\frac{k (\alpha_x-\alpha_t)}{(k+1)(2k+1)}
\end{equation}
(here, the derivatives are full derivatives; since we are making no assumption on the geometry of the profiles $U$, working with the directional derivatives 
would not yield a gain).

\medskip

In order to satisfy the condition (\ref{eq:NonHomCond}), it suffices that (\ref{eq:ratio}) is bounded by the first term in the maximum whose leading order
term (in $k$) is 
\[
 {(-t)}^{-\frac{\beta-1}{2k+1}}.
\]
Ignoring the quadratic perturbation for a moment, the flow will satisfy the condition as long as
\[
 \alpha_t \le \alpha_x + \frac{k+1}{k} (\beta-1),
\]
and in particular, as long as
\[
 \alpha_t \le \alpha_x + (\beta-1).
\]
Finally, in order to absorb the perturbation, it is enough to require
\begin{equation}\label{eq:line}
 \alpha_t < \alpha_x + (\beta-1),
\end{equation}
producing the dividing line in Figure 1. This yields  Theorem~\ref{th:yay}, (i).

\medskip

For  Theorem~\ref{th:yay}, (ii), it suffices to notice that the dividing line will sweep the potentially singular region 
as soon as $\beta > \frac{1+\sqrt{2}}{2}$.

\medskip

This completes the proof of  Theorem~\ref{th:yay}.

\bigskip

\bigskip

\bigskip

\centerline{\textbf{Acknowledgments}}

\bigskip

The work of Z.G. is supported in part by the National Science Foundation grant DMS--2009607,
``Toward criticality of the Navier-Stokes regularity problem''. 

\medskip

We thank an anonymous referee for their insightful and thoughtful comments.

\bigskip

\bigskip

\bigskip

\noindent \textbf{COI Statement:} On behalf of all authors, the corresponding author states that there is no conflict of interest.

\bigskip

\bigskip

\noindent \textbf{Data Availability Statement:} This manuscript has no associated data.

\bigskip

\bigskip

\bibliographystyle{abbrvnat}
\bibliographystyle{plainnat}

\def\cprime{$'$}

\end{document}